\documentclass[11pt]{amsbook}

\usepackage{xypic}
\usepackage{multicol}
\usepackage{rotating}
\usepackage{amssymb}

\newcommand{\LB}[1]{\textbf{\fontsize{10pt}{12pt}\selectfont #1}}
\newcommand{\LBmorespace}[1]{\textbf{\fontsize{10pt}{12pt}\selectfont #1}\par\vspace*{\baselineskip}}
\newcommand{\enlarge}[1]{\scalebox{.9}{${#1}$}}

\makeatletter
\renewenvironment{theindex}{\newpage
\addcontentsline{toc}{chapter}{List of Symbols}%
\pagestyle{plain}\let\item\@idxitem
\setlength{\columnsep}{30pt}
\begin{multicols}{3}[\centering{\Large\bfseries List of Symbols}]
\par\bigskip}{\end{multicols}}
\makeatother

\newcommand{\noqed}{\renewcommand{\qedsymbol}{{}}\vspace*{-18pt}}
\DeclareMathOperator{\nheadarrow}{{\!\!\!\xymatrix@C=12pt{\ar@{->}[r]&{}}}\hspace*{-4pt}}
\let\to=\nheadarrow
\let\rightarrow=\nheadarrow

\newcommand{\cW}{\mathcal{W}}
\newcommand{\cC}{\mathcal{C}}
\newcommand{\dotcC}{{\mathcal{C}}}

\newtheorem*{proposition1.2}{Proposition 1.2}
\newtheorem*{proposition1.4}{Proposition 1.4}
\newtheorem*{proposition2.2}{Proposition 2.2}
\newtheorem*{proposition2.6}{Proposition 2.6}
\newtheorem*{theoremA}{Theorem A}

\newtheorem*{theoremB}{Theorem B}

\newtheorem{Theorem}{Theorem}[chapter]

\newtheorem{ExtensionTheorem}[Theorem]{Extension Theorem}
\newtheorem{CompatibilityTheorem}[Theorem]{Compatibility Theorem}
\newtheorem{InductiveHypothesis}[Theorem]{Inductive Hypothesis}

\newtheorem{HRL}[Theorem]{Homotopy Rotation Lemma}
\newtheorem*{corollaryC}{Corollary C}
\newtheorem*{corollaryD}{Corollary D}
\newtheorem*{theorem3.7}{Theorem 3.7}
\newtheorem*{theorem3.9}{Theorem 3.9}
\newtheorem*{theorem5.4}{Theorem 5.4}
\newtheorem*{theorem5.7}{Theorem 5.7}
\newtheorem*{corollary6.4}{Corollary 6.4}
\newtheorem*{theorem*}{Theorem}

\newtheorem*{MainTheorem*}{Main Theorem}
\newtheorem{proposition}[Theorem]{Proposition}
\newtheorem{corollary}[Theorem]{Corollary}
\newtheorem{lemma}[Theorem]{Lemma}

\theoremstyle{definition}
\newtheorem*{definition5.6}{Definition 5.6}
\newtheorem{definition}[Theorem]{Definition}

\newtheorem*{remark*}{Remark}

\newtheorem*{conjectureE}{Conjecture E}
\theoremstyle{remark}
\newtheorem*{note*}{Note}

\numberwithin{section}{chapter}
\numberwithin{equation}{chapter}

\makeindex

\begin{document}

\frontmatter

\title{Abelian Properties of Anick Spaces}

\author{Brayton Gray}
\address{Department of Mathematics, Statistics and Computer Science,
         University of Illinois at Chicago,
         851 S.~Morgan Street,
         Chicago, IL, 60607-7045, USA} 
\email{brayton@uic.edu}






\setcounter{tocdepth}{3} 
\tableofcontents

\begin{abstract}
Anick spaces are closely connected with
both EHP sequences and the study of torsion
exponents. In addition they refine the secondary
suspension and enter unstable periodicity. In
this work we describe their
$H$-space properties as well as universal properties.
Techniques include a new kind on Whitehead product defined for maps out of co-H spaces, calculations in an additive category that lies between the unstable category and the stable category, and a controlled version of the extension theorem of \cite{GT10}.
\end{abstract}

\maketitle

\mainmatter

\chapter{Introduction}\label{chap1}

\section{Statement of Results}\label{subsec1.1}

By an Anick space we mean a homotopy CW
complex $T_{2n-1}$\index{$T_{2n-1}$|LB} which occurs in a fibration sequence.
\begin{equation}\label{eq1.1a}
\xymatrix{
\Omega^2S^{2n+1}\ar@{->}[r]^{\pi_n}
&S^{2n-1}\ar@{->}[r] 
&T_{2n-1}\ar@{->}[r]
&\Omega S^{2n+1}
}
\end{equation}
where the composition
\begin{equation*}
\xymatrix{
\Omega^2S^{2n+1}\ar@{->}[r]^{\pi_n}
&S^{2n-1}\ar@{->}[r]^{E^2} 
&\Omega^2S^{2n+1}
}
\end{equation*}
is homotopic to the $p^r$th power map on $\Omega^2S^{2n+1}$.

We will say that a space is a homotopy-Abelian
 $H$-space
if it has a homotopy associative and homotopy
commutative $H$-space structure. 

Throughout this work we will assume that all spaces
are localized at $p\geqslant 5$ unless otherwise indicated.
\begin{theoremA}
There is a homotopy-Abelian Anick space for any
$n\geqslant 1$ and $r\geqslant 1$.
\end{theoremA}

We will write $P^{2n}(p^r)$\index{$P^{2n}(p^r)$|LB} for the $2n$ dimensional Moore space
$S^{2n-1}\cup_{p^r} e^{2n}$.
From (\ref{eq1.1a}) we see that $T_{2n-1}$ is $2n-2$ connected
and there is a $2n$-equivalence
\[
i\colon P^{2n}(p^r)\to T_{2n-1}
\]

For any homotopy-Abelian $H$-space $Z$, let $[T_{2n-1},Z]_H$\index{[\ ,\ ]@$[\ ,\ ]_H$|LB}
be the Abelian group of homotopy classes of based $H$-maps
from $T_{2n-1}$ to~$Z$. Let
\[
p_k(Z)\index{$p_k(Z)$|LB}=p^{r+k-1}\pi_{2np^k-1}(Z;Z/p^{r+k}).
\]
\begin{theoremB}
$[T_{2n-1},Z]_H\cong \underset{\leftarrow}{\lim} G_k(Z)$\index{$G_k(Z)$|LB} where $G_0(Z)=[P^{2n},Z]
=\pi_{2n}(Z;Z/p^r)$ and there are exact sequences:
\[
\xymatrix{
0\ar@{->}[r]&p_k(\Omega
Z)\ar@{->}[r]^{\enlarge{e}}&G_k(Z)\ar@{->}[r]^{\enlarge{r}}&G_{k-1}(Z)\ar@{->}[r]^{\enlarge{\beta}}&p_k(Z).
}
\]
In particular, if $p^r\pi_*(Z)=0$, there is an isomorphism
\[
[P^{2n}(p^r),Z]\approx[T_{2n-1}, Z]_H
\]
given by the restriction
\[
P^{2n}(p^r)\to T_{2n-1}.
\]
\end{theoremB}

Several examples are given in Chapter~\ref{chap7}. In particular
\begin{corollaryC}
Given two homotopy-Abelian Anick spaces for the same values of $n$, $r$ and $p>3$, there is an $H$ map between them which is a homotopy equivalence.
\end{corollaryC}
\begin{corollaryD}
In any homotopy-Abelian $H$-space structure on an
Anick space, the identity map has order $p^r$.
\end{corollaryD}

Let $T_{2n}\index{$T_{2n}$|LB}=S^{2n+1}\{p^r\}$. Then analogous results\footnote{In this case no torsion requirements are needed for
Theorem~B.} to theorems~A and ~B are well known
(\cite{Nei83}, \cite{Gra93a}) when $p>3$, and in particular, any
map 
\[
\alpha\colon P^{n+1}(p^r)\to P^{m+1}(p^r)
\]
corresponds to a
unique $H$-map $\widehat{\alpha}$ such that the diagram
\[
\xymatrix{
P^{n+1}(p^r)\ar@{->}[r]^{\alpha}\ar@{->}[d]&P^{m+1}(p^r)\ar@{->}[d]\\
T_n\ar@{->}[r]^{\widehat{\alpha}}&T_m
}
\]
commutes up to homotopy, for any $n$ and $m$.
This corollary was the object of the conjectures in~\cite{Gra93a}.

In developing these results, several new techniques
of geometric homotopy theory are introduced. These may
be of some use in other problems. A summary of some of
these techniques can be found in section~\ref{subsec1.4}.

In an appendix, we show that these results do not
generally hold if $p=3$, and we treat the
special case when $n=1$.

The author would like to thank
Joseph Neisendorfer for many
helpful conversations during
this work.

\section{History}\label{subsec1.2}

A map $\pi_n\colon \Omega^2S^{2n+1}\to S^{2n-1}$ with the property that
the composition
\[
\xymatrix{
\Omega^2S^{2n+1}\ar@{->}[r]^{\pi_n}&S^{2n-1}\ar@{->}[r]^{E^2}&\Omega^2S^{2n+1}
}
\]
is homotopic to the $p^r$th power map  was first
discovered when $p\geqslant 3$ by Cohen, Moore and
Neisendorfer (\cite{CMN79b}) and played a
crucial role in determining the maximal
exponent for the torsion in the homotopy groups
of spheres.

In \cite{CMN79c}, the authors raised the question
of whether a fibration such as \ref{eq1.1a} could exist.
The feasibility of constructing a
secondary EHP sequence refining the
secondary suspension (\cite{Ma75}, \cite{Co83})
together with a theory of compositions (\cite{Tod56})
was studied in \cite{Gra93a}, \cite{Gra93b}. This led the author
to conjecture the existence of Anick spaces
with theorems~A and~B
and corollaries~C and~D.

At about the same time, David Anick was studying the
decomposition of the loop space on a finite complex (\cite{Ani92}).
His intention was to find a list of indecomposable spaces
which, away from a few small primes, could be used for
decomposition. This led to the construction of a sequence of
spaces for $p\geqslant 5$. The limit of this sequence is the space
sought after in \cite{Gra93a}. This work of Anick was
published in a 370-page book (\cite{Ani93}). In \cite{AG95},
the authors showed that the Anick space so
constructed admitted an $H$-space structure when $p\geqslant 5$.
They also proved a weaker version of Theorem~B.
They showed that if $p_k(Z)=0$ for all $k$, an extension to an Anick space existed,
but there was no indication
that the extension would be an $H$-map or that it
would be unique. At that time it was thought that
the torsion condition was a peculiarity
of the approach and,  it was
conjectured that, as in the case of $T_{2n}=S^{2n+1}\{p^r\}$,
this requirement was
unnecessary.

In \cite{T01}, the author asserted theorem A and corollary C
and that for each homotopy Abelian $H$-space $Z$,
there is an equivalence $[P^{2n}(p^r),Z]\simeq[T_{2n-1},Z]_H$.
This, however, is not consistent with the results in~\cite{Gra12},
where counter examples are provided with $Z$ is as Eilenberg--MacLane space. The assertions in~\cite{T01} depend on the author's Theorem~2.1 which is quoted as ``to appear in Topology'' in the author's
bibliography. This, however, did not appear. Many of
the results in the author's section~5 are inconsistent
with results we obtain here. The main
result of the author's section~4 is valid and we give a
much simplified proof of it here~(2.9).

In \cite{GT10} a much simpler construction of
the Anick spaces was obtained which worked for all $p\geqslant 3$. This
result replicated the results of \cite{Ani93} and \cite{AG95}
and extended them to the case $p=3$. Furthermore they showed that the homotopy type of an Anick
space that supports an $H$-space structure is uniquely
characterized by $n$, $p$ and $r$.

In \cite{Gra12}, a proof was given 
that if $p_k(\Omega Z)=0$
for all $k$, there
is at most one
extension of a
map $\alpha\colon P^{2n}\to Z$
to an $H$-map
$\widehat{\alpha}\colon T\to Z$ and examples were presented
to show that this torsion condition is necessary. The proof
we give here is entirely different.

In \cite{GT10}, the authors constructed the EHP
sequences conjectured in \cite{Gra93a}.
\begin{align*}
\xymatrix{
&T_{2n-1}\ar@{->}[r]^{E}&\Omega T_{2n}\ar@{->}[r]^H &BW_n\\
&T_{2n}\ar@{->}[r]^-<<<<{E}&\Omega T_{2n+1}\ar@{->}[r]^{H}&BW_{n+1}
}
\end{align*}
(where $T_{2n}=S^{2n+1}\{p^r\}$), and $BW_n$\index{$BW_n$|LB} lies in a fibration sequence:
\begin{equation*}
\xymatrix{
S^{2n-1}\ar@{->}[r]^{E^2}&\Omega^2S^{2n+1}\ar@{->}[r]^{\nu}&BW_n. 
}
\end{equation*}\index{nu@$\nu$|LB}
More details appear in the next section.
The only remaining unsettled conjecture
in \cite{Gra93a} for the Anick spaces is the following:
\begin{conjectureE}
There is a homotopy equivalence
\[
BW_n\simeq \Omega T_{2np-1}(p)
\]
where $T_{2np-1}(p)$ is the Anick space with $r=1$.
\end{conjectureE}

In \cite{The11}, Theriault constructed $T_{2n-1}(2^r)$ for
$r\geqslant 3$, but there is no $H$-space structure in this case. 

\section{Methods and Modifications}\label{subsec1.3}

Throughout this work we will fix $n$ and abbreviate $T_{2n-1}$
as~$T$\index{$T$|LB} if
this will not lead to confusion. The construction in~\cite{GT10} begins with a fibration
sequence:
\begin{equation}\label{eq1.2new}
\xymatrix{
W_n\ar@{->}[r]&S^{2n-1}\ar@{->}[r]^{\enlarge{E^2}}& \Omega^2S^{2n+1}\ar@{->}[r]^{\enlarge{\nu}}&BW_n
}
\end{equation}
where $E^2$ is the double suspension map (\cite{Gra88}). The
authors construct a factorization
of the map~$\nu$:
\[
\xymatrix{
\Omega^2S^{2n+1}\ar@{->}[r]^->>>>{\enlarge{\Omega\partial}}& \Omega
S^{2n+1}\{p^r\}\ar@{->}[r]^-<<<{\enlarge{H}}&BW_n
}
\]
where $\partial$ occurs in the fibration sequence defining $S^{2n+1}\{p^r\}$,
the fiber of the degree $p^r$ map on $S^{2n+1}$. For any choice of~$H$,
there is a homotopy commutative diagram:
\begin{equation}\label{eq1.3new}
\begin{split}
\xymatrix{\Omega^2S^{2n+1}\ar@{=}[r]\ar@{->}[d]_{\pi_n}&\Omega^2S^{2n+1}\ar@{->}[d]_{p^r}\\
S^{2n+1}\ar@{->}[r]\ar@{->}[d]&\Omega^2S^{2n+1}\ar@{->}[r]^{\nu}\ar@{->}[d]_{\Omega \partial}&BW_n\ar@{=}[d]\\
T\ar@{->}[r]^->>>>>>E\ar@{->}[d]&\Omega S^{2n+1}\{p^r\}\ar@{->}[r]^->>>H\ar@{->}[d]&BW_n\\
\Omega S^{2n+1}\ar@{=}[r]&\Omega S^{2n+1}
}
\end{split}
\end{equation}
and consequently for each choice of~$H$, the fiber of~$H$ is an
Anick space by~(\ref{eq1.1a}).

In \cite{GT10}, the authors proceed to show that for
any such choice, the Anick space admits an $H$-space
structure such that the fibration~\ref{eq1.1a} is an $H$-fibration.
Furthermore, they prove
\begin{proposition1.4}[{\cite[4.9]{GT10}}]
Fix $n$, $r$ and $p\geqslant 3$. Then any two Anick spaces which admit
an $H$-space structure are homotopy equivalent.
\end{proposition1.4}

The $H$-space structure is constructed in~\cite{GT10}
as follows. The authors construct a
splitting:\addtocounter{equation}{1}
\begin{equation}\label{eq1.5new}
\Sigma T\simeq G\vee W
\end{equation}
where $G$ is atomic and $W$ is a $4n-1$ connected wedge of
Moore spaces. A map $\varphi\index{varphi@$\varphi$|LB}\colon G\to S^{2n+1}\{p^r\}$ is constructed
from (\ref{eq1.5new}) and the adjoint of the map~$E$ in (\ref{eq1.3new})
\[
\xymatrix{
G\ar@{->}[r]&\Sigma T\ar@{->}[r]^->>>>{\enlarge{\widetilde{E}}}&S^{2n+1}\{p^r\}
}
\]
and a space $E$ is constructed as the fiber of~$\varphi$:
\[
\xymatrix{
\Omega S^{2n+1}\{p^r\}\ar@{->}[r]^-<<<<{\partial}&
E\ar@{->}[r]&G\ar@{->}[r]^->>>>{\enlarge{\varphi}}&S^{2n+1}\{p^r\}.
}
\]\index{$G$|LB}

The $H$-space structure arises from a factorization
 of~$H$ through the space~$E$:
\begin{equation}\label{eq1.6new}
\begin{split}
\xymatrix{
\Omega G\ar@{=}[r]\ar@{->}[d]_{\enlarge{h}}&\Omega G\ar@{->}[d]\\
T\ar@{->}[r]^->>>>{\enlarge{E}}\ar@{->}[d]&\Omega
S^{2n+1}\{p^r\}\ar@{->}[r]^-<<<{\enlarge{H}}\ar@{->}[d]_{\enlarge{\partial'}}&BW_n\ar@{=}[d]\\
R
\ar@{->}[r]\ar@{->}[d]&E\ar@{->}[r]^{\enlarge{\nu_{\infty}}}\ar@{->}[d]_{\enlarge{\pi}}
&BW_n\\
G\ar@{=}[r]&G
}
\end{split}
\end{equation}\index{$R$|LB}
$h$ has a right homotopy inverse $g\colon T\to \Omega G$ defined by
the splitting (\ref{eq1.5new}) since $T$ is atomic. Thus $T$ inherits
an $H$ space structure as a retract of~$\Omega G$. The $H$-space structure
depends on $h$ and consequently on the choice of $\nu_{\infty}$. The
splitting (\ref{eq1.5new}) and map $\nu_{\infty}$ are defined inductively. The space $G$ is filtered
by subspaces $G_k$ where
\begin{equation}\label{eq1.7new}
G_k=G_{k-1}\cup_{\alpha_k}\index{alphak@$\alpha_k$|LB}CP^{2np^k}(p^{r+k})
\end{equation}
and the fibration $\xymatrix{E\ar@{->}[r]^{\enlarge{\pi}}&G}$ is the union of the induced
fibrations
\[
\Omega S^{2n+1}\{p^r\}\to E_k\to G_k.\index{$E$|LB}\index{$E_k$|LB}
\]
The splitting (\ref{eq1.5new}) is approximated by a sequence
of splittings:
\[
\Sigma T^{2np^k}\simeq G_k\vee W_k
\]
and $\nu_{\infty}$ is simultaneously constructed by induction
over the restriction of~(\ref{eq1.6new}) to~$G_k$.
The map $\nu_{\infty}$ is then the limit of maps
\[
\nu_k\index{nuk@$\nu_k$|LB}\colon E_k\to BW_n.
\]
The extension theorem \cite[21]{GT10} is applied which
guarantees that any map $\nu_{k-1}\colon E_{k-1}\to BW_n$ extends
to~$E_k$. An arbitrary choice is made for each $k>0$.
It seems likely that the number of choices for $\nu_{\infty}$\index{nuinfty@$\nu_{\infty}$|LB} is
uncountable.

\section{Outline of Modifications}\label{subsec1.4}

The basis of this paper is to modify and sharpen the
construction in \cite{GT10} as described in section~\ref{subsec1.3}. As
explained there, the $H$-space properties that the Anick
space inherits depend on a choice of a map
\[
\nu_k\colon\allowbreak E_k\to\nobreak BW_n.
\]

In Chapter~\ref{chap2}, we introduce maps
\[
\Gamma_k\index{Gammak@$\Gamma_{k}$|LB}\colon \Omega G_k * \Omega G_k\to E_k
\]
and prove that if we choose $\nu_k$ such that $\nu_k\Gamma_k$ is null homotopic for each $k$, the induced $H$-space structure will be homotopy-Abelian. We also recall, at this point, various facts about the Anick spaces that were developed in \cite{GT10} which will be needed in the sequel.

In Chapter~\ref{chap3}, we recall from \cite{Gra11} the construction
of a Whitehead product pairing
\[
[G,X]\times[H,X]\to [G\circ H,X]
\]
where $G$ and $H$ are simply connected co-$H$ spaces and
$G\circ H$\index{$G\circ H$|LB} is a new simply connected co-$H$ space. This
generalizes the classical pairing:
\[
[\Sigma A,X]\times [\Sigma B,X]\to [\Sigma A\wedge B,X].
\]
We also generalize Neisendorfer's theory of relative
Whitehead products and $H$-space based Whitehead
products in the $\bmod\, p^r$ homotopy of a principal
fibration \cite{Nei10a} by replacing Moore spaces with arbitrary co-$H$
spaces. We then
reduce the question of whether $\nu_k\Gamma_k$ is null
homotopic to whether a sequence of iterated $H$-space based
and relative Whitehead
products $G_k\circ(G_k\circ\dots G_k)\to E_k$ are annihilated by $\nu_k$
for all $i\geqslant 2$.

In Chapter~\ref{chap4} we construct $\bmod\, p^{r+i-1}$ homotopy classes
$a(i)$ and $c(i)$ in $E_k$ for $i\leqslant k$ and a $\bmod\, p^{r+k}$ homotopy
class $\beta_k$\index{betak@$\beta_k$|LB} which will play a key role. We also introduce
``index $p$ approximation'' and show that the iterated
Whitehead product under investigation can be approximated
by iterated Whitehead products in homotopy groups with
coefficients in $Z/p^s$ for $r\leqslant s\leqslant r+k$. This approximation
excludes
the case $n=1$ which is handled in the appendix. These classes are the obstructions and
we seek to choose $\nu_k$ which annihilates them.

When $k>0$, the obstructions actually belong to two classes,
$A$ and~$C$ depending on whether they involve the $a(i)$ or the~$c(i)$. (It turns out that any obstruction involving
both $a(i)$ and $c(j)$ is automatically in the kernel of~$\nu_k$.)
In \ref{subsec5.1} we simplify the procedure by defining a
quotient space $J_k$ of~$E_k$ which is universal for
annihilating the classes as~$C$. $J_k$\index{$J_k$|LB} is a principal
fibration over a space $D_k$ which is a quotient\footnote{The space $D_k$,
defined differently, occurs in the original construction of Anick~(\cite{Ani93}).
The results of~\cite{AG95} 
are obtained by
replacing $D_k$ by~$G_k$.}
of~$G_k$
and we seek a factorization of~$\nu_k$:\index{gammak@$\gamma_{k}$|LB}
\[
\xymatrix{
E_k\ar@{->}[r]^{\enlarge{\tau_k}}
&J_k\ar@{->}[r]^->>>>{\enlarge{\gamma_k}}& BW_n
}
\index{tauk@$\tau_k$|LB}
\]
where $\gamma_k$
annihilates the obstructions in~$A$.
In section~\ref{subsec5.2} we introduce a congruence relation
among homotopy classes and the relative Whitehead
products and $H$-space based Whitehead products have
better properties in the congruence homotopy category.
This allows for a further reduction in 
obstructions to a collection of $\bmod\, p^r$ homotopy classes.

In Chapter~\ref{chap6} we introduce the controlled extension theorem~(\ref{theor6.1}).
This is a modification of the extension theorem in \cite[2.2]{GT10}
which allows maps defined on the total space of an induced
fibration of a principal fibration over a subspace to be extended over
the total space
under certain conditions. In the controlled
extension theorem, conditions are given for the extension 
to annihilate certain maps maps $u\colon P\to E$. This is
immediately applied to the case $k=0$ where we construct $\nu_o$ by
induction over the skeleta of a space~$F_0$.

A complexity arises because for each $k>0$, there are
level~$k$ obstructions in infinitely many dimensions.
When we modify $\nu_k$ to eliminate these obstructions,
we can't assume that it will be an extension of~$\nu_{k-1}$
and consequently the level $k-1$ obstructions may
reappear. A separate argument (\ref{lem6.40}) dispenses
with this issue.  
In section~\ref{subsec6.2}
we introduce the inductive hypothesis (\ref{theor6.7}) and
a space $F_k$ is analyzed to prepare for the inductive 
step. This is accomplished in section~\ref{subsec6.3}.

In Chapter~\ref{chap7}, we discuss the universal properties
of the Anick spaces. From the fibration sequence (\ref{eq1.6new})
we extract the following fibration sequence
\[
\xymatrix{
\ar@{->}[r]^-<<<{*}&\Omega R\ar@{->}[r]&\Omega
G\ar@{->}[r]^{h}&T\ar@{->}[r]^{*}&\mbox{}
}
\]\index{$h$|LB}
which we think of as a presentation of~$T$. 

The proof of Theorem~B depends on an understanding
of the map $R\to G$. From~\cite[4.8]{GT10} we
know that for any choice of~$\nu_{\infty}$, $R$ is a wedge of Moore spaces.
Certain of these Moore spaces are needed to resolve the
relationship between $H_*(\Omega G;Z/p)$, which has infinitely
many generators and $H_*(T;Z/p)$ which has two
generators. These are the classes $a(i)$ and~$c(i)$. The others
are necessary to enforce the homotopy commutativity 
in~$H_*(T;Z/p^r)$. These are either Whitehead products
or generalized Whitehead products\footnote{In particular, the first element in $\pi_*(G)$ of order
$p^{r+k}$ could not be a classical Whitehead product for~$k>0$. It is
defined as a composition 
\[
\xymatrix{
P^{4np^k}(p^{r+k})\ar@{->}[r]&G_k\circ G_k\ar@{->}[r]^-<<<<{\enlarge{W}}&G_k
}
\]
where
$W$ is a generalized Whitehead product.} defined by co-$H$-spaces
in~\cite{Gra11}. The obstruction to extensions
depend on certain homotopy classes
\[
\widetilde{\beta_k}\colon P^{2np^k-1}(p^{r+k})\to \Omega G_{k-1}
\]
which must be annihilated in order for an extension to proceed.

In the appendix we discuss the case $n=1$ and the case $p=3$.  

\section{Conventions and Notation}\label{subsec1.5}

All spaces will be localized at a prime $p\geqslant 3$ and
usually we will assume $p\geqslant 5$. $H_*(X)$ and $H^*(X)$
will designate the $\bmod\, p$ homology and cohomology.
If other coefficients are used (usually $Z_{(p)}$) they will
be specified in the usual way.

We write $P^m(p^s)=S^{m-1}\cup_{p^s} e^m$ for the Moore space.
Throughout we will fix $r\geqslant 1$ and we will
always have $s\geqslant r$. We will
abbreviate $P^m(p^r)$ simply as $P^m$\index{$P^m$|LB}. We will write
$\iota_{m-1}$ and $\pi_m$ for the usual maps
\[
\xymatrix{
S^{m-1}\ar@{->}[r]^->>>>{\iota_{m-1}}&P^m(p^s)\ar@{->}[r]^{\pi_m}&S^m.
}
\]

We designate the symbols $\beta$\index{beta@$\beta$|LB}, $\sigma$\index{sigma@$\sigma$|LB},
$\rho$\index{rho@$\rho$|LB} for the maps
\begin{align*}
&\beta\colon P^m(p^s)\to P^{m+1}(p^s)\\
&\rho\colon P^m(p^s)\to P^{m}(p^{s+1})\\
&\sigma\colon P^m(p^s)\to P^{m}(p^{s-1})
\end{align*}
with $\beta=\iota_m\pi_m$, $\pi_m\rho=\pi_m$ and
$\sigma\iota_{m-1}=\iota_{m-1}$.
These symbols will not be indexed by the dimension
and can be composed, so that we have formulas
\begin{align*}
&\beta=\sigma\beta\rho\\
&p=\sigma\rho=\rho\sigma\\
&\beta\sigma^t=p^t\sigma^t\beta
\end{align*}
where $p$ is the degree $p$ self map. We write
\[
\delta_t\index{deltat@$\delta_t$|LB}=\beta\rho^t
\]
and will frequently use the cofibration sequence
\begin{multline}\label{eq1.9}
\xymatrix@C=37pt{
P^{m-1}(p^s)\vee P^m(p^s)\ar@{->}[r]^->>>>>>{-\delta_t\vee
\rho^t}&P^m(p^{s+t})}\\
\xymatrix@C=33pt{\ar@{->}[r]^-<<<<{p^s}&P^m(p^{s+t})\ar@{->}[r]^->>>>>{\sigma^t\vee\sigma^t\beta}&P^m(p^s)\vee
P^{m+1}(p^s)
}
\end{multline}
especially when $s=r+k-1$ and $t=1$.

We will write $\nu_p(m)$
for the largest exponent of $p$ that
divides $m$ and often set $s=\nu_p(m)$. For any map $x$ we
will write $\widetilde{x}$ for either its left or right adjoint, if there
is no possibility of confusion.

By a diagram of fibration sequences, we will mean
a diagram in which any sequences, either
vertical or horizontal are fibration sequences up to
homotopy.

\chapter{Abelian Structures}\label{chap2}

We begin by reviewing some material about
principal fibrations. In section~\ref{subsec2.1}, we recall the
construction of $BW_n$ (\cite{Gra88}) and the extension theorem
(\cite{GT10}), and for certain principal fibrations we
construct a natural map $\Gamma\colon \Omega B*\Omega B\to E$ in section~\ref{subsec2.2}. $\Gamma$~ is the
lynchpin for
generalizing Neisendorfer's $H$-space based Whitehead
products. In \ref{prop2.3} we give a short proof of a
result of Theriault giving a criterion for an $H$-space structure to be 
homotopy-Abelian. We use this to relate the map~$\Gamma$ to the obstructions
for the induced $H$-space
structure on the fiber being
homotopy Abelian. We conclude with Proposition~\ref{prop2.12}
 which presents the conditions we
will establish
in the
next 4 chapters. Finally, we recall some results
from \cite{GT10} that will be used in the sequel.

\section{Preliminaries}\label{subsec2.1}

In \cite{Gra88}, a clutching construction
was described for Hurewicz fibrations in case that the base
is a mapping cone. This construction is particularly simple
in the case of a principal fibration.

Suppose $\varphi\colon B\to X$. We describe a principal fibration
\[
\xymatrix{
\Omega X\ar@{->}[r]^{i}&E\ar@{->}[r]^{\pi}&B
}
\]
where $E=\{(b,\omega)\in B\times PX\mid  \omega(1)=\varphi(b)\}$, where $PX$
is the
space of paths $\omega\colon I\to X$ with $\omega(0)=\ast$.

In case $B=B_0\cup_{\theta}CA$, we have a pair of principal
fibrations:
\[
\xymatrix{
&\Omega X\ar@{->}[d]\ar@{=}[r]&\Omega X\ar@{->}[d]\\
&E_0\ar@{->}[r]\ar@{->}[d]_{\pi_0}&E\ar@{->}[d]_{\pi}\\
A\ar@{-->}[ur]\ar@{->}[r]^{\theta}&B_0\ar@{->}[r]&B
}
\]
Clearly $\theta$ lifts to a map $\theta'\colon A\to E_0$. We assert that
there is a lifting $\overline{\theta}\colon A\to E_0$ such
that the composition:
\[
\xymatrix{
A\ar@{->}[r]^{\enlarge{\overline{\theta}}}& E_0\ar@{->}[r]&E
}
\]
is null homotopic.
For if the composition
\[
\xymatrix{
A\ar@{->}[r]^{\theta'}&E_0\ar@{->}[r]&E
}
\]
is essential, it factors through $\Omega X$ up to homotopy, and we can
use the principal action
\[
\xymatrix{
\Omega X\times (E,E_0)\ar@{->}[r]^->>>>{a}&(E,E_0)
}
\]
to define a different lifting $\overline{\theta}\colon A\to E_0$ of
$\theta$ in which
the composition into $E$ is null homotopic. In
particular, the
composition
\[
\xymatrix{
(CA,A)\ar@{->}[r]^{\overline{\theta}}&(E,E_0)\ar@{->}[r]^{\pi}&(B,B_0)
}
\]
induces an isomorphism in homology.
\begin{proposition}[{\cite{Gra88}) (Clutching Construction}]\label{prop2.1}
Suppose 
\[
\xymatrix{
(E,E_0)\ar@{->}[r]^{\enlarge{\pi}}&(B,B_0)}
\]
is a Hurewicz
fibration with fiber $F$ where $B=B_0\cup_{\theta}CA$.
Then there is a map $\xymatrix{
F\times(CA,A)\ar@{->}[r]^-<<<{\enlarge{\varphi}}&(E,E_0),
}
$
and a pushout diagram
\[
\xymatrix{
F\times CA\ar@{->}[r]^-<<<{\varphi}&E\\
F\times A\ar@{->}[r]^{\varphi}\ar@{->}[u]&E_0\ar@{->}[u]
}
\]
where $\pi\varphi\colon F\times CA\to B$ is the projection onto $CA\subset
B$.
In particular\addtocounter{equation}{1}
\begin{equation}\label{eq2.2}
E/E_0\cong F\ltimes  \Sigma A.
\end{equation}
In case that $\pi$ is a principal fibration, we can take $\varphi$
to be the composition:
\[
\xymatrix@C=30pt{
\Omega X\times (CA,A)\ar@{->}[r]^{1\times\overline{\theta}}
&\Omega X\times (E,E_0)
\ar@{->}[r]^-<<<<<{a}
&(E,E_0)
}
\]
where $a$ \index{$a$|LB} is the principal action map.
\end{proposition}

The following result (\cite{Gra88}) is a simple application of~\ref{prop2.1}.\addtocounter{Theorem}{1}
\begin{proposition}\label{prop2.3}
Localized at a prime $p>2$, there is a fibration
sequence
\[
\xymatrix{
\Omega^2S^{2n+1}\ar@{->}[r]^->>>{\partial}&BW_n\times
S^{4n-1}\ar@{->}[r]&S^{2n}\ar@{->}[r]^->>>>>{E}&\Omega S^{2n+1}
}
\]
where $\partial$ factors 
$\xymatrix{\Omega^2S^{2n+1}\ar@{->}[r]^{\nu}&BW_n\ar@{->}[r]^->>>{i_1}&BW_n\times
S^{4n-1}}$
and the homotopy fiber of $\nu$ is $S^{2n-1}$
\[
\xymatrix{
S^{2n-1}\ar@{->}[r]^{E^2}&\Omega^2S^{2n+1}\ar@{->}[r]^{\nu}&BW_n.
}
\]
Furthermore if $p\geqslant 5$, $BW_n$ has
a homotopy-Abelian $H$-space
structure and $\nu$ is an $H$-map. 
\end{proposition}

Using the work of Cohen, Moore and Neisendorfer,
Theriault has proved
\begin{proposition}[{\cite{The08}}]\label{prop2.4}
$BW_n$ has $H$-space exponent\footnote{We say that an $H$-space has
$H$-space exponent $q$
if the $q^{\text{th}}$ power map is null homotopic in some association.}
$p$.
\end{proposition}
An important application of \ref{prop2.1} is the extension theorem
\begin{ExtensionTheorem}[{\cite[2.1]{GT10}}]\label{theor2.5}
Suppose $A$ is a co-$H$
space such that the map $\overline{\theta}\colon A\to E_0$ is divisible by $q$ in
the
co-$H$-space structure, and $Z$ is a connected $H$-space
with $H$-space exponent $q$. Then the restriction
\[
[E,Z]\to [E_0,Z]
\]
is onto.
\end{ExtensionTheorem}

This is a powerful tool and is the key to showing
that each map $\nu_{k-1}\colon\allowbreak E_{k-1}\to\nobreak BW_n$ extends to a map
$\nu_k\colon E_k\to BW_n$ in \cite{GT10}. In section~\ref{subsec6.1} we will
enhance
this to the controlled
extension theorem, which will allow for a good
choice of extension.

\section{Whitehead Products}\label{subsec2.2}

Whitehead and Samelson products in homotopy groups
with coefficients were introduced by Neisendorfer (\cite{Nei80}).
These included relative Whitehead products and in
\cite{Nei10a}, he introduced $H$-space based Whitehead
products. These will be useful in Chapter~\ref{chap3}
where we will generalize these constructions to classes
defined on co-$H$ spaces. At this point we will introduce
a basic construction and show how it is related to
the question of homotopy-Abelian $H$-space structures.

We begin with the construction of a map
\[
\omega\index{omega@$\omega$|LB}\colon \Omega U*\Omega V\to U\vee V.
\]
We use the decomposition $\Omega U*\Omega V=\Omega U\times C(\Omega V)\cup
C(\Omega U)\times \Omega V$
where $C$ is the cone functor with vertex at $0$ and define $\omega$ by 
\begin{align*}
&\xymatrix{
\Omega U\times C(\Omega V)\ar@{->}[r]^-<<<{\pi_2}&C(\Omega
V)\ar@{->}[r]^-<<<{\epsilon}&V
\ar@{->}[r]&U\vee V
}\\
&\xymatrix{
C(\Omega U)\times \Omega V\ar@{->}[r]^-<<<{\pi_1}&C(\Omega
U)\ar@{->}[r]^-<<<{\epsilon}&U
\ar@{->}[r]&U\vee V}
\end{align*}
where $\epsilon$ is an evaluation.
\begin{proposition}[{\cite{Gan70}}]\label{prop2.6}
$\omega$ lifts to a homotopy
equivalence with the homotopy fiber of the inclusion $U\vee V\to U\times V$.
\end{proposition}
\begin{proof}
The homotopy fiber is the union of the parts over $U$ and $V$; i.e.,
\[
F= \Omega U\times PV\cup PU\times \Omega V.
\]
For any space $W$, we define a homotopy equivalence of pairs
\[
\xi\colon (C(\Omega W),\Omega W)\to (PW,\Omega W)
\]
by $\xi(\omega,s)(t)=\omega(st)$. According to \cite{ST72}, the
pair $(PW,\Omega W)$ is an NDR pair, so
the induced map
\[
\widehat{\xi}\index{xih@$\widehat{\xi}$|LB}
\colon \Omega U\times C(\Omega V)\cup C(\Omega U)\times \Omega V\to
\Omega U\times PV\cup PU\times \Omega V
\]
is a homotopy equivalence.
\end{proof}

Now suppose that
\[
\xymatrix{
\Omega X\ar@{->}[r]^{i}&E\ar@{->}[r]^{\pi}&B
}
\]
is a principal fibration induced by a map $\varphi\colon B\to X$
where $X$ is an $H$-space. We can and will assume that
the multiplication on~$X$ has a strict unit (\cite[11.1.11]{Nei10a}).
\begin{proposition}\label{prop2.7}
There is a strictly commutative diagram,
\[
\xymatrix{
\Omega B *\Omega
B\ar@{->}[d]_{\omega}\ar@{->}[r]^-<<<{\Gamma}&E\ar@{->}[d]_{\pi}\\
B\vee B\ar@{->}[r]^{\nabla}&B
}
\]
where $\nabla$\index{\{\ ,\ \}$_{\times}\nabla$@$\nabla$|LB}
is the folding map. Furthermore, $\Gamma$\index{Gamma@$\Gamma$|LB} is natural
with respect to the data.
\end{proposition}
\begin{proof}
The square
\[
\xymatrix@C=30pt{
B\vee B\ar@{->}[rr]^{\nabla}\ar@{->}[d]&&B\ar@{->}[d]_{\varphi}\\
B\times B\ar@{->}[r]^{\varphi\times \varphi}&X\times X\ar@{->}[r]^{\mu}&X
}
\]
is strictly commutative and induces a map of the induced
principal fibrations. $\Gamma$ is the composition of
$\widehat{\xi}$
with the map
\[
\Omega B\times PB\cup PB\times \Omega B\to E=\{(b,\omega)\in B\times PX\mid
\omega(1)=\varphi(b)\}
\]
given by the formula
\[
(\omega_1,\omega_2)\to\begin{cases}
(\omega_1(1),\overline{\omega}(t))&\text{if}\ \omega_2(1)=\ast\\
(\omega_2(1),\overline{\omega}(t))&\text{if}\ \omega_1(1)=\ast\\
\end{cases}
\]
where
$\overline{\omega}(t)=\mu(\varphi(\omega_1(t)),\varphi(\omega_2(t))$.
\end{proof}

Recall (for example \cite[3.4]{Gra11}) that there is a
natural homotopy equivalence
\[
\Sigma(X\wedge Y)\simeq X *Y
\]
such that the diagram
\[
\xymatrix@C=0pt{
\Sigma(X\wedge Y)\ar@{->}[d]_{-\tau}&{}\simeq{}&X*Y\ar@{->}[d]_{\tau}\\
\Sigma(Y\wedge X)&{}\simeq{}&Y * X
}
\]
commutes up to homotopy, where the maps labeled $\tau$ are the
transposition maps.
\begin{proposition}\label{prop2.8}
Suppose that $X$ is homotopy commutative.
Then there is a homotopy commutative diagram:
\[
\xymatrix{
\Sigma(\Omega B\wedge \Omega B)\ar@{->}[dd]_{-\tau}&{}\simeq{}&\Omega
B*\Omega B\ar@{->}[dr]^{\Gamma}\ar@{->}[dd]_{\tau}&\\
&&&E\\
\Sigma(\Omega B\wedge \Omega B)&{}\simeq{}&\Omega B*\Omega
B\ar@{->}[ur]_{\Gamma} 
}
\]
\end{proposition}
\begin{proof}
By \cite[11.1.11]{Nei10a}, we can assume that the homotopy
of commutation $\mu_t(x_1,x_2)$ is stationary on the axes, where
$\mu_0(x_1,x_2)=\mu(x_1,x_2)$ and $\mu_1(x_1,x_2)=\mu(x_2,x_1)$. We then
define
\[
\overline{\omega}_t(s)=\mu_t\left(\varphi(\omega_1(s)),\varphi(\omega_2(s))\right)
\]
and use this to define $\Gamma_t\colon \Omega B *\Omega B\to E$, a homotopy
between $\Gamma$ and $\Gamma \tau$.
\end{proof}

\section{Theriault's Criterion}\label{subsec2.3}

The map $\omega\colon \Omega B * \Omega B\to B\vee B$ plays a role in a
useful condition for an $H$-space to have a homotopy-Abelian structure.
\begin{proposition}[{\cite[4.12]{T01}}]\label{prop2.9}
Suppose that
\[
\xymatrix{
\Omega B\ar@{->}[r]^{h}&F\ar@{->}[r]^{i}&E\ar@{->}[r]^{\pi}&B
}
\]
is a fibration sequence in which $i$ is null homotopic.
Suppose that there is a lifting $\overline{\omega}$ of $\nabla\omega$ in
the
diagram:
\[
\xymatrix{
\Omega B*\Omega B\ar@{->}[r]^-<<<{\overline{\omega}}\ar@{->}[d]_{\omega}&E\ar@{->}[d]_{\pi}\\ 
B\vee B\ar@{->}[r]^{\nabla}&B
}
\]
Then the $H$-space structure defined on $F$ by any right
inverse $g\colon F\to \Omega B$ of $h$ defines a homotopy-Abelian
$H$-space structure.
\end{proposition}
\begin{proof}
For any pointed space $Z$ let $G=[Z,\Omega B]$ and
$X=[Z,F]$. Then $G$ is a group which acts on $X$ via the
action map
\[
\xymatrix{
\Omega B\times F\ar@{->}[r]^-<<<{a}&F.
}
\]
Since $h$ has a right homotopy inverse, the orbit
of $*\in X$ is all of $X$. The adjoint of the composition
\[
\xymatrix{
\Sigma(\Omega B\wedge \Omega B)\simeq \Omega B\times \Omega
B\ar@{->}[r]^-<<<{\omega}&B\vee B\ar@{->}[r]^-<<<{\nabla}&B
}
\]
is well known to be homotopic to the 
commuter map
\[
\xymatrix{
\Omega B\wedge\Omega B\ar@{->}[r]^-<<<{c}&\Omega B.
}
\]
(See, for example, \cite[3.4]{Gra11}). Consequently, the
existence of $\overline{w}$ implies that every commutator in~$G$
acts trivially on $*\in X$; i.e., $g(h*)=h(g*)$.
Let $N=\{g\mid g*=*\}$ be the stabilizer of~$*$. Then $N$
is a normal subgroup since if $g\in N$, $(hgh^{-1})(*)=
(hg)(h^{-1}*)=h^{-1}(hg*)=g*=*$. Consequently
$X=G/N$ is a quotient group of~$G$. It is Abelian
since $g(h*)=h(g*)$. This group structure on
$X=[Z,F]$ is natural for maps in~$Z$. Apply this in
case $Z=F\times F$ and $Z=F\times F\times F$ to construct a homotopy-Abelian $H$-space structure on~$F$. 
\end{proof}

Now recall that we have fixed $G_k$ and
$\varphi_k\index{varphik@$\varphi_k$|LB}\colon G_k\to S^{2n+1}\{p^r\}$ and we set $G=\bigcup G_k$. Let
$E=\bigcup E_k$. Consider the
commutative diagram:\setcounter{equation}{9}
\begin{equation}\label{eq2.10}
\begin{split}
\xymatrix{
&\Omega S^{2n+1}\{p^r\}\ar@{->}[d]\\
\Omega G * \Omega G\ar@{->}[d]_{\omega}\ar@{->}[r]^{\Gamma}&E\ar@{->}[d]\\
G\vee G\ar@{->}[r]^{\nabla}&G
}
\end{split}
\end{equation}\addtocounter{Theorem}{1}
\begin{proposition}\label{prop2.11}
If the composition:
\[
\xymatrix{
\Omega G\ast \Omega
G\ar@{->}[r]^-<<<{\enlarge{\Gamma}}&E\ar@{->}[r]^->>>>{\enlarge{\nu_{\infty}}}&BW_n
}
\]
is null homotopic, the induced $H$-space structure on~$T$
is homotopy Abelian.
\end{proposition}
\begin{proof}
Compare (\ref{eq2.10}) with (\ref{eq1.6new}) and apply (\ref{prop2.9}).
\end{proof}

\section{Compatibility of Modifications}\label{subsec2.4}
In the sequel we will construct maps $\nu_k\colon E_k\to BW_n$
such that the composition:
\[
\xymatrix{
\Omega G_k\ast \Omega G_k\ar@{->}[r]^-<<<{\enlarge{\Gamma_k}}&
E_k\ar@{->}[r]^->>>>{\enlarge{\nu_k}}&BW_n
}
\]
is null homotopic. We begin with an arbitrary choice of
$\nu_k$ as in~\cite{GT10} and modify it using the controlled
extension theorem~(\ref{theor6.1}). Having done this, we have no
reason to assume that the composition:
\[
\xymatrix{
E_{k-1}\ar@{->}[r]^{\enlarge{e_k}}
&E_k\ar@{->}[r]^->>>>{\enlarge{\nu_k}}&BW_n
}
\]\index{$e_k$|LB}
is homotopic to $\nu_{k-1}$. However, all modifications occur
in dimensions $2np^k$ and larger, so we can assume
these maps agree up to dimension $2np^k-2$. The proof of
Theorem~A will then follow from
\begin{proposition}\label{prop2.12}
Suppose we can construct maps $\nu_k\index{nuk@$\nu_k$|LB}\colon E_k\to BW_n$ for each $k\geqslant 0$ such that:

\textup{(a)}\hspace*{0.5em}$\xymatrix{\Omega^2S^{2n+1}\ar@{->}[r]^->>>>{\enlarge{\Omega\partial}}&\Omega
S^{2n+1}\{p^r\}\ar@{->}[r]^-<<<<{\enlarge{\partial'}}&E_k\ar@{->}[r]^->>>{\enlarge{\nu_k}}&BW_n}$\newline
induces an isomorphism in $H_{2np-2}$

\textup{(b)}\hspace*{0.5em}The composition:
\[
\xymatrix{
\Omega G_k\ast \Omega
G_k\ar@{->}[r]^-<<<{\enlarge{\Gamma_k}}&E_k\ar@{->}[r]^->>>>{\enlarge{\nu_k}}&BW_n
}
\]
is null homotopic for each $k\geqslant 0$.

\textup{(c)}\hspace*{0.5em}The restrictions of $\nu_ke_k$ and $\nu_{k-1}$
to the $2np^k-2$
skeleton of $E_{k-1}$ are homotopic.

Then there is a map
\[
\nu_{\infty}\index{nuinfty@$\nu_{\infty}$|LB}\colon E\to BW_n
\]
such that $\nu_{\infty}\Gamma$ is null homotopic and the induced
$H$-space structure on~$T$ is homotopy Abelian.
\end{proposition}
\begin{proof}
Since the inclusion:
\[
\bigcup_{k\geqslant 0}E_k^{2np^{k+1}-2}\to E
\]
is a homotopy equivalence, we can define $\nu_{\infty}\colon E\to BW_n$
which restricts to $\nu_k$ on $E_k^{2np^{k+1}-2}$. Since $\partial'\colon
\Omega S^{2n+1}\{p^r\}\to E$
factors through $E_k$\index{$E_k$|LB} for each $k$, we have a homotopy
commutative diagram:
\[
\xymatrix{
\Omega S^{2n+1}\{p^r\}\ar@{=}[r]\ar@{->}[dd]&\Omega
S^{2n+1}\{p^r\}\ar@{->}[dr]^{\enlarge{H}}\ar@{->}[dd]&\\
&&BW_n\\
E_k\ar@{->}[r]&E\ar@{->}[ru]_{\enlarge{\nu_k}}&
}
\]
From the case $k=0$, we see that
the composition:
\[
\xymatrix{
\Omega^2 S^{1n+1}\ar@{->}[r]^->>>>{\enlarge{\Omega\partial}}&\Omega
S^{2n+1}\{p^r\}\ar@{->}[r]^-<<<<{\enlarge{H}}&BW_n
}
\]
induces an isomorphism in $H_{2np-2}$. Using cup product
and the Bockstein, we can conclude that this
composition is an epimorphism in all dimensions and
consequently the fiber of this composition is $S^{2n-1}$. Now
compare with (\ref{eq1.3new}) to see that the fiber of~$H$ is an
Anick space.

In the diagram
\[
\xymatrix{
(\Omega G_k\ast\Omega G_k)^{2np^{k+1}-2}\ar@{->}[r]\ar@{->}[d]&
E_k^{2np^{k+1}-2}
\ar@{->}[r]^{\enlarge{\nu_k}}
\ar@{->}[d]
&
BW_n
\ar@{=}[d]\\
\Omega G\ast \Omega G\ar@{->}[r]^{\enlarge{\Gamma}}&E\ar@{->}[r]^{\enlarge{\nu_{\infty}}}&BW_n,
}
\]
the upper composition is null homotopic for each~$k$,
so the lower composition is null homotopic as well, thus
the result follows from~\ref{prop2.11}.
\end{proof}

\section{Properties of $G$ and $T$}\label{subsec2.5}

We now recall, for future use, the
properties of $G$ and $T$ that we will be using in the sequel.
\begin{Theorem}[{\cite{GT10}}]\label{theor2.14}
For $p\geqslant 3$, $r\geqslant 1$ and $n\geqslant 1$, there is an
Anick space $T$; i.e., there is a fibration sequence
\[
\xymatrix{
\Omega^2S^{2n+1}\ar@{->}[r]^{\pi_n}&S^{2n-1}\ar@{->}[r]&T\ar@{->}[r]&\Omega S^{2n+1}
}
\]
such that the composition
\[
\xymatrix{
S^{2n-1}\ar@{->}[r]^{E^2}&\Omega^2S^{2n+1}\ar@{->}[r]^{\pi_n}&S^{2n+1}
}
\]
has degree $p^r\!$. Furthermore we have the following properties:

\textup{(a)}\hspace*{0.5em}There exists a space $G$ and maps $f\index{$f$|LB},
g\index{$g$|LB}, h\index{$h$|LB}$
such that the compositions
\begin{align*}
&\xymatrix{
G\ar@{->}[r]^{f}&\Sigma T\ar@{->}[r]^{\widetilde{g}}&G
}\\
&\xymatrix{
T\ar@{->}[r]^{g}&\Omega G\ar@{->}[r]^{h}&T
}
\end{align*}
are homotopic to the identity.

\textup{(b)}\hspace*{0.5em}Both $T$ and $G$ are atomic and $p$-complete.

\textup{(c)}\hspace*{0.5em}The homotopy type of $T$ and $G$ are unique
satisfying these conditions.

\textup{(d)}\hspace*{0.5em}$\Sigma T\wedge T$ is a wedge of Moore spaces.

\textup{(e)}\hspace*{0.5em}For any $H$-space structure defined by the maps
$g,h$ in~\textup{(a)}, there is an $H$-map   
\[
E\index{$E$|LB}\colon T\to \Omega S^{2n+1}\{p^r\}
\]
such that $Eh\sim \Omega \varphi$.

\textup{(f)}\hspace*{0.5em}$H^*(T;Z)$ is generated by classes $v_i\index{$v_i$|LB}\in
H^{2np^i}(T;Z)$
for each $i\geqslant 0$ subject to the relations $v_i^p=pv_{i+1}$ and
$p^rv_0=0$.

\textup{(g)}\hspace*{0.5em}$H_*(T)\simeq Z/p[v]\otimes \bigwedge(u)$ where
$|v|\index{$v$|LB}=2n$ and $|u|\index{$u$|LB}=2n-1$.
Furthermore $\beta^{(r+i)}(v^{p^i})=uv^{{p^i}-1}$.

\textup{(h)}\hspace*{0.5em}$H_m(G;Z)=\begin{cases}
Z/p^{r+i}&\text{if}\ m=2np^i\\
0&\text{otherwise}.
\end{cases}$

\textup{(i)}\hspace*{0.5em}$\Sigma G$ and $\Sigma^2T$ are each
homotopy equivalent to a wedge of
Moore spaces.

\textup{(j)}\hspace*{0.5em}Any choice of $\nu_k\colon E_k\to BW_n$ in~(\ref{eq1.6new})
has a right
homotopy inverse.

\textup{(k)}\hspace*{0.5em}$R$ is homotopy equivalent
to a wedge of $\bmod\, p^s$ Moore spaces for $s\geqslant r$.

\textup{(l)}\hspace*{0.5em}$\Sigma G\wedge G$ is homotopy equivalent to a
wedge of
Moore spaces
\end{Theorem}
\begin{proof}
Most of these are restatements of results in~\cite{GT10}. Properties (a),
(b), (c), (d), (e), and (f)
are
respectively 4.4, 4.7, 4.9, 4.3(m), 4.6, and~4.1 of \cite{GT10}.
Property~(g) follows immediately by applying the
Serre spectral sequence to~1.1. Properties~(h) and~(i)
are 4.3(c) and 4.5 respectively. For~(j), a right
homotopy inverse is given by the composition
\[
BW_n\to BW_n\times S^{4n-1}\simeq E_{(1)}\to E_0\to E_k
\]
constructed from the proof of~\ref{prop3.5}. 
 Property (k) is 4.8. For (l) note
that $\Omega G\simeq T\times \Sigma R$
since $hg\sim 1$. Thus
\[
\Sigma\Omega G\simeq \Sigma T\vee \Sigma\Omega R\vee\Sigma T\wedge\Omega
R\simeq W_1\vee \Sigma T
\]
where $W_1$ is a wedge of Moore spaces by (i) and (k). Thus $\Sigma^2\Omega G\in \cW$ and
\begin{align*}
\Sigma \Omega G\wedge\Omega G&\simeq W_1\wedge \Omega G\vee \Sigma
T\wedge\Omega G\simeq W_2\vee T\wedge\Sigma\Omega G\\
&\simeq W_3\vee \Sigma T\wedge T
\end{align*}
which is in a wedge of Moore space by property~(d).
\end{proof}

\chapter{Whitehead Products}\label{chap3}

In this chapter we will review and extend some results in \cite{Gra11}
which generalize the notion of Whitehead
products. In particular\footnote{Throughout Chapter~\ref{chap3}, $G$ and~$H$ will designate
an arbitrary co-$H$ space. Then we will return
$G$ as to designating
the co-$H$ space corresponding to~$T$ in future chapters.} given two simply
connected co-$H$ spaces $G$ and $H$, we will construct a new
co-$H$ space $G\circ H$\index{$G\circ H$|LB} and a cofibration sequence
\[
\xymatrix{
G\circ H\ar@{->}[r]^{W}& G\vee H\ar@{->}[r]& G\times H.
}
\]
Composition with $W$ defines a more general notion of Whitehead products. In
section~\ref{subsec3.1} we will review the material in~\cite{Gra11}.
In section~\ref{subsec3.2} we will discuss relative Whitehead
products and $H$-space based Whitehead products
which have been developed for homotopy groups with
coefficients in $Z/p^r$ by Neisendorfer~\cite{Nei10a}. We will
define these in the total space of
a principal fibration using co-$H$ spaces in place of
Moore spaces. In section~\ref{subsec3.3} we will use these
products to decompose $\Omega G*\Omega H$ and $\Omega G\ltimes H$ when $G$
and $H$ are co-$H$ spaces and decompose the map~$\Gamma$ from
section~\ref{subsec2.2}
as a wedge of iterated Whitehead products when the base is a co-$H$ space. In
section~\ref{subsec3.4} we recall and generalize slightly the
results of Neisendorfer in the case that $G$ and $H$ are Moore spaces.

\section{Defining Whitehead Products Using co-$H$ Spaces}\label{subsec3.1}

Given two simply connected co-$H$ spaces $G,H$, we introduce a new
co-$H$ space $G\circ H$ together with a cofibration sequence:
\[
\xymatrix{
G\circ H\ar@{->}[r]^-{W}&G\vee H\ar@{->}[r]&G\times H.
}
\]
To do this, suppose that $G$ and~$H$ are given co-$H$ space
structures by constructing right inverses to the respective
evaluation maps:
\begin{gather*}
\xymatrix{G\ar@{->}[r]^-{\nu_1}&\Sigma\Omega G\ar@{->}[r]^-{\epsilon_1}&G}\\
\xymatrix{H\ar@{->}[r]^-{\nu_2}&\Sigma\Omega H\ar@{->}[r]^-{\epsilon_2}&H.}
\end{gather*}\index{epsilon@$\epsilon$|LB}
We define a self map $e\colon \Sigma(\Omega G\wedge\Omega H)\to
\Sigma(\Omega G\wedge\Omega H)$ as the
composition
\[
\kern-3pt\fontsize{9}{11}\selectfont
\xymatrix{
\Sigma(\Omega G\wedge\Omega H)
\ar@{->}[r]^->>>>{\epsilon_1\wedge 1}
&G\wedge \Omega H
\ar@{->}[r]^->>>>{\nu_1\wedge 1}
&\Sigma(\Omega G\wedge\Omega H)
\ar@{->}[r]^->>>>{1\wedge \epsilon_2}
&\Omega G\wedge H
\ar@{->}[r]^->>>>{1\wedge \nu_2}
&\Sigma(\Omega G\wedge\Omega H);
}
\]
$G\circ H$ is then defined as the telescopic direct limit of~$e$. We
then have:
\begin{proposition}[{\cite[2.1,2.3]{Gra11}}]\label{prop3.1}
The identity map of~$G\circ H$ factors:
\[
\xymatrix{
G\circ H\ar@{->}[r]^-{\psi}&
\Sigma(\Omega G\wedge\Omega H)\ar@{->}[r]^-{\theta} 
& G\circ H
.}
\]

Furthermore, if $f\colon G\to G'$ and $g\colon H\to H'$ are co-$H$ maps,
there are induced co-$H$ maps so that the diagram
\[
\xymatrix{
G\circ H\ar@{->}[d]_{f\circ g}\ar@{->}[r]^-{\psi}&\Sigma(\Omega G\wedge \Omega
H)\ar@{->}[r]^-{\theta}\ar@{->}[d]_{\Sigma(\Omega f\wedge \Omega g)}&G\circ
H\ar@{->}[d]_{f\circ g}\\
G'\circ H'\ar@{->}[r]^-{\psi'}&\Sigma(\Omega G'\wedge \Omega
H')\ar@{->}[r]^-{\theta'}&G'\circ H'
}
\]
commutes up to homotopy.
\end{proposition}

Since $G\circ H$ is the limit of the telescope defined by~$e$,
$\theta e\sim \theta$, so the composition
\[
\xymatrix{
G\circ H\ar@{->}[r]^-{\psi}&\Sigma(\Omega G\wedge \Omega H)\ar@{->}[r]^-e&
\Sigma(\Omega G\wedge \Omega H)\ar@{->}[r]^-{\theta}&G\circ H 
}
\]
is homotopic to the identity. The map~$e$, however, is a
composition of 4 maps between co-$H$ spaces, and thus $G\circ H$
is a retract of 3 different co-$H$ spaces and one of them,
$\Sigma(\Omega G\wedge \Omega H)$, in two potentially distinct ways. This
provides 4
potentially distinct co-$H$ space structures on $G\circ H$. We
choose the structure defined by $\psi$ and $\theta$; viz.,
\[
\xymatrix{
G\circ H\ar@{->}[r]^-{\psi}&\Sigma(\Omega G\wedge\Omega
H)\ar@{->}[r]^-{\Sigma\widetilde{\theta}}&\Sigma\Omega(G\circ H)
}
\]
or equivalently
\begin{multline*}
\kern-7pt\xymatrix@C=12pt{
G\!{}\circ{}\! H\ar@{->}[r]^-{\psi}&\Sigma(\Omega G\!{}\wedge{}\!\Omega H)
\ar@{->}[r]&\Sigma(\Omega G\!{}\wedge\Omega H)\vee\Sigma(\Omega
G\!{}\wedge{}\! \Omega
H)}\hspace*{-6pt}
\xymatrix@C=26pt{
\ar@{->}[r]^->>>>>{\theta\!{}\wedge{}\!\theta}&G\!{}\circ{}\! H\!{}\vee{}\!
G\!{}\circ{}\! H  
}
\end{multline*}
where $\widetilde{\theta}$ is the adjoint of~$\theta$.
\begin{proposition}[{\cite[2.3,2.5]{Gra11}}]\label{prop3.2}
There are co-$H$ equivalences\linebreak[4]
\mbox{$G\circ \Sigma X\simeq G\wedge X$}, $\Sigma(G\circ H)\simeq G\wedge H$
which are natural for co-$H$ maps in $G$ and~$H$
and continuous maps in~$X$.
\end{proposition}
\begin{proposition}[{\cite[3.3]{Gra12}}]\label{prop3.3}
There is a natural cofibration sequence
\[
\xymatrix{
G\circ H\ar@{->}[r]^-W&G\vee H\ar@{->}[r]&G\times H  
}
\]
where $W$\index{$W$|LB} is the composition:
\[
\xymatrix{
G\circ H\ar@{->}[r]^-<<<<{\psi}&\Sigma(\Omega G\wedge\Omega
H)\ar@{->}[r]^->>>>{\omega}&G\vee H
.}
\]
\end{proposition}
\begin{definition}\label{def3.4}
Let $\alpha\colon G\to X$ and $\beta\colon H\to X$. We define
the Whitehead product\footnote{We use the notation $\{\alpha,\beta\}$
rather than the usual $[\alpha,\beta]$
since in an important application we need to make a distinction.
That is the case when $G$ and~$H$ are both Moore spaces. In this
case $G\circ H$ is a wedge of two Moore spaces. By choosing the
higher dimensional one, Neisendorfer~\cite{Nei80} defines internal
Whitehead products in homotopy with coefficients in $Z/p^r$. This is
denoted $[\alpha,\beta]$, while $\{\alpha,\beta\}$ is the ``external''
Whitehead product.}
\[
\{\alpha,\beta\}\colon G\circ H\to X
\]
as the composition
\[
\xymatrix{
G\circ H\ar@{->}[r]^-W&G\vee H\ar@{->}[r]^-{\alpha\vee\beta}&X  
.}
\]
\end{definition}
\begin{proposition}\label{prop3.5}
Each Whitehead product $\{\alpha,\beta\}\colon G\circ H\to X$
factors through the ``universal Whitehead product''
\[
w\index{$w$|LB}=\nabla\omega\colon \Sigma(\Omega X\wedge\Omega X)\to X\vee X\to X.
\]
\end{proposition}
\begin{proof}
The right hand square in the following diagram
determines a map from the upper row which is a
cofibration sequence to the lower row which is a
fibration sequence
\[
\xymatrix{
G\circ H\ar@{->}[r]^-W\ar@{->}[d]_{\xi}&G\vee H\ar@{->}[r]\ar@{->}[d]_{\alpha\vee\beta}&G\times H\ar@{->}[d]\\  
\Sigma(\Omega X\wedge\Omega X)\ar@{->}[r]^-{\omega}&X\vee X\ar@{->}[r]& X\times X 
}
\]
and $\xi$ is unique up to homotopy since $G\circ H$ is a co-$H$
space and $\Omega(X\vee X)\simeq \Omega(X\times X)\times
\Omega\Sigma(\Omega X\wedge\Omega X)$. It
follows that
\[
\makebox[78pt]{}\{\alpha,\beta\}=\nabla(\alpha\vee\beta)W\sim
\nabla\omega\xi=w\xi.\makebox[78pt]{}\rlap{\qed}
\]\noqed
\end{proof}

\section{$H$-space Based and Relative Whitehead Products}\label{subsec3.2}

In this section we will discuss
$H$-space based Whitehead products and relative
Whitehead products. In the case that $G$ and $H$ are
Moore spaces, this material is covered in~\cite{Nei10a}, and
what we present is a mild generalization. We need to consider
Whitehead products instead of their adjoints---the
Samelson products (which Neisendorfer considered) since the domains are not necessarily suspensions. We also consider principal fibrations,
so these products occur in the total space rather than the fiber of a fibration as in 
Neisendorfer's version. We wish to thank Joe Neisendorfer for
several interesting conversations during the
development of this material.

We begin with a principal fibration
\[
\xymatrix{
\Omega X\ar@{->}[r]^-{i}&E\ar@{->}[r]^-{\pi}&B  
}
\]
induced by a map $\varphi\colon B\to X$. The (external)
relative
Whitehead product then is a pairing
\[
[G,B]\times[H,E]\to [G\circ H,E].
\]
In the case that $X$ is a homotopy commutative $H$-space with
strict unit, we also define the $H$-space based
Whitehead product. It is a pairing
\[
[G,B]\times[H,B]\to [G\circ H,E].
\]

Suppose we are given maps:
\[
\xymatrix{
G\ar@{->}[r]^{\enlarge{\alpha}}&B,\ H\ar@{->}[r]^{\enlarge{\beta}}&B,\
G\ar@{->}[r]^{\enlarge{\gamma}}& E,\ H\ar@{->}[r]^{\enlarge{\delta}}& E.
}
\]

We will use the notation
\[
\{\alpha,\gamma\}_{r}\in[G\circ H,E]
\]\index{\{\ ,\ \}$_r$|LB}
for the relative Whitehead product and
\[
\{\alpha,\beta\}_{\times}\in [G\circ H,E]
\index{\{\ ,\ \}$_{\times}$|LB}
\]
for the $H$-space based Whitehead product. These
products and the absolute Whitehead product
are related by the following formulas to be proved:
\begin{align*}
\pi\{\alpha,\beta\}_{\times}\sim\{\alpha,\beta\}
\colon G\circ H\to B;\tag{3.6\textup{c}}\\
\{\pi\gamma,\pi\delta\}_{\times}\sim\{\gamma,\delta\}\colon G\circ H\to E;\tag{3.6\textup{e}}\\
\pi\{\alpha,\delta\}_r\sim\{\alpha,\pi\delta\}\colon G\circ H\to
B;\tag{3.11\textup{c}}\\
\{\pi\gamma,\delta\}_r\sim\{\gamma,\delta\}\colon G\circ H\to
E.\tag{3.11\textup{e}}\\
\{\alpha,\delta\}_r\sim\{\alpha,\pi\delta\}_{\times}\colon G\circ H\to
E;\tag{3.12}
\end{align*}\index{\{\ ,\ \}|LB}

We begin with the $H$-space based Whitehead
product. These are defined using the map $\Gamma$ from~(\ref{prop2.7}).  
The product $\{\alpha,\beta\}_{\times}$ is defined as the homotopy
class of the upper composition in the diagram:
\[
\xymatrix{
G\circ H\ar@{->}[r]^-<<<{\psi}\ar@{->}[dr]_W&
\Sigma(\Omega G\wedge\Omega H)\simeq\Omega G *\Omega H\ar@{->}[d]_{\omega}\ar@{->}[r]&\Omega
B*\Omega B\ar@{->}[r]^->>>>>{\Gamma}\ar@{->}[d]_{\omega}&
E\ar@{->}[d]_{\pi}\\
&G\vee H\ar@{->}[r]^-{\alpha\vee\beta}&B\vee B\ar@{->}[r]^{\nabla}&B.}
\]
\begin{proposition}\label{prop3.6}
Given $\alpha\colon G\to B$ and $\beta\colon H\to B$, the
homotopy class of the $H$-space based Whitehead product
\[
\{\alpha,\beta\}_{\times}\colon G\circ H\to E
\]
depends only on the homotopy classes of $\alpha$ and~$\beta$.
Furthermore

\textup{(a)}\hspace*{0.5em}If $f\colon G'\to G$ and $g\colon H'\to H$ are
co-$H$ maps,
\[
\{\alpha,\beta\}_{\times}(f\circ g)\sim\{\alpha f,\beta g\}_{\times}.
\]

\textup{(b)}\hspace*{0.5em}Given an induced fibration
\[
\xymatrix{
E'\ar@{->}[r]^-{\widetilde{\xi}}\ar@{->}[d]&
E\ar@{->}[d]\\
B'\ar@{->}[r]^-{\xi}&
B\ar@{->}[r]^-{\varphi}&
X
}
\]
and $\alpha'\colon G\to B'$, $\beta'\colon H\to B'$, we have
\[
\widetilde{\xi}\left\{\alpha',\beta'\right\}_{\times}\sim
\left\{\xi\alpha',\xi\beta'\right\}_{\times}\colon G\circ H\to E.
\]

\textup{(c)}\hspace*{0.5em}$\pi\{\alpha,\beta\}_{\times}\sim\{\alpha,\beta\}\colon
G\circ H\to B$.

\textup{(d)}\hspace*{0.5em}Suppose $\eta\colon
X\to X'$ is a strict $H$-map
and we have a pointwise commutative diagram
\[
\xymatrix{
B\ar@{->}[r]^-{\xi}\ar@{->}[d]_{\varphi}&
B'\ar@{->}[d]_{\varphi'}\\
X\ar@{->}[r]^-{\eta}&
X'
}
\]
which defines a map of principal fibrations:
\[
\xymatrix{
E\ar@{->}[r]^-{\widetilde{\xi}}\ar@{->}[d]&
E'\ar@{->}[d]\\
B\ar@{->}[r]^-{\xi}&
B'
.}
\]
Then
\[
\widetilde{\xi}\{\alpha,\beta\}_{\times}\sim\{\xi\alpha,\xi\beta\}_{\times}\colon
G\circ H\to E'.
\]

\textup{(e)}\hspace*{0.5em}$\{\pi\gamma,\pi\delta\}_{\times}\sim\{\gamma,\delta\}\colon
G\circ H\to E$.
\end{proposition}
\begin{proof}
These all follow directly from the definition
except for~\textup{(e)}. To prove this we apply \textup{(d)} to the diagram
\[
\xymatrix{
E\ar@{->}[r]^-{\pi}\ar@{->}[d]_k&
B\ar@{->}[d]_{\varphi}\\
PX\ar@{->}[r]^-{\epsilon}&
X
}
\]
where $k(b,\omega)=\omega$. Give $PX$ the $H$-space
structure
of pointwise multiplication of paths in $X$. Then
$\epsilon$ is a strict $H$-map. This gives a map of principal
fibrations:
\begin{align*}
&\makebox[152pt]{}\xymatrix{
\widetilde{E}\ar@{->}[r]^-{\widetilde{\pi}}\ar@{->}[d]_e&
E\ar@{->}[d]_{\pi}\\
E\ar@{->}[r]^-{\pi}&
B.
}\\[-15pt]
&\makebox[341pt]{}\qed
\end{align*}
\noqed
\end{proof}

By \textup{(d)} we have
\[
\widetilde{\pi}\{\gamma,\delta\}_{\times}\sim\{\pi\gamma,
\pi\delta\}_{\times}.
\]
It suffices to show that $\widetilde{\pi}\sim e$ by applying
part~\textup{(c)} to the left 
hand fibration. The space $\widetilde{E}\subset (B\times PX)\times PPX$ can
be
described as follows
\[
\widetilde{E}=\left\{(b,\sigma)\in B\times PPX\mid
\varphi(b)=\sigma(1,1),\sigma(s,0)=\sigma(0,t)=*\right\}
\]
with $\widetilde{\pi}(b,\sigma)=(b,\omega)$ where $\omega(t)=\sigma(t,1)$
and $e(b,\sigma)=(b,\omega')$
where $\omega'(t)=\sigma(1,t)$. Define $F\colon \widetilde{E}\times
I\times I\to X$ by $\sigma$. The
result then follows from:
\begin{HRL}\label{HRL3.7}
Suppose $F\colon A\times I\times I\to B$ and
$F(a,0,t)=F(a,s,0)=F(*,s,t)=*$. Then there is a
homotopy 
\[
H\colon A\times I\times I\to B
\]
such that
\begin{align*}
H(a,0,t)&=F(a,1,t)\\
H(a,1,t)&=F(a,t,1)\\
H(a,s,1)&=F(a,1,1)\\
H(a,s,0)&=H(*,s,t)=*
\end{align*}
\end{HRL}
\begin{proof}
The left side and the bottom of the square are
mapped to the basepoint. By rotating from the top
to the right hand side pivoting at the point $(1,1)$, we
obtain the required homotopy.
\end{proof}

We now describe the relative Whitehead product.
We assume a principal fibration
\[
\xymatrix{
\Omega X\ar@{->}[r]^{i}&E\ar@{->}[r]^{\pi}&B
}
\]
induced by a map $\varphi\colon B\to X$ (as in section~\ref{subsec2.1}), but
we won't assume an $H$-space structure on~$X$. 
Define $k\colon E\to PX$ by the second component, so $k(e)(1)=
\varphi\pi(e)$. The principal action
\[
a\colon \Omega X\times \xymatrix{E\ar@{->}[r]&E}
\]
is defined by the formula
\[
a(\omega, e)=(\pi(e),\omega')
\]
where $\omega'$ is given by\[
\omega'(t)=\begin{cases}
\omega(2t)&0\leqslant t\leqslant 1/2\\
k(e)(2t-1)&1/2\leqslant t\leqslant 1.
\end{cases}
\]
We then describe a strictly commutative diagram of
vertical fibration sequences
\[
\xymatrix{
\Omega B\ar@{->}[r]^{d}\ar@{->}[d]&\Omega X\ar@{->}[d]_{i}\\
\Omega B\times E\cup PB\ar@{->}[r]^-<<<<{\Gamma'}
\ar@{->}[d]_{\epsilon\vee\pi_2
}&E\ar@{->}[d]_{\pi}\\
B\vee E\ar@{->}[r]^{1\vee \pi}&B
}
\]\index{Gammaprime@$\Gamma^{\prime}$|LB}
where $\Omega B\times E\cup PB$ is to be considered as a subspace of
$PB\times E$. The map $\epsilon \colon PB\to B$ is endpoint evaluation
and $\Gamma'$ is defind by the formula
\begin{align*}
\Gamma'(\omega,e)&=a(\varphi\omega,e)&&\text{for}\ (\omega,e)\in \Omega
B\times E\\
\Gamma'(\omega)&=(\omega(1),\varphi\widetilde{\omega})&&\text{for}\
\omega\in PB 
\end{align*}
where\footnote{For convenience, we extend maps $\xymatrix{[0,1]\ar@{->}[r]^{f} &X}$ to
the
real line by $f(x)=f(0)$ for $x<0$ and $f(x)=f(1)$ for $x>1$.}
$\widetilde{\omega}(t)=\omega (2t)$, and the map $d\colon \Omega B\to
\Omega X$ is given
by $d(\omega)=\varphi\widetilde{\omega}$. The left hand fibration is the
principal
fibration induced by the projection $\pi_1\colon B\vee E\to B$.

Observe that
\[
\Omega B\times E\cup PB\simeq \Omega B\times E\cup C(\Omega B)\simeq \Omega
B\ltimes E.
\]

We record an important commutative diagram\addtocounter{equation}{7}
\begin{equation}\label{eq3.8}
\begin{split}
\xymatrix{
\Omega B\times E\ar@{->}[d]\ar@{->}[r]^{\Omega \varphi\times 1}&\Omega
X\times E\ar@{->}[d]_{a}\\
\Omega B\ltimes E\simeq \Omega B\times E\cup PB\ar@{->}[r]^-<<<<{\Gamma'}&E
}
\end{split}
\end{equation}
which will be useful in evaluating the relative
Whitehead products in homology.

Consider the strictly commutative square
\[
\xymatrix{
B\vee E\ar@{->}[d]\ar@{=}[r]&B\vee E\ar@{->}[d]_{\pi_1}\\
B\times E\ar@{->}[r]^{\pi_1}&B.
}
\]
Taking homotopy fibers vertically, we obtain a
diagram of principal fibrations
\[
\xymatrix{
\Omega(B\times E)\ar@{->}[d]\ar@{->}[r]^{\Omega \pi_1}&\Omega B\ar@{->}[d]\\
\Omega B\times PE\cup PB\times \Omega E\ar@{->}[r]^-<<<{\zeta}\ar@{->}[d]&\Omega B\times
E\cup PB\ar@{->}[d]\\
B\vee E\ar@{=}[r]&B\vee E
}
\]
The map $\zeta$\index{zeta@$\zeta$|LB} is defined by the formula
\[
\xymatrix{
\Omega B\times PE\ar@{->}[r]^{1\times \epsilon}&\Omega B\times E\\
PB\times \Omega E\ar@{->}[r]^-<<<<<{\pi_1}&PB
}
\]
Combining these diagrams, we obtain a strictly
commutative diagram
\begin{equation}\label{eq3.9}
\begin{split}
\xymatrix@C=14pt{
\Omega B*\Omega E\ar@{->}[rrd]_{W}
&\hspace*{-75pt}\llap{\makebox[0pt]{$\simeq$}}\hspace*{-75pt}
&\Omega B\times PE\cup PB\times \Omega E\ar@{->}[r]^-<<{\zeta}\ar@{->}[d]
&\Omega B\times E\cup PB\ar@{->}[r]^-<<<{\Gamma'}\ar@{->}[d] 
&E\ar@{->}[d]_{\pi}\\
&&B\vee E\ar@{=}[r]
&B\vee E\ar@{->}[r]^{1\vee\pi}&B.
}
\end{split}
\end{equation}
For $\alpha\colon G\to B$ and $\delta\colon H\to E$,
we define the relative Whitehead product
\[
\{\alpha,\delta\}_r
\colon G\circ H\to E
\]
as the composition
\begin{equation}\label{eq3.10}
\xymatrix@C=14pt{
G\circ H\ar@{->}[r]^->>{\psi}&\Omega G*\Omega H\ar@{->}[r]&\Omega B *\Omega
E\ar@{->}[r]^->>>{\zeta}&\Omega B\times E\cup PB\ar@{->}[r]^-<<{\Gamma'}&E
}
\end{equation}
and, analogous to \ref{prop3.6}, we have\addtocounter{Theorem}{3}
\begin{proposition}\label{prop3.11}
The homotopy class of the relative
Whitehead product $\{\alpha,\delta\}_r$ depends only on the 
homotopy classes of $\alpha$ and~$\delta$. Furthermore

\textup{(a)}\hspace*{0.5em}If $f\!\colon G'\to G$ and $g\colon H'\to H$ are
co-$H$ maps, then
\[
\{\alpha,\delta\}_r\cdot (f\circ g)\sim \{\alpha f, \delta g\}_r.
\]

\textup{(b)}\hspace*{0.5em}Given an induced fibration
\[
\xymatrix{
E'\ar@{->}[d]\ar@{->}[r]^-{\widetilde{\xi}}&E\ar@{->}[d]\\
B'\ar@{->}[r]^-{\xi}&B\ar@{->}[r]^-{\varphi}&X
}
\]
and classes $\alpha'\colon G\to B'$, $\delta'\colon H\to E'$, we
have
\[
\widetilde{\xi}\{\alpha',\delta'\}_r\sim\{\xi\alpha',\widetilde{\xi}\delta'\}_r.
\]

\textup{(c)}\hspace*{0.5em}$\pi\{\alpha,\delta\}_r\sim
\{\alpha,\pi\delta\}$.

\textup{(d)}\hspace*{0.5em}Suppose we have a strictly commutative diagram
\[
\xymatrix{
B\ar@{->}[r]^{\xi'}\ar@{->}[d]_{\varphi}&B'\ar@{->}[d]_{\varphi'}\\
X\ar@{->}[r]^{\eta}&X'
}
\]
inducing a map between principal fibrations:
\[
\xymatrix{
E\ar@{->}[r]^{\xi'}\ar@{->}[d]_{\pi}&E'\ar@{->}[d]_{\pi'}\\
B\ar@{->}[r]^{\xi}&B'
}
\]
Then $\xi'\{\alpha,\delta\}_r\sim\{\xi\alpha,\xi'\delta\}_r$.

\textup{(e)}\hspace*{0.5em}$\{\pi\gamma,\delta\}_r\sim\{\gamma,\delta\}
$.
\end{proposition}
\begin{proof}
All parts except~(e) follow directly from the definitions.
For part~(e) we construct a map of principal fibrations exactly
as in \ref{prop3.6}(e):
\[
\xymatrix{
\widetilde{E}\ar@{->}[r]^{\widetilde{\pi}}\ar@{->}[d]_e&E\ar@{->}[d]_{\pi}\\
E\ar@{->}[r]^{\pi}&B
}
\]
Recall that $e\sim\widetilde{\pi}$, and since $PX$ is contractible, both
$e$ and~$\widetilde{\pi}$ are homotopy equivalences. Choose
$\widetilde{\delta}\colon H\to \widetilde{E}$ such that
$\widetilde{\pi}\widetilde{\delta}\sim \delta$. Then by part~(d) we have
\[
\{\pi\gamma,\delta\}_r\sim
\widetilde{\pi}\{\gamma,\widetilde{\delta}\}_r\sim
e\{\gamma,\widetilde{\delta}\}_r\sim \{\gamma,e\widetilde{\delta}\}
\]
by part~(c). However this is homotopic to
$\{\gamma,\widetilde{\pi}\widetilde{\delta}\}\sim\{\gamma,\delta\}$.
\end{proof}

At this point we will discuss the compatibility of the $H$-space
based Whitehead product and the relative Whitehead
product.
\begin{proposition}\label{prop3.12}
Suppose $X$ is an $H$-space with strict unit 
and we are given $\alpha\colon G\to B$, $\delta\colon H\to E$. Then
\[
\{\alpha,\delta\}_r\sim
\{\alpha,\pi\delta\}_{\times}
\]
\end{proposition}
\begin{proof}
To prove this we will combine two homotopies.
The first homotopy will replace the sequential
composition of paths in the definition of the
action map~$a$ and~$\Gamma'$ with a blending of the homotopies
using the $H$-space structure in~$X$. The second homotopy
will apply the homotopy rotation lemma~(\ref{HRL3.7}). Recall the
map $k\colon B\to PX$ with the property that $\epsilon k\sim \varphi \pi$
\begin{align*}
&\makebox[148pt]{}\xymatrix{
E\ar@{->}[r]^{\pi}\ar@{->}[d]_{k}&B\ar@{->}[d]\\
PX\ar@{->}[r]^{\epsilon}&X
}\\*[-16pt]
&\makebox[341pt]{}\qed
\end{align*}\noqed
\end{proof}
\begin{lemma}\label{lem3.13}
There is a homotopy $a_s\colon \Omega X\times E\to E$ with
$a_1=a$ and $a_0$ given by the formula
\[
a_0(\omega,e)=(\pi(e),\mu(\omega(t),k(e)(t)))
\]
and a compatible homotopy $\Gamma'_s\colon \Omega B\times E\cup PB\to E$
with $\Gamma'_1=\Gamma'$ and $\Gamma'_0$ given by the formula
\begin{align*}
\Gamma'_0(\omega,e)&=(\pi(e),\mu(\varphi\omega(t),k(e)(t))\\
\Gamma'_0(\omega)&=(\omega(1),\omega\varphi)\quad \text{for}\ \omega\in PB
\end{align*}
\end{lemma}
\begin{proof}
Recall that any map $\omega\colon [0,1]\to X$
is to be extended to a map $\omega\colon R\to X$ by defining
$\omega(x)=\omega(0)$
if $x<0$ and $\omega(x)=\omega(1)$ if $x>1$. Then, for example,
\[
\mu(\omega(2t),k(e)(2t-1))=\begin{cases}
\omega(2t)&\text{if}\ 0\leqslant t\leqslant 1/2\\
k(e)(2t-1)&\text{if}\ 1/2\leqslant t\leqslant 1
\end{cases}
\]
since $\omega(1)=k(e)(0)=*$, the unit for $\mu$. We define
\[
a_s(\omega,e)=(\pi(e),\omega_s)
\]
where
\[
\omega_s(t)=\mu\left(\omega\left(\dfrac{2t}{2-s}\right),k(e)\left(\dfrac{2t-s}{2-s}\right)\right).
\]
We define $\Gamma'_s(\omega,e)=a_s(\varphi\omega,e)$ and
$\Gamma'_s(\omega)=\left(\omega(1),\varphi\omega\left(\dfrac{2t}{2-s}\right)\right)$
in case $\omega\in PB$.
\end{proof}

Using $\Gamma'_0$ we consider the composition
\[
\xymatrix@C=33pt{
\Omega B*\Omega E\ar@{->}[r]^->>>>>{\zeta}&\Omega B\times E\cup
PB\ar@{->}[r]^-<<<<{\Gamma'_0}&E.
}
\]
Using the identification \mbox{$\Omega B *\Omega E\simeq PB\times \Omega
E\cup\Omega E\times PB$},
we have the following formula for this composition\addtocounter{equation}{3}
\begin{equation}\label{eq3.14}
\xymatrix{
(\omega_1,\omega_2)\ar@{->}[r]&
(\omega_1,\omega_2(1))\ar@{->}[r]&
(\pi\omega_2(1),\mu(\varphi\omega_1(t),k(\omega_2(1))(t))).
}
\end{equation}

We now apply \ref{HRL3.7} to the homotopy
\[
F\colon PE\times I\times I\to X
\]
given by $F(\omega,s,t)=k(\omega(t))(s)$ to obtain a
homotopy
\[
H\colon PE\times I\to X
\]
with
\begin{align*}
H(\omega,1,t)&=k(\omega(1))(t)\\
H(\omega,0,t)&=\varphi\pi\omega(t)\\
H(\omega,s,0)&=*\\
H(\omega,s,1)&=\varphi\pi\omega(1).
\end{align*}
From this we construct a homotopy
\[
\Gamma_s\colon \Omega B\times PE\cup PB\times \Omega E\to E
\]
given by
\[
\Gamma_s(\omega_1,\omega_2)=\begin{cases}
\omega_1(1)\quad\mu(\varphi\omega_1(t),H(\omega_2,s,t))&\text{if}\ \omega_2(1)=*\\
\pi\omega_2(1)\quad\mu(\varphi\omega_1(t),H(\omega_2,s,t))&\text{if}\ \omega_1(1)=*.
\end{cases}
\]
Then $\Gamma_1=\Gamma(1*\Omega\pi)$ and $\Gamma_0=\Gamma'_0\zeta$. We have
proved\addtocounter{Theorem}{1}
\begin{proposition}\label{prop3.15}
If $X$ is an $H$ space with a strict unit,
there is a homotopy commutative diagram:
\begin{align*}
&\makebox[108pt]{}\xymatrix{
\Omega B\times E\cup PB\ar@{->}[r]^-<<<<<{\Gamma'}&E\\
\Omega B * \Omega E\ar@{->}[r]^{1*\Omega\pi}\ar@{->}[u]^{\zeta}&\Omega
B*\Omega B\ar@{->}[u]^{\Gamma} 
}\\*[-14pt]
&\makebox[341pt]{}\qed
\end{align*}
\end{proposition}
Clearly \ref{prop3.12} follows from \ref{prop3.15}.

We next describe a simplification of the relative
Whitehead product in case $G=\Sigma A$.
\begin{proposition}\label{prop3.16}
Suppose $\alpha\colon \Sigma A\to B$ and $\delta\colon H\to E$. Then
the relative Whitehead product
\[
\xymatrix@C=38pt{
A\wedge H\simeq \Sigma A\circ H\ar@{->}[r]^-<<<<<<{\{\alpha,\delta\}_r}&E
}
\]
is represented by the composition
\[
\xymatrix@C=33pt{
A\wedge H\ar@{->}[r]^{\theta}&A\ltimes
H\ar@{->}[r]^->>>>>{\widetilde{\alpha}\ltimes\delta}&\Omega B\ltimes
E\ar@{->}[r]^-<<<<{\Gamma'}&E
}
\]
where $\theta$ is a right homotopy inverse to the projection\footnote{$\theta$
will depend on the co-$H$ structure of~$H$.}
which pinches
$H$ to a point.
\end{proposition}
\begin{proof}
To construct the map $\theta$ we need to generalize the
context in which the map $\zeta$ was defined in~\ref{eq3.9}. The
homotopy fiber of the map
\[
\xymatrix{
CX\cup X\times Y\ar@{->}[r]^-<<<<{\pi_2}&Y
}
\]
which pinches $CX$ to the basepoint is of the form
\[
CX\times \Omega Y\cup X\times PY\subset CX\times PY.
\]
Using the homotopy equivalence $\xi\colon (C(\Omega Y),\Omega Y)\to (PY,\Omega Y)$
(see \ref{prop2.6}), we get a homotopy equivalent fibration sequence
\begin{equation}\label{eq3.17}
\tag{3.17}\kern-60pt\llap{\protect\footnotemark}\hspace*{60pt}
\xymatrix@R=3pt{
X*\Omega
Y\ar@{->}[r]^->>>>{\zeta}&\parbox[b]{63pt}{\mbox{}\newline
\vspace*{4pt}
$\underset{\rotatebox{90}{\makebox[21pt]{$\simeq$}}}{CX\cup
X\times Y}$}\ar@{->}[r]^-<<<{\pi_2}&Y\\
&X\ltimes Y&
}
\end{equation}%
\footnotetext{Curiously there is also a cofibration sequence
\[
\xymatrix{
X*Y\ar@{->}[r]^{\zeta'}&X\ltimes \Sigma Y\ar@{->}[r]^{\pi_2}&\Sigma Y
}
\]
where $\zeta'$ is the composition 
\[
\xymatrix{
X*Y\ar@{->}[r]&X*\Omega \Sigma Y\ar@{->}[r]^{\zeta}&X\ltimes\Sigma Y.
}
\]}\index{zetaprime@$\zeta'$|LB}
Furthermore, the composition
\[
\xymatrix{
X*\Omega Y\ar@{->}[r]^-<<<<{\zeta}&X\ltimes
Y\ar@{->}[r]&X\wedge Y
}
\]
collapses $X\cup CX\times \Omega Y$ to a point, so there is a commutative
square
\[
\xymatrix{
X*\Omega Y\ar@{->}[r]^{\zeta}\ar@{->}[d]_{\simeq}&X\ltimes Y\ar@{->}[d]\\
X\wedge \Sigma\Omega Y\ar@{->}[r]^{1\wedge \epsilon}&X\wedge Y.
}
\]

The relative Whitehead product in \ref{prop3.16} is given by the
upper composition in the homotopy commutative diagram
\[
\kern-8pt\xymatrix@C=8pt{
&&&\Omega B*\Omega E\ar@{->}[r]^->>>>{\zeta}&\Omega B\!{}\times{}\! E\!{}\cup{}\!
PB\ar@{->}[r]^-<<{\Gamma'}&E\\
(\Sigma A)\!{}\circ{}\!H\ar@{->}[r]&\Sigma(\Omega\Sigma A\!{}\wedge{}\!\Omega
H)
&\kern-201pt\llap{\makebox[0pt]{$\!{}\simeq{}\!$}}\kern-201pt
&\Omega\Sigma A*\Omega H\ar@{->}[r]^->>{\zeta}\ar@{->}[u]^{\Omega
\alpha*\Omega \delta}&\Omega\Sigma A\!{}\times{}\!
H\!{}\cup{}\! P\Sigma A\ar@{->}[u]\\
A\!{}\wedge{}\!
H\ar@{->}[u]^-<<<{\!{}\simeq{}\!}\ar@{->}[r]^->>>>{1\!{}\wedge{}\!\nu}&\Sigma(A\!{}\wedge{}\!\Omega
H)\ar@{->}[u]
&\kern-201pt\llap{\makebox[0pt]{$\!{}\simeq{}\!$}}\kern-201pt
&A*\Omega H\ar@{->}[u]\ar@{->}[r]^->>>>{\zeta}&A\!{}\times{}\! H\!{}\cup{}\!
CA\ar@{->}[u]
&\kern-201pt\llap{\makebox[0pt]{$\!{}\simeq{}\!$}}\kern-201pt
&A\!{}\ltimes{}\!H
}
\]
where the lower composition is the map $\theta$. 
The right hand vertical map is the composition
\[
\xymatrix@C=38pt{
A\ltimes H\ar@{->}[r]^->>>>>>{\widetilde{\alpha}\ltimes \delta}&\Omega B\ltimes
E\simeq \Omega B\times E\cup PB.
}
\]
By the homotopy commutative square above,
$\theta$ has a right homotopy inverse since $\epsilon \nu\sim 1$,
and $\theta$ projects trivially to~$H$ since $\pi_2\zeta$ is null homotopic.
\end{proof}\addtocounter{Theorem}{1}
\begin{corollary}\label{cor3.18}
Suppose $\alpha\colon \Sigma A\to B$ and $\delta\colon H\to E$. Then for any
ring $R$, the homomorphism
\[
(\{\alpha,\delta\}_r)_*\colon H_*(A\wedge H;R)\to H_*(E;R)
\]
is given by the composition
\[
\xymatrix@C=38pt{
H_*(A\wedge H;R)\subset H_*(A\times
H;R)\ar@{->}[r]^-<<<<<{(\widetilde{\varphi \alpha}\times \delta)_*}&H_*(\Omega
X\times E;R)\ar@{->}[r]^-<<<<<{a_*}&H_*(E;R)
}
\]
\end{corollary}
\begin{proof}
Apply \ref{eq3.8} and \ref{prop3.16}.
\end{proof}

\section{Iterated Whitehead Products and the Decomposition
of $\Omega G*\Omega H$}\label{subsec3.3}
\markright{\ref{subsec3.3}.\hspace*{5pt}\uppercase{Iterated Whitehead Products and Decomposition
of} $\Omega G*\Omega H$}

We need, also, to discuss iterated Whitehead products. Suppose
\[
\alpha_i\colon G_i\to X\]
for $1\leqslant i\leqslant n$. We define the iterated
Whitehead
product
\[
\{\alpha_n,\alpha_{n-1},\dots,\alpha_1\}\colon
G_n\circ(G_{n-1}\circ\dots\circ G_1)\dots)\to X
\]
as $\{\alpha_n,\{\alpha_{n-1},\dots,\alpha_1\}\}$. In case $G_i=G$ for each
$i$ we define
\[
G^{[n]}=G\circ G^{[n-1]}.
\]
We also define $G^{[i]}H^{[j]}$\index{$G^{[i]}$|LB}\index{$G^{[i]}H^{[j]}$|LB} as $G\circ(G^{[i-1]}H^{[j]}_*)$ when $i>1$
and as
$G\circ H^{[j]}$ when $i=1$.

Suppose now that $\alpha\colon G\to G\vee H$ and $\beta\colon H\to G\vee H$
are
the inclusions. We then consider
\[
ad^i(\alpha)(\{\alpha,\beta\})=\{\alpha,\dots,\alpha, \beta\}\colon
G^{[i+1]}H\to G\vee H.
\]\index{$ad^i$|LB}

Given a principal fibration
\[
\xymatrix{
\Omega X\ar@{->}[r]&E\ar@{->}[r]&B
}
\]
and maps $\alpha_i\colon G\to B$, $\beta\colon H\to E$, we define
\[
\{\alpha_n,\dots, \alpha_1,\beta\}_r\colon
G_n\circ(G_{n-1}\circ\dots\circ(G_1\circ H)\dots)\to E
\]
as
\[
\{\alpha_n,\{\alpha_{n-1},\dots,\alpha_1,\beta\}_r\}_r.
\]

By an iterated application of \ref{prop3.11}(c), we
have\addtocounter{equation}{4}
\begin{equation}\label{eq3.19}
\pi\{\alpha_n,\dots,\alpha_1,\beta\}_r=\{\alpha_n,\dots,\alpha_1,\pi \beta\}.
\end{equation}

Now consider the principal fibration
\[
\xymatrix{
\Omega G\ar@{->}[r]^->>>>{i}&\Omega G\ltimes H\ar@{->}[r]^{\pi}&G\vee H.
}
\]\addtocounter{Theorem}{1}
Let $K=\bigvee\limits_{i\geqslant 1}G^{[i]}H$. Let $\beta\colon H\to \Omega G\ltimes H$ be the inclusion of the second factor.
\begin{proposition}\label{prop3.20}
The maps $ad_r^i\index{$ad_r^i$|LB}(\alpha)(\beta)\colon G^{[i]}H\to\Omega G\ltimes H$ define a
homotopy equivalence
\[
K\vee H\simeq\bigvee\limits_{i\geqslant 1} G^{[i]}H\vee H\to \Omega G\ltimes H.
\]
\end{proposition}
\begin{proof}
According to \cite[3a]{Gra11}, such a homotopy
equivalence exists where the maps
\[
\xi_i\colon G^{[i]}H\to \Omega G\ltimes H
\]
are chosen so that $\pi\xi_i\sim \{\alpha,\dots,\alpha,\pi\beta\}$. However
by (\ref{eq3.19}), $\pi\{\alpha,\dots,\alpha,\beta\}_r\sim
\{\alpha,\dots,\alpha,\pi\beta\}$; since the map $i\colon \Omega G\to \Omega G\ltimes H$
is null homotopic, $\xi_i\sim \{\alpha,\dots,\alpha,\beta\}_r$.
\end{proof}

We now consider the principal fibration
\[
\xymatrix{
\Omega G\times \Omega H\ar@{->}[r]&\Omega G*\Omega H\ar@{->}[r]&G\vee H.
}
\]
Using the map $\psi\colon G\circ H\to \Omega G*\Omega H$ we define iterated
relative Whitehead products:
\[
ad_r^{i,j}=ad_r^i(\beta)ad_r^j(\alpha)(\psi)\colon H^{[i]}G^{[j]}(G\circ
H)\to \Omega G*\Omega H
\]
\begin{proposition}\label{prop3.21}
The maps $ad_r^{i,j}$ for $i\geqslant 0$, $j\geqslant 0$ define
a homotopy equivalence
\[
\bigvee\limits_{\substack{i\geqslant 0\\
j\geqslant 0}} H^{[i]}G^{[j]}(G\circ H)\to \Omega G*\Omega H.
\]
\end{proposition}
\begin{proof}
Consider the diagram of principal fibrations
\[
\xymatrix{
\Omega H\ar@{->}[r]^->>>>>{\iota_2}\ar@{->}[d]&\Omega G\times \Omega H\ar@{->}[d]\\
\Omega G*\Omega H\ar@{=}[r]\ar@{->}[d]_{\zeta}&\Omega G*\Omega
H\ar@{->}[d]_{\pi}\\
\Omega G\ltimes H\ar@{->}[r]&G\vee H
}
\]
The maps $ad_r^j\colon G^{[j]}(G\circ H)\to \Omega G\ltimes H$ defined by
\ref{prop3.20} lift to $\Omega G*\Omega H$ since $\zeta$ has a right
homotopy inverse.
Since the liftings project by $\pi$ onto the maps:
\[
ad^j\colon G^{[j]}(G\circ H)\to G\vee H,
\] these
liftings are homotopic to the relative
Whitehead product defined by~$\pi$. However $\Omega G\ltimes H$ is
homotopy equivalent to $H\vee K$ so the maps
\[
ad_r^iH^{[i]}K\to \Omega G*\Omega H
\]
define a homotopy equivalence
\[
\bigvee\limits_{i\geqslant 0}H^{[i]}K\to \Omega G * \Omega H.
\]
Furthermore
\[
K=\bigvee\limits_{j\geqslant 0}G^{[j]}(G\circ H)
\]
and the relative
Whitehead
products defined by the left hand fibration are mapped to the
corresponding relative Whitehead products in the right
hand fibration. Thus we have
\[
\smash{\makebox[38.5pt]{}\xymatrix{\raisebox{-8pt}{$\smash{\bigvee\limits_{\substack{i\geqslant 0\\
j\geqslant 0}}H^{[i]}G^{[j]}(G\circ H)}$}\ar@{->}[r]^-<<<<{\simeq}&\raisebox{-7pt}{$\smash{\bigvee\limits_{i\geqslant
0}H^{[i]}K}$}\ar@{->}[r]^->>>>{\simeq}&\Omega G*\Omega H
}}
\makebox[38.5pt]{}\rlap{\qed}
\]\noqed
\end{proof}
\vspace*{8pt}
\begin{Theorem}\label{theor3.22}
Suppose
\[
\xymatrix{
\Omega X\ar@{->}[r]^{i}&E\ar@{->}[r]^{\pi}&G
}
\]
is a principal fibration induced by a map $\varphi\colon G\to X$
where $X$ is an $H$-space with strict unit. Suppose $\nu\colon E\to Z$.
Then the composition
\[
\xymatrix{
\Omega G*\Omega G\ar@{->}[r]^-<<<{\Gamma}&E\ar@{->}[r]^{\nu}&Z
}
\]
is null homotopic iff the compositions
\[
\nu ad_r^i(\alpha)(\{\alpha,\alpha\}_{\times})\colon G^{[[i+2]}\to E\to Z
\]
are null homotopic for each $i\geqslant 0$, where $\alpha\colon G\to G$ is
the
identity map.
\end{Theorem}
\begin{proof}
In this case $G=H$ and the map of principal
fibrations
\[
\xymatrix{
\Omega G\times \Omega G\ar@{->}[d]\ar@{->}[r]&\Omega X\ar@{->}[d]\\
\Omega G*\Omega G\ar@{->}[d]\ar@{->}[r]&E\ar@{->}[d]\\
G\vee G\ar@{->}[r]^{\nabla}&G
}
\]
maps $H^{[i]}G^{[j]}(G\circ H)$ to $G^{[i+j]}(G\circ G)$
which only depends on $i+j$.
\end{proof}

\section{Neisendorfer's Theory for Homotopy with Coefficients}\label{subsec3.4}

In the case that the co-$H$ spaces are Moore spaces,
the resulting Whitehead products occur in the
homotopy groups with coefficients. The adjoint
theory of Samelson products is due to Neisendorfer~\cite{Nei80},
and was crucial in the work of~\cite{CMN79a,CMN79b,CMN79c}.
This theory has been further developed
in~\cite{Nei10a} where $H$-space based Whitehead products
were introduced.

We need to make a mild generalization of this in that we
must consider the case where
\[
G=\Sigma P^m(p^r)\qquad H=\Sigma P^n(p^s)\qquad s\geqslant r.
\]
In this case
\[
G\circ H=\Sigma P^{m+n}(p^r)\vee \Sigma P^{m+n-1}(p^r).
\]
This splitting is not unique and we must choose a splitting.

Choose a map
\[
\Delta\index{Delta@$\Delta$|LB}\colon P^{m+n}(p^s)\to P^m(p^s)\wedge P^n(p^s)
\]
so that the diagram\addtocounter{equation}{3}
\begin{equation}\label{eq3.23}
\begin{split}
\xymatrix{
P^{m+n}(p^s)\ar@{->}[r]^-{\Delta}\ar@{->}[d]_{\pi_{m+n}}&P^m(p^s)\wedge P^n(p^s)\ar@{->}[d]\\
S^{m+n}\ar@{->}[r]^-{\simeq}&S^m\wedge S^n
}
\end{split}
\end{equation}
commutes up to homotopy. Such a choice is possible
when $m,n\geqslant 2$ for $p$ odd and is unique up to homotopy.

Neisendorfer \cite{Nei80} has produced internal
Whitehead and Samelson products for homotopy
with $Z/p^s$ coefficients. The Whitehead product of
$x\in \pi_{m+1}(X;Z/p^s)$ and $y\in \pi_{n+1}(X;Z/p^s)$ is an element
\[
[x,y]\in\pi_{m+n+1}(X;Z/p^s)
\]
defined as the homotopy
class of the composition:
\begin{align}\label{eq3.24}
P^{m+n+1}(p^s)&=\xymatrix{\Sigma
P^{m+n}(p^s)\ar@{->}[r]^->>>>{\enlarge{\Sigma\Delta}}&\Sigma
P^m(p^s)\wedge P^n(p^s)}\\
&=\xymatrix{P^{m+1}(p^s)\circ
P^{n+1}(p^s)\ar@{->}[r]^-<<<{\enlarge{\{x,y\}}}&X}\notag
\end{align}

As we will need to consider such pairings with different
coefficients, suppose $x\in \pi_{m+1}(X;Z/p^r)$ and
$y\in \pi_{n+1}(X;Z/p^{r+t})$. We can still form the external
Whitehead product:
\[
\Sigma P^m(p^r)\wedge P^n(p^{r+t})=\xymatrix{P^{m+1}(p^r)\circ
P^{n+1}(p^{r+t})\ar@{->}[r]^-<<<{\enlarge{\{x,y\}}}&X.}
\]

Since the map of degree $p^{r+t}$ on $P^m(p^r)$ is null
homotopic, there is a splitting:
\[
P^m(p^r)\wedge P^n(p^{r+t})\simeq P^{m+n}(p^r)\vee P^{m+n+1}(p^r).
\]

We now choose an explicit splitting. Recall (\ref{eq1.5new})
$\delta_t=\beta\rho^t$.\addtocounter{Theorem}{2}
\begin{proposition}\label{prop3.25}
There is a splitting of $P^m(p^r)\wedge P^n(p^{r+t})$
defined by the two compositions:
\begin{gather*}
 \xymatrix{P^{m+n}(p^r)\ar@{->}[r]^-{\enlarge{\Delta}}&P^m(p^r)\wedge
 P^n(p^r)\ar@{->}[r]^{\enlarge{1\wedge
\rho^t}}& P^m(p^r)\wedge P^n(p^{r+t})}\\
 \xymatrix{P^{m+n-1}(p^r)\ar@{->}[r]^-{\enlarge{\Delta}}&P^m(p^r)\wedge
 P^{n-1}(p^r)\ar@{->}[r]^-{\enlarge{1\wedge \delta_t}} &P^m(p^r)\wedge P^n(p^{r+t})}
\end{gather*}
\end{proposition}
\begin{proof}
$(1\wedge \pi_n)(1\wedge \rho^t)\Delta=(1\wedge \pi_n)\Delta$ induces a
$\text{mod}\, p$ homology isomorphism, so $(1\wedge \rho^t)\Delta$ induces
a homology monomorphism. The second composition
factors
\begin{multline*}
\xymatrix@C=35pt{P^{m+n-1}(p^r)\ar@{->}[r]^-{\enlarge{\Delta}}&P^m(p^r)\wedge
P^{n-1}(p^r)\ar@{->}[r]^-{\enlarge{1\wedge \pi_{n-1}}}&P^m(p^r)\wedge
S^{n-1}}\\
\xymatrix@C=35pt{\ar@{->}[r]^-{\enlarge{1\wedge \iota_{n-1}}}&P^m(p^r)\wedge
P^{n}(p^{r+t})}
\end{multline*}
and the composition of the first two maps is a homotopy
equivalence. Since the third map induces a
$\text{mod}\, p$ homology monomorphism, this
composition does as well. Counting ranks, we see that the
two maps together define a homotopy equivalence:\index{$e$|LB}
\[
\makebox[33pt]{}
\xymatrix{e\colon P^{m+n}(p^r)\vee
P^{m+n-1}(p^r)\ar@{->}[r]^-{\enlarge{\simeq}}&P^m(p^r)\wedge
P^n(p^{r+t})}
\makebox[33pt]{}\rlap{\qed}
\]\noqed
\end{proof}

We apply this to the internal Whitehead product
(\ref{eq3.24}) to get
\begin{proposition}\label{prop3.26}
\[
\makebox[21pt]{}
\{x,y\}e=[x,y \rho^t]\vee[x,y\delta_t]\colon P^{m+n}(p^r)\vee
P^{m+n-1}(p^r)\to X.
\makebox[21pt]{}\rlap{\qed}
\]
\end{proposition}

\ref{prop3.26} resolves the external Whitehead product with
different coefficients into internal Whitehead products
with coefficients in $Z/p^r$ as considered by Neisendorfer.

Suppose now that we are given a
principal fibration
\[
\Omega X\to E\to B
\]
classified by a map $\varphi\colon B\to X$ where $X$ is a homotopy
commutative $H$-space with strict unit and we are given
classes $u\in \pi_m(B;Z/p^r)$ and $v\in \pi_n(B;Z/p^{r+t})$. Then we
have
\begin{proposition}\label{prop3.27}
\[
\{u,v\}_{\times}e=[u,v\rho^t]_{\times}\vee[u,v\delta_t]_{\times}\colon
P^{m+n}(p^r)\vee P^{m+n-1}(p^{r+t})\to E.
\]
\end{proposition}
\begin{proof}
Both $\{u,v\}_{\times}e$ and
$[u,v\rho^t]_{\times}\vee[u,v\delta_t]_{\times}$ are the images under $\Gamma$ of maps:
\[
P^{m+n}(p^r)\vee P^{m+n-1}(p^r)\to \Sigma(\Omega B\wedge \Omega B)
\]
which are homotopic after projection
\[
\xymatrix{\Sigma(\Omega B\wedge \Omega
B)\ar@{->}[r]^-{\enlarge{\omega}}&B\vee B}
\]
by \ref{prop3.26}. Since $\Omega \omega$ has a left homotopy universe,
these
maps are homotopic. Composing with $\Gamma\colon \Sigma(\Omega B\wedge
\Omega B)\to E$
finishes the proof.
\end{proof}

Similar to \ref{prop3.27}, we  have
\begin{proposition}\label{prop3.28}
$\{x,u\}_re=[x,u\rho^t]_r\vee[x,u\delta_t]_r$ where
$x\in \pi_m(B;Z/p^r)$ and $u\in \pi_n(E;Z/p^{r+t})$.\qed
\end{proposition}

There is one special case of this that we will need in section~\ref{subsec6.3}.
This involves relative Whitehead products $[x,u]_r$
when $u\colon\!\! \xymatrix@C=10pt{S^n\ar@{->}[r]{}\!&E}$ and $x\colon\!\!
\xymatrix@C=10pt{P^m\ar@{->}[r]{}\!&\nobreak B}\!\!$. In this case
\[
[x,u]_r=\{x,u\}_r\colon \xymatrix{P^m\circ S^n\ar@{->}[r]&E}.
\]
\begin{proposition}\label{prop3.29}
$[x,u\pi_n]_r=[x,u]_r\colon \xymatrix{P^{m+n-1}\ar@{->}[r]&E}$.
\end{proposition}
\begin{proof}
$\{x,u\pi_n\}_re=[x,u\pi_n]_r\vee 0$
since $\pi_n\delta_t=0$. Consequently we have a homotopy
commutative diagram
\[
\xymatrix@R=50pt@C=35pt{
P^{m+n-1}\vee P^{m+n-2}\ar@{->}[r]^-e\ar@{->}[dr]_-{[x,u\pi_n]_r\vee
0\makebox[3pt]{}}&
P^m\circ P^n\ar@{->}[r]^-{1\circ \pi_n}\ar@{->}[d]^-<<<<{\{x,u\pi_n\}_r}&
P^m\circ S^n\ar@{->}[dl]^-{[x,u]_r}&
\hspace*{-28pt}\llap{\enlarge{\simeq{}}}
P^{m+n-1}\\
&E&
}
\]
where the upper composition is homotopic to projection onto the
first factor.
\end{proof}

Suppose then we are given a principal fibration
\[
\xymatrix{
\Omega X\ar@{->}[r]^-{\enlarge{i}}&
E\ar@{->}[r]^-{\enlarge{\pi}}&
B
}
\]
induced by a map $\varphi\colon B\to X$ where $X$ is a homotopy
commutative $H$-space with strict unit. Suppose we are
given classes
\[
\alpha\in\pi_m(B;Z/p^r),\ \beta\in\pi_n(B;Z/p^r),\ \gamma\in
\pi_k(E;Z/p^r),\ \delta\in \pi_{\ell}(E;Z/p^r).
\]
Recall that by using the map $\Delta$  we define the internal
$H$-space based
Whitehead product
\[
[\alpha,\beta]_{\times}=\{\alpha,\beta\}_{\times}\Delta\in \pi_{m+n-1}(E;Z/p^r)
\]
and internal relative Whitehead product
\[
[\alpha,\gamma]_r=\{\alpha,\gamma\}_r\Delta\in \pi_{m+k-1}(E;Z/p^r).
\]
These are related as in~\ref{prop3.6}, \ref{prop3.11} and \ref{prop3.12}.
\begin{proposition}\label{prop3.30}\begin{enumerate}
\item[\textup{(a)}]$\pi_*[\alpha,\beta]_{\times}=[\alpha,\beta]\in\pi_{m+n-1}(B;Z/p^r)$
\item[\textup{(b)}]$[\pi_*\gamma,\pi_*\delta]_{\times}=[\gamma,\delta]\in\pi_{k+\ell-1}(E;Z/p^r)$
\item[\textup{(c)}]$[\alpha,\delta]_r=[\alpha,\pi_*\delta]_{\times}\in\pi_{m+\ell-1}(E;Z/p^r)$
\item[\textup{(d)}]$\pi_*[\alpha,\delta]_r=[\alpha,\pi_*\delta]\in\pi_{m+\ell-1}(B;Z/p^r)$
\item[\textup{(e)}]$[\pi_*\gamma,\delta]_r=[\gamma,\delta]\in\pi_{k+\ell-1}(E;Z/p^r)$
\end{enumerate}
\end{proposition}
According to Neisendorfer \cite{Nei10a}, we also have
standard Whitehead product formulas:
\begin{proposition}\label{prop3.31}
The following identities hold:

\makebox[16pt]{\textup{(a)\hfill}}\hspace*{0.5em}$[\alpha,\beta]_{\times}=-(-1)^{(m+1)(n+1)}[\beta,\alpha]_{\times}$

\makebox[16pt]{\textup{(b)\hfill}}\hspace*{0.5em}$[\alpha_1+\alpha_2,\beta]_{\times}=[\alpha_1,\beta]_{\times}+[\alpha_2,\beta]_{\times}$

\makebox[16pt]{\textup{(c)\hfill}}\hspace*{0.5em}$[\alpha,[\beta,\eta]]_{\times}=[[\alpha,\beta],\eta]_{\times}+(-1)^{(m+1)(n+1)}[\beta,[\alpha,\eta]]_{\times}$\\
for $\eta\in\pi_j(B;Z/p^r)$

\makebox[16pt]{\textup{(c$'$)\hfill}}\hspace*{0.5em}$[\alpha,[\beta,\gamma]_r]=[[\alpha,\beta],\gamma]_r+(-1)^{(m+1)(n+1)}[\beta,[\alpha,\gamma]_r]_r$

\makebox[16pt]{\textup{(d)\hfill}}\hspace*{0.5em}$\beta^{(r)}[\alpha,\beta]_{\times}=[\beta^{(r)}\alpha,\beta]_{\times}+(-1)^{m+1}[\alpha,\beta^{(r)}\beta]_{\times}$
where $\beta^{(r)}$ is the Bockstein associated with the
composition $P^k(p^r)\to P^{k+1}(p^r)$ for appropriate~$k$.
\end{proposition}
\begin{proof}
See \cite{Nei10a}. Neisendorfer considers the adjoint
Samelson products, so there is a dimension shift.
\end{proof}

\chapter{Index $p$ approximation}\label{chap4}

The goal of this chapter is to replace the co-$H$ spaces
$G_k^{[i]}$ from \ref{theor3.22} by a finite wedge of Moore spaces in case $n>1$. The iterated
Whitehead products involving $G_k$ are then replaced by
iterated Whitehead products in $\bmod\, p^s$ homotopy, which are more manageable. In \ref{subsec4.1}, we construct certain $\bmod\, p^{r+i-1}$
homotopy classes $a(i)$ and $c(i)$ for $i\leqslant k$. This is a
refinement of a similar construction in~\cite{GT10}, and leads
to a ladder of cofibration sequences. In~\ref{subsec4.2}, we
construct new co-$H$ spaces $L_k$ when $n>1$, and introduce
index~$p$ approximation. 
Using this we exploit the fact (\cite{The08}) that the
identity map of~$BW_n$ has order~$p$ to reduce the size of the set of 
obstructions. This allows for the replacement
of the iterated relative and $H$-space
based Whitehead products based on~$G_k$ with iterated relative and $H$-space based
Whitehead products in the $\bmod\, p^s$ homotopy groups for $r\leqslant
s\leqslant r+k$. The case $n=1$ is simpler and we
show that $T$ is homotopy-Abelian in the appendix.
Nevertheless, the constructions in Chapters \ref{chap4}, \ref{chap5} and \ref{chap6} will be used in Chapter~\ref{chap7} in case $n=1$ as well.

\section{Construction of the co-$H$ Ladder}\label{subsec4.1}
In this section we will assume an arbitrary $H$-space structure on the Anick space
as given in~\cite{GT10} and use its existence to develop certain
maps $a(k)$\index{$a(k)$|LB} and $c(k)$ for $k\geqslant 1$. We begin with a strengthening
of~\cite[4.3(d)]{GT10}.
\begin{proposition}\label{prop4.1}
There is a map
\[
e\colon P^{2np^k}(p^{r+k-1})\vee P^{2np^k+1}(p^{r+k-1})\to \Sigma T
\]
which induces an epimorphism in $\bmod\, p$ homology in
dimensions $2np^k$ and $2np^k+1$. Furthermore the composition of~$e$
with the map
\[
\xymatrix{
\Sigma T \ar@{->}[r]^->>>{\widetilde{E}}&S^{2n+1}\{p^r\}
}
\]
is null homotopic, where $\widetilde{E}$ is adjoint to
the map $E$ of~\ref{theor2.14}(e).
\end{proposition}
\begin{proof}
Recall by \ref{theor2.14}(g)
\[
H_*(T)\simeq Z/p[v]\otimes\Lambda(u)
\]
where $|v|=2n$, $|u|=2n-1$ and $\beta^{(r+i)}(v^{p^i})=uv^{p^i-1}$.

Using some $H$-space
structure map $\mu$ we consider the Hopf
construction:
\[
H(\mu)\colon\Sigma (T\wedge T)\to \Sigma T.
\]
Note that in homology
\[
\left(H(\mu)\right)_*(\sigma\otimes x\otimes y)=\sigma\otimes \mu_*(x\otimes y)
\]
if $|x|>0$ and $|y|>0$. We now define homology classes
\begin{align*}
\alpha&\in H_{2np^k+1}(\Sigma (T\wedge T);Z/p)\\
\beta&\in H_{2np^k}(\Sigma (T\wedge T);Z/p)
\end{align*}
by the formulas
\begin{align*}
\alpha&= -\sigma\otimes v^{p^{k-1}}\otimes v^{p^{k-1}(p-1)}\\
\beta&=\sigma\otimes v^{p^{k-1}}\otimes uv^{p^{k-1}(p-1)-1}
\end{align*}
\begin{align*}
\intertext{so we have}\\[-24pt]
\left(H(\mu)\right)_*\alpha&=-\sigma\otimes v^{p^k}\\
\left(H(\mu)\right)_*\beta&=\sigma\otimes uv^{p^k-1}.
\end{align*}
Also $\beta^{(r+k-1)}(\alpha)$ and $\beta^{(r+k-1)}(\beta)$ are both
nonzero. 
By \ref{theor2.14}(d),
$\Sigma (T\wedge T)$ is a wedge
of Moore
spaces; consequently there are maps
\begin{align*}
&a\colon P^{2np^k+1}(p^{r+k-1})\to \Sigma(T\wedge T)\\
&b\colon P^{2np^k}(p^{r+k-1})\to \Sigma(T\wedge T)
\end{align*}
such that $\alpha$ is in the image of $a_*$ and $\beta$ is in the
image of~$b_*$. Combining these we get a map $e$
\[
\xymatrix{P^{2np^k+1}(p^{r+k-1})\vee P^{2np^k}(p^{r+k-1})\ar@{->}[r]^-{\enlarge{a\vee
b}}&\Sigma(\Gamma\wedge T)\ar@{->}[r]^-{\enlarge{H(\mu)}}&\Sigma T}
\]
such that $\sigma\otimes v^{p^k}$ and $\sigma\otimes uv^{p^k-1}$ are in the image
of~$e_*$.
From this we see that there is a homotopy commutative
diagram\addtocounter{equation}{1}
\begin{equation}\label{eq4.2}
\begin{split}
\xymatrix{
P^{2np^k}(p^{r+k-1})\vee
P^{2np^k+1}(p^{r+k-1})\ar@{->}[r]^-{e}\ar@{->}[d]_{b\vee a}&\Sigma
T^{2np^k}\ar@{->}[d]\\
\Sigma(T\wedge T)\ar@{->}[r]^-{H(\mu)}&\Sigma T
}
\end{split}
\end{equation}
where $e$ induces an epimorphism in $\text{mod}\, p$ homology
in dimensions $2np^k$ and $2np^k+1$.
It remains to show that the composition
\[
\xymatrix{
\Sigma(T\wedge T)\ar@{->}[r]^-{\enlarge{H(\mu)}}&
\Sigma
T\ar@{->}[r]^-{\enlarge{\widetilde{E}}}&
S^{2n+1}\{p^r\}
}
\]
is null homotopic.

However, since
 $E\colon T\to \Omega S^{2n+1}\{p^r\}$ is an $H$ map by \ref{theor2.14}(e), there is a homotopy commutative diagram:
\[
\xymatrix@C=45pt{
\Sigma(T\wedge T)\ar@{->}[r]^->>>>>>>{\Sigma(E\wedge
E)}\ar@{->}[d]_{H(\mu)}&\Sigma(\Omega S^{2n+1}\{p^r\}\wedge
\Omega S^{2n+1}\{p^r\})\ar@{->}[d]^{H(\mu')}\\
\Sigma T\ar@{->}[r]^->>>>>>>>>>>{\Sigma E}&\Sigma\Omega S^{2n+1}\{p^r\}
.}
\]
where $\mu'$ is the loop space structure map on $\Omega S^{2n+1}\{p^r\}$.
Since $\Omega S^{2n+1}\{p^r\}$ is a loop space, the right hand
map is part of the classifying space structure
\[
\xymatrix{
\Sigma(\Omega S^{2n+1}\{p^r\}\wedge\Omega
S^{2n+1}\{p^r\})\ar@{->}[d]_{H(\mu')}\ar@{->}[r]&\dots\ar@{->}[r]&E^{\infty}\ar@{->}[d]\\
\Sigma\Omega S^{2n+1}\{p^r\}\ar@{->}[r]&\dots\ar@{->}[r]&S^{2n+1}\{p^r\}
}
\]
where $E_{\infty}$ is contractible and the bottom horizontal map is
the evaluation map. The result follows since $\widetilde{E}$ is the
composition:
\[
\makebox[64pt]{}
\xymatrix{
\Sigma T\ar@{->}[r]&
\Sigma\Omega S^{2n+1}\{p^r\}\ar@{->}[r]^-{\enlarge{ev}}&
S^{2n+1}\{p^r\}.
\makebox[64pt]{}
\rlap{\qed}
}
\]\noqed
\end{proof}

Now recall from \ref{theor2.14}(a) the maps
\[
\xymatrix{
T\ar@{->}[r]^{g}&\Omega G\ar@{->}[r]^{h}&T
}
\]
with $hg\sim 1$. Restricting we get
$\xymatrix{T^{2np^k}\ar@{->}[r]^{\enlarge{g_k}}
&\Omega G_k}$.\index{$g_k$|LB} Let
$\varphi_{k}$ be
the restriction of $\varphi$ to $G_k$. Then we have
\[
\xymatrix{
&T
\ar@{->}[d]_{g}
&T^{2np^k}
\ar@{->}[l]
\ar@{->}[d]^{g_*}\\
\Omega G
\ar@{=}[r]
\ar@{->}[d]_{h}
&\Omega G
\ar@{->}[d]_{\Omega\varphi}
&\Omega G_k
\ar@{->}[l]
\ar@{->}[d]_{\Omega\varphi_k}\\
T
\ar@{->}[r]^->>>>>{E}
&\Omega S^{2n+1}\{p^r\}
\ar@{=}[r]
&\Omega S^{2n+1}\{p^r\}
}
\]
where the left hand square commutes up to homotopy by \ref{theor2.14}(e).
Since $hg=1$, we get the homotopy commutative square:
\[
\xymatrix{
T^{2np^k}\ar@{->}[r]^{g_k}\ar@{->}[d]&\Omega G_k\ar@{->}[d]_{\Omega \varphi_k}\\
T\ar@{->}[r]^->>>>>{E}&\Omega S^{2n+1}\{p^r\}.
}
\]
Since $e$ factors through $\Sigma T^{2np^k}$, we combine this with
\ref{prop4.1} to see that the central composition
\[
\xymatrix{
&&E_k\ar@{->}[d]\\
P^{2np^k}(p^{r+k-1})\vee
P^{2np^k+1}(p^{r+k-1})\ar@{-->}[rru]\ar@{->}[r]\ar@{->}[dr]^{e}&\Sigma
T^{2np^k}\ar@{->}[r]\ar@{->}[d]&G_k\ar@{->}[d]^{\varphi_k}\\
&\Sigma T\ar@{->}[r]^->>>>>{\widetilde{E}}&S^{2n+1}\{p^r\}
}
\]
factors through $E_k$. We state this 
as\addtocounter{Theorem}{1}
\begin{proposition}\label{prop4.3}
For any $H$-space structure on~$T$ with corresponding maps $h$ and~$g$, there is a lifting
 of $\widetilde{g}_ke$ to~$E_k$
\[
\xymatrix@C=50pt{
P^{2np^k}(p^{r+k-1})\vee P^{2np^k+1}(p^{r+k-1})\ar@{->}[r]^-{a(k)\vee
c(k)}\ar@{->}[d]_e&E_k\ar@{->}[d]_{\pi_k}\\
\Sigma T^{2np^k}\ar@{->}[r]^-{\widetilde{g}_k}&G_k
.}
\]\index{$c(k)$|LB}
\end{proposition}

In the diagram below, the left column is a standard cofibration
sequence (\ref{eq1.9}) and the right column is a fibration sequence
defined by the projection~$\pi$:
\[
\xymatrix@C=50pt{
P^{2np^k}(p^{r+k})
\ar@{->}[d]_{p^{r+k-1}}
\ar@{->}[rr]^{\theta_1}
&&
\Omega P^{2np^k\!{}+{}\!1}(p^{r+k}) 
\ar@{->}[d]
\\
P^{2np^k}(p^{r+k})
\ar@{->}[rr]^{\theta_2}
\ar@{->}[d]_{\sigma\vee \sigma \beta}
&&
J
\ar@{->}[d]
\\ 
P^{2np^k}(p^{r+k-1})\!{}\vee{}\! P^{2np^k+1}(p^{r+k-1})
\ar@{->}[d]_{-\delta_1\vee \rho}
\ar@{->}[r]^-{a(k)\vee c(k)}
&
E_k
\ar@{->}[r]^-{\pi_k}
&
G_k
\ar@{->}[d]_{\pi}
\\
P^{2np^k+1}(p^{r+k})
\ar@{=}[rr]
&&
P^{2np^k\!{}+{}\!1}(p^{r+k}).
}
\]
The homological properties of~$e$ and \ref{prop4.3} imply that the
bottom region commutes up to homotopy since
$\widetilde{g}_k$ has a right homotopy inverse.
The maps $\theta_1$ and $\theta_2$ are induced from this region in the standard way.
For dimensional reasons $\theta_2$ factors through $G_{k-1}\subset J$
and since
\[
\xymatrix{
E_{k-1}\ar@{->}[r]\ar@{->}[d]&G_{k-1}\ar@{->}[d]\\  
E_k\ar@{->}[r]&G_k
}
\]
is a pullback diagram, $\theta_2$ factors through 
$E_{k-1}$. $\theta_1$ factors through 
\[
P^{2np^k-1}(p^{r+k})
\]
also
for dimensional reasons. We obtain
\begin{Theorem}\label{theor4.4}
There is a homotopy commutative ladder of
cofibrations:\footnote{In order to keep the notation from being too
cumbersome we will sometimes write $\beta_k$,
 $a(k)$ and $c(k)$
for the composition $\pi_{k-1}\beta_k$, 
$\pi_ka(k)$, and $\pi_kc(k)$ if it will not lead to
confusion.}\index{betak@$\beta_k$|LB}
\[
\kern-4pt
\xymatrix@C=49pt{
P^{2np^k}(p^{r+k})
\ar@{->}[d]_{p^{r+k-1}}
\ar@{=}[rr]
&&
P^{2np^k}(p^{r+k})
\ar@{->}[d]_{\alpha_k}\\
P^{2np^k}(p^{r+k})
\ar@{->}[r]^-{\beta_k}
\ar@{->}[d]_{\sigma\vee \sigma \beta}
&
E_{k-1}
\ar@{->}[r]^-{\pi_{k-1}}
\ar@{->}[d]
&
G_{k-1}
\ar@{->}[d]\\
P^{2np^k}(p^{r+k-1})\!{}\vee{}\! P^{2np^k+1}(p^{r+k-1})
\ar@{->}[r]^-{a(k)\vee c(k)}
\ar@{->}[d]_{-\delta_1\vee \rho}
&
E_k
\ar@{->}[r]^-{\pi_k}
&
G_k
\ar@{->}[d]_{\pi}\\
P^{2np^k+1}(p^{r+k})
\ar@{=}[rr]
&&
P^{2np^k+1}(p^{r+k})
.}
\]
Furthermore, for any choice of retraction $\nu_{k-1}\colon E_{k-1}\to
BW_n$, we can 
construct $a(k)$ and~$c(k)$
 so that
$\nu_{k-1}\beta_k\sim *$. 
\end{Theorem}
\begin{proof}
We need only demonstrate the last statement.
Suppose we are given a map $\overline{\beta}_k\colon P^{2np^k}(p^{r+k})\to
E_{k-1}$ so that the
diagram commutes up  to homotopy. Given a retraction
$\nu_{k-1}\colon \xymatrix{E_{k-1}\ar@{->}[r]&BW_n}$, we get a splitting
\[
\Omega E_{k-1}\simeq \Omega \overline{R}_{k-1}\times W_n
\]
by \ref{theor2.14}(j) where $\overline{R}_{k-1}$ is the fiber of $\nu_{k-1}$. We can then write
$\beta_k=\overline{\beta}_k-\epsilon$
where $\epsilon$ is the component of $\overline{\beta}_k$ that factors
through $W_n$ and 
$\beta_k$ factors through
$R_{k-1}$. Since each map $P^{2np^k}(p^{r+k})\to W_n$ has order $p$,
$\epsilon$
has order $p$ and thus $p^{r+k-1}\overline{\beta}_k=p^{r+k-1}\beta_k$ as
$r+k-1\geqslant 1$. Thus
the upper region commutes up to homotopy when $\overline{\beta}_k$
is replaced by $\beta_k$. Since $p^{r+k-1}\epsilon=0$, $\epsilon$ factors
\[
\xymatrix@C=32pt{
P^{2np^k}(p^{r+k})\ar@{->}[r]^-{\sigma\vee\sigma\beta}&
P^{2np^k}(p^{r+k-1})\vee
P^{2np^k+1}(p^{r+k-1})\ar@{->}[r]^-{\epsilon'}&
E_{k-1}
}
\]
as the left hand column is a cofibration sequence. We
now redefine $a(k)$ and $c(k)$ by subtracting off the
appropriate components of $\epsilon'$
and the middle region
now commutes up to homotopy. Since this alteration
of $a(k)$ and $c(k)$ factors through $E_{k-1}$, the projections to
$G_k$ vanish when projected to $P^{2np^k}(p^{r+k})$ so the bottom region also commutes up to homotopy.
\end{proof}

During the inductive construction of $\nu_k$ we will
be assuming that $\nu_i$ is defined for $i<k$ and
the alterations in \ref{theor4.4} have been made so that
$\beta_k$ is in the kernel of~$\nu_{k-1}$.\index{nupm@$\nu_p(m)$|LB}

\section{Index $p$ Approximation}\label{subsec4.2}

The goal of this section is to replace the co-$H$
space  $G_k$ by a sequence of approximations. The end
result will be to replace $G_k^{[i]}$ by a wedge of $\text{mod}\,p^s$
Moore spaces for $r\leqslant s\leqslant r+k$. We begin with a cofibration
sequence based on the ladder~\ref{theor4.4}. Throughout this
section we will exclude the case~$n=1$. The case $n=1$ is dealt with in the
appendix.
\begin{proposition}\label{prop4.5} 
For $k\geqslant 1$ there is a cofibration sequence
\[
\xymatrix{
P^{2np^k}(p^{r+k})\ar@{->}[r]^-{\xi_k}&L_k
\ar@{->}[r]^-{\zeta_k}&G_k\ar@{->}[r]^-{\pi'}
&P^{2np^k+1}(p^{r+k})
}
\]\index{$L_k$|LB}\index{piprime@$\pi'$|LB}where $L_k=G_{k-1}\vee P^{2np^k}(p^{r+k-1})\vee P^{2np^k+1}(p^{r+k-1})$,
$\zeta _k$ is induced
by the inclusion of~$G_{k-1}$ and the maps $\pi_ka(k)$
and $\pi_kc(k)$, and $\pi'=p^{r+k-1}\pi$.
\end{proposition}
\begin{proof}
This is a standard consequence of a ladder in
which each third rail is an equivalence (as in the usual proof of the
Mayer--Vietoris sequence).
\end{proof}
\begin{proposition}\label{prop4.6}
If $n>1$, there is a unique co-$H$ 
space
structure on~$L_k$ so that the cofibration in~\ref{prop4.5} is
a cofibration of co-$H$  maps.
\end{proposition}
\begin{proof}
Let $P$ be the pullback in the diagram:
\[
\xymatrix{
P\ar@{->}[r]\ar@{->}[d]&G_k\vee G_k\ar@{->}[d]\\
L_k\times L_k\ar@{->}[r]& G_k\times G_k.
}
\]
There is a map $\eta\colon L_k\to P$ which projects to the diagonal
map on $L_k$ and the composition:
\[
\xymatrix{
L_k\ar@{->}[r]^-{\zeta _k}&G_k\ar@{->}[r]&G_k\vee G_k.
}
\]
We first assert that $\eta$ is unique up to homotopy. Since $L_k$ is a
wedge of Moore spaces, it suffices to show that if
$\epsilon\colon L_k\to P$ projects trivially to~$G_k\vee G_k$ and
$L_k\times L_k$, it is
itself trivial. Now the homotopy fiber of the
map $P\to L_k\times L_k$ is the same as
the homotopy fiber of~$G_k\vee G_k\to G_k\times G_k$, i.e., $\Sigma(\Omega
G_k\wedge \Omega G_k)$;
we conclude that $\epsilon$ must factor through \mbox{$\Sigma(\Omega G_k\wedge
\Omega G_k)$},
and that the composition
\[
\xymatrix{
L_k\ar@{->}[r]^-{\epsilon'}&\Sigma(\Omega G_k\wedge \Omega G_k)\ar@{->}[r]& G_k\vee G_k
}
\]
is null homotopic. This
implies that $\epsilon'$ is null homotopic and hence $\epsilon$
is as well since $\Omega(G_k\vee G_k)\to \Omega (G_k\times G_k)$ has a
right
homotopy inverse.

The map $\zeta_k$ is a $2np^k-1$ equivalence since
$P^{2np^k+1}(p^{r+k})$
is $2np^k-1$ connected. Since $L_k$ and $G_k$ are both $2n-1$
connected, this implies that the composition
\[
\xymatrix{
\Sigma(\Omega L_k\wedge \Omega L_k)\ar@{->}[r]&
\Sigma(\Omega L_k\wedge \Omega G_k)\ar@{->}[r]&
\Sigma(\Omega G_k\wedge \Omega G_k)
}
\]
is a $2np^k+2n-2$ equivalence. Now consider the
diagram of vertical fibrations:
\[
\xymatrix{
\Sigma(\Omega L_k\wedge \Omega L_k)\ar@{->}[r]\ar@{->}[d]&
\Sigma(\Omega G_k\wedge \Omega G_k)\ar@{->}[r]^-{\simeq}\ar@{->}[d]&
\Sigma(\Omega G_k\wedge \Omega G_k)\ar@{->}[d]\\
L_k\vee L_k\ar@{->}[r]\ar@{->}[d]&P\ar@{->}[r]\ar@{->}[d]&G_k\vee G_k\ar@{->}[d]\\
L_k\times L_k\ar@{=}[r]&L_k\times L_k\ar@{->}[r]&G_k\times G_k.
}
\]
From this we see that the map $L_k\vee L_k\to P$ is a $2np^k+2n-2$
equivalence. Since $n>1$, that implies that there is a unique
lifting of $\eta\colon L_k\to P$ to $L_k\vee L_k$, which defines a co-$H$ 
space structure on~$L_k$ such that $\zeta_k$ is a co-$H$ 
map.

Similarly, we observe that the $2np^k+2n-3$ skeleton of
the fiber of the map $L_k\vee L_k\to G_k\vee G_k$ is
$P^{2np^k}(p^{r+k})\vee P^{2np^k}(p^{r+k})$,
so the composition
\[
\xymatrix{
P^{2np^k}(p^{r+k})\ar@{->}[r]^-{\xi_k}&L_k\ar@{->}[r]&L_k\vee L_k
}
\]
factors through $P^{2np^k}(p^{r+k})\vee P^{2np^k}(p^{r+k})$ and such a
factorization defines a co-$H$ space structure on $P^{2np^k}(p^{r+k})$.
That structure, of course, is unique. So this $\xi_k$ is a co-$H$ 
map with the suspension structure.
\end{proof}

Warning: $L_k$ does not
split as a co-product of co-$H$ spaces. In particular,
the inclusion $P^{2np^k+1}(p^{r+k-1})\to L_k$ is not a co-$H$ map.
If it were, the map
\[
\xymatrix{
P^{2np^k+1}(p^{r+k-1})\ar@{->}[r]^-{c(k)}&E_k\ar@{->}[r]^-{\pi_k}&G_k 
}
\]
would be a co-$H$ map, contradicting~\cite[2.2]{AG95}.

Write $[k]\colon \Sigma X\to \Sigma X$ for the $k$-fold sum of the 
identity map.
\begin{definition}\label{def4.7}
Suppose
$\xymatrix{G\ar@{->}[r]^{\enlarge{f}}&H\ar@{->}[r]^{\enlarge{g}}&\Sigma K}$ is a cofibration
sequence of co-$H$ spaces and co-$H$ maps. We will say that $f$ is an index
$p$ approximation if there is a co-$H$ map
$g'\colon \Sigma H\to \Sigma^2 K$ such that $\Sigma g$ factors
\[
\xymatrix{
\Sigma H\ar@{->}[r]^-{g'}&\Sigma^2K\ar@{->}[r]^-{[p]}&\Sigma^2K\\
}
\]
up to homotopy. $f\!\colon G\to H$ will be
called an iterated index $p$ approximation if $f$ is
homotopic to a composition
\[
G=G_0\to G_1\to\dots\to G_m=H
\]
where each map $G_i\to G_{i+1}$ is an index $p$ approximation.
\end{definition}

Thus, for example, $\zeta_k\colon L_k\to G_k$ is an index
$p$ approximation.
\begin{proposition}\label{prop4.8}
Suppose that $f\!\colon G\to H$ is an iterated
index $p$ approximation and $\nu\colon H\to BW_n$. Then
$\nu$ is null homotopic iff $\nu\! f$ is null homotopic.
\end{proposition}
\begin{proof}
We will only consider the case when $f$ is an
index $p$ approximation, as the general result follows 
by an easy induction.
Suppose then that $f\!\colon G\to H$ is an index $p$ approximation and
$\Sigma g$ factors up to homotopy:
\[
\xymatrix{
\Sigma H\ar@{->}[r]^-{g'}&\Sigma^2 K\ar@{->}[r]^-{[p]}&\Sigma^2K.
}
\]
Assume that $\nu\! f$ is null homotopic, so we can factor $\nu$ as
\[
\xymatrix{
H\ar@{->}[r]^-{g}&\Sigma K\ar@{->}[r]^-{\nu'}&BW_n.
}
\]
Consider the diagram:
\[
\xymatrix{
\Omega \Sigma^2K\ar@{->}[r]^-{\Omega[p]}&
\Omega \Sigma^2K\ar@{->}[r]^-{\Omega\Sigma\nu'}&
\Omega \Sigma BW_n\ar@{->}[r]&
BW_n\\
H\ar@{->}[r]^-{g}\ar@{->}[u]^{\widetilde{g}'}&\Sigma
K\ar@{->}[r]^-{\nu'}\ar@{->}[u]&BW_n\ar@{=}[ur]\ar@{->}[u]  
}
\]
Since $BW_n$ is homotopy associative (\cite{Gra88}), the
upper composition
is an $H$-map. This composition is thus inessential if its
restriction to $\Sigma K$ is inessential. However this restriction
factors through $[p]\colon\Sigma K\to \Sigma K$. Since $BW_n$ has
$H$-space exponent $p$ (\cite{The08}), we conclude that
the upper composition is inessential. This $\nu\sim \nu'g$ is
inessential as well.
\end{proof}
\begin{lemma}\label{lem4.9}
There is a homotopy commutative diagram
\[
\xymatrix@R=5pt{
(\Sigma^2H)\circ K
\ar@{->}[r]^-{[p]\circ 1}
\ar@{}[d]^{\rotatebox{90}{$\simeq$}}
&(\Sigma^2H)\circ K
\ar@{}[d]^{\rotatebox{90}{$\simeq$}}
\\
(\Sigma H)\wedge K
\ar@{->}[r]^-{[p]\wedge 1}
\ar@{}[d]^{\rotatebox{90}{$\simeq$}}
&(\Sigma H)\wedge K
\ar@{}[d]^{\rotatebox{90}{$\simeq$}}
\\
\Sigma^2(H\circ K)
\ar@{->}[r]^-{[p]}
&\Sigma^2(H\circ K)
}
\]
where the equivalences are co-$H$ equivalences.
\end{lemma}
\begin{proof}
The vertical equivalences follow from~\ref{prop3.2}. These
equivalences are natural for co-$H$ maps. However
$[p]\colon \Sigma H\to \Sigma H$ is a co-$H$ map since $H$ is a co-$H$
space.
\end{proof}
\begin{lemma}\label{lem4.10}
Suppose
$\xymatrix{G_1\ar@{->}[r]^-{\enlarge{\alpha}}&G_2\ar@{->}[r]^-{\enlarge{\beta}}&G_3}$ is a cofibration
sequence
of co-$H$  spaces and co-$H$  maps. Then for each co-$H$  space $H$,
\[
\xymatrix@R=12pt{
G_1\circ H
\ar@{->}[r]^-{\alpha\circ 1}
&
G_2\circ H
\ar@{->}[r]^-{\beta\circ 1}
&
G_3\circ H\\
H\circ G_1
\ar@{->}[r]^-{1\circ \alpha}
&
H\circ G_2
\ar@{->}[r]^-{1\circ \beta}
&
H\circ G_3
}
\]
are both cofibration sequences.
\end{lemma}
\begin{proof}
In the extended cofibration sequence, the composition of two
adjacent maps is null homotopic
\[
\xymatrix{
G_1\ar@{->}[r]^-{\alpha}
&G_2\ar@{->}[r]^-{\beta}
&G_3\ar@{->}[r]^-{\gamma}
&\Sigma G_1\ar@{->}[r]^-{\Sigma\alpha}
&\Sigma G_2\ar@{->}[r]^-{\Sigma\beta}
&\Sigma G_3\ar@{->}[r]&\dots
}
\]
and all maps are co-$H$  maps. It follows that the same
is true for the sequence:
\[
\xymatrix{
G_1\circ H
\ar@{->}[r]^-{\alpha\circ 1}
&
G_2\circ H
\ar@{->}[r]^-{\beta\circ 1}
&
G_3\circ H
\ar@{->}[r]^-{\gamma\circ 1}&
(\Sigma G_1)\circ H\ar@{->}[r]&\dots 
}
\]
where $\Sigma (G_1\circ H)\simeq G_1\wedge H\simeq (\Sigma G_1)\circ H$. Since
this 
sequence also induces an exact sequence in homology
it is a cofibration sequence. The other case is similar.
\end{proof}
\begin{proposition}\label{prop4.11}
If $f\!\colon G\to H$ is an index $p$ approximation,
the maps $f\circ 1\colon G\circ L\to H\circ L$ and $1\circ f\!\colon
L\circ G\to L\circ H$ are index
$p$ approximations as well.
\end{proposition}
\begin{proof}
Factor $g$ as
\[
\xymatrix{
\Sigma H\ar@{->}[r]^-{g'}&\Sigma^2K\ar@{->}[r]^-{[p]}&\Sigma^2K
}
\]
and consider the diagram
\begin{align*}
\xymatrix{
(\Sigma H)\circ L\ar@{->}[r]^-{g\circ 1}&(\Sigma^2K)\circ L
&\hspace*{-16pt}\llap{\enlarge{\simeq{}}}
\Sigma^2(K\circ L)\\
(\Sigma H)\circ L
\ar@{->}[r]^-{g'\circ 1}
\ar@{=}[u]^{=}&
(\Sigma^2K)\circ L
\ar@{->}[u]^{[p]\circ 1}
&\hspace*{-16pt}\llap{\enlarge{\simeq{}}}
\Sigma^2(K\circ L)
\ar@<8pt>[u]^{[p]}}
\end{align*}
where the right hand square commutes by \ref{lem4.9}. The
map 
\[
g\circ 1\colon (\Sigma H)\circ L\to (\Sigma^2K)\circ L
\]
is the cofiber
of
$f\circ 1$, and $g'\circ 1$ is a co-$H$ map since $g'$ is a co-$H$ map.
Thus $f\circ 1$ is an index $p$ approximation. The other
case is similar.
\end{proof}
\begin{corollary}\label{cor4.12}
Suppose $\xymatrix{G\ar@{->}[r]^-{\enlarge{f}}&H}$ is an index $p$
approximation. Then
\[
f^{[i]}\!\colon G^{[i]}\to H^{[i]}
\]
is an iterated index $p$ approximation.
\end{corollary}
\begin{proof}
We first observe that
\[
\xymatrix{G\circ H^{[j]}\ar@{->}[r]^-{\enlarge{f\circ 1}}&H\circ
H^{[j]}=H^{[j+1]}}
\]
is an index $p$ approximation by \ref{prop4.11}. We then see
by induction that
\begin{align*}
G^{[i]}H^{[j]}=G\circ (G^{[i-1]}H^{[j]})\to
G\circ(G^{[i-2]}H^{[i+1]})=G^{[i-1]}H^{[j+1]}
\end{align*}
is an index $p$ approximation. Finally
\[
G^{[i]}\to G^{[i-1]}H
\]
is an iterated index $p$ approximation by induction since
it factors as
\[
G^{[i]}=G\circ(G^{[i-1]})\to G\circ(G^{[i-2]}H)=G^{[i-1]}H.
\]
Consequently
\[
G^{[i]}\to G^{[i-1]}H\to G^{[i-2]}H^{[2]}\to \dots\to G\circ H^{[i-1]}\to
H^{[i]}
\]
is an iterated index $p$ approximation.
\end{proof}
\begin{Theorem}\label{theor4.13}
Suppose
\[
\Omega X\to E\to G
\]
is a principal fibration classified by a map $\varphi\colon G\to X$ where
$X$ is an $H$-space with strict unit. Suppose $f\!\colon H\to G$ is
an index
$p$ approximation. Then, for any map $\nu\colon E\to BW_n$ the
compositions
\[
\xymatrix{
\Sigma(\Omega G\wedge \Omega G)\ar@{->}[r]^-{\enlarge{\Gamma}}&
E\ar@{->}[r]^-{\enlarge{\nu}}&
BW_n
}
\]
is null homotopic iff the composition
\[
\xymatrix@C=56pt{
\Sigma(\Omega H\wedge \Omega H)\ar@{->}[r]^-{\enlarge{\Sigma(\Omega f\wedge\Omega
f)}}&
\Sigma(\Omega G\wedge \Omega
G)\ar@{->}[r]^-{\enlarge{\Gamma}}&
E\ar@{->}[r]^-{\enlarge{\nu}}&
BW_n
}
\]
is null homotopic.
\end{Theorem}
\begin{proof}
Suppose the composition
\[
\xymatrix{
\Sigma(\Omega H\wedge \Omega H)\ar@{->}[r]&\Sigma(\Omega G\wedge \Omega
G)\ar@{->}[r]^-{\enlarge{\Gamma}}&
E\ar@{->}[r]^-{\enlarge{\nu}}&
BW_n
}
\]
is null homotopic. Let $\alpha\colon G\to G$ be the identity map and
$\beta=\alpha f$. Since $f\!\colon H\to G$ is a co-$H$ map,
there is a homotopy commutative diagram:
\[
\xymatrix@C=78pt{
H^{[i]}\ar@{->}[r]^-{ad_r^{i-2}(\beta)(\{\beta,\beta\}_{\times})}
\ar@{->}[d]_{f^{[i]}}
&\Sigma(\Omega H\wedge \Omega H)
\ar@{->}[d]\\
G^{[i]}
\ar@{->}[r]^-{ad^{i-2}(\alpha)(\{\alpha,\alpha\}_{\times})}
&\Sigma(\Omega G\wedge \Omega G)
}
\]
by \ref{prop3.6}(d) and \ref{prop3.11}(d).
Since $f$ is an index $p$ approximation,  $f^{[i]}$ is
an iterated
index $p$ approximation by \ref{cor4.12}; thus the compositions
\[
\xymatrix@C=74pt{
G^{[i]}\ar@{->}[r]^-{\enlarge{ad^{i-2}(\alpha)(\{\alpha,\alpha\}_{\times})}}&
\Sigma(\Omega G\wedge \Omega G)\ar@{->}[r]^-{\enlarge{\Gamma}}&
E\ar@{->}[r]^-{\enlarge{\nu}}&
BW_n
}
\]
are null homotopic for all $i\geqslant 2$. The result then follows
from \ref{theor3.22}.
\end{proof}

We will use this result to transfer
conditions on $\nu_k$ to the composition:
\[
\xymatrix{
\Sigma(\Omega L_k\wedge \Omega L_k)\ar@{->}[r]&
\Sigma(\Omega G_k\wedge \Omega
G_k)\ar@{->}[r]^-{\enlarge{\Gamma_k}}
&E_k.
}
\]
We need to iterate this. We have to consider the
issue that for $\zeta_k\colon\allowbreak L_k\to\nobreak G_k$ to be a co-$H$ map, we
need to use an exotic co-$H$ space structure on~$L_k$.
We will show that the triviality of the
composition above does not depend on the
co-$H$\ space structure of~$L_k$. To see this, recall that
the map\linebreak[4]
$\Gamma_k\colon\allowbreak \Sigma(\Omega G_k\wedge\nobreak G_k)\to E_k$ was defined in
section~\ref{subsec2.2} based on the fact that $E_k$ was defined
by a principal fibration
\[
\Omega S^{2n+1}\{p^r\}\to E_k\to G_k
\]
classified by a map $\varphi_k\colon G_k\to S^{2n+1}\{p^r\}$ where
$S^{2n+1}\{p^r\}$
is an $H$-space with
$H$-space structure map chosen to have a strict
unit. The fact that $G_k$ is a co-$H$ 
space was not used.

For any space $X$ and map $\zeta \colon X\to G_k$, we
can construct the pullback
\[
\xymatrix{
\Omega S^{2n+1}\{p^r\}
\ar@{=}[r]
\ar@{->}[d]&
\Omega S^{2n+1}\{p^r\}\ar@{->}[d]\\
E(X)\ar@{->}[r]\ar@{->}[d]&E_k\ar@{->}[d]\\
X\ar@{->}[r]^-{\zeta}&G_k
}
\]
which is induced by the composition $\varphi_k\zeta $.
Consequently there is a map $\Gamma(X)\colon \Sigma(\Omega X\wedge \Omega
X)\to E(X)$
and a strictly commutative diagram:
\[
\xymatrix{
\Sigma(\Omega X\wedge \Omega X)\ar@{->}[r]\ar@{->}[d]_{\Gamma(X)}&
\Sigma(\Omega G_k\wedge \Omega G_k)\ar@{->}[d]_{\Gamma_k}\\
E(X)\ar@{->}[r]&E_k.
}
\]
Consider the homotopy equivalence
\[
X_k=G_{k-1}\vee P^{2np^k}(p^{r+k-1})\vee P^{2np^k+1}(p^{r+k-1})\simeq L_k
\]
where we give $X_k$ the split co-$H$ space structure, so this map
is not a co-$H$ map. Nevertheless, we have a
strictly commutative diagram
\[
\xymatrix{
\Sigma(\Omega X_k\wedge\Omega
X_k)\ar@{->}[r]^-{\simeq}\ar@{->}[d]_{\Gamma(X_k)}&
\Sigma(\Omega L_k\wedge\Omega L_k)\ar@{->}[r]\ar@{->}[d]_{\Gamma(L_k)}&
\Sigma(\Omega G_k\wedge\Omega G_k)\ar@{->}[d]_{\Gamma_k}\\
E(X_k)\ar@{->}[r]^-{\simeq}\ar@{->}[d]&E(L_k)\ar@{->}[r]\ar@{->}[d]&E_k\ar@{->}[d]_{\pi_k}\\
X_k\ar@{->}[r]^-{\simeq}&L_k\ar@{->}[r]^-{\zeta_k}&G_k
}
\]
since $\zeta_k$ is an iterated index $p$ approximation, we have
\begin{proposition}\label{prop4.14}
For any map $\nu\colon E_k\to BW_n$, $\nu\Gamma_k$ is
null homotopic iff the composition
\[
\xymatrix{
\Sigma(\Omega X_k\wedge \Omega X_k)\ar@{->}[r]^-{\enlarge{\Gamma(X_k)}}&
E(X_k)\simeq E(L_k)\ar@{->}[r]&E_k\ar@{->}[r]^-{\enlarge{\nu}}&
BW_n
}
\]
is null homotopic.\qed
\end{proposition}

We now define spaces with split co-$H$ space
structures (coproducts in the category of co-$H$ 
spaces):
\begin{align*}
W(j,k)&=\bigvee\limits^k_{i=j}P^{2np^i}(p^{r+i-1})\vee P^{2np^i+1}(p^{r+i-1})\\
G(j,k)&=G_j\vee W(j+1,k)\\
L(j,k)&=L_j\vee W(j+1,k).
\end{align*}\index{$G(j,k)$|LB}
Consequently we have homotopy equivalences
\[
G(j,k)\simeq L(j+1,k)
\]
which are not co-$H$ equivalences, and index 
$p$ approximations
\[
\xymatrix{L(j,k)\ar@{->}[r]^-{\enlarge{\zeta_j\vee 1}}& G(j,k).}
\]
This leads to a chain:
\[
f\!\colon G(0,k)\simeq L(1,k)\to G(1,k)\simeq L(2,k)\to\dots\to G(k-1,k)\simeq L_k\to
G_k.
\]
\begin{Theorem}\label{theor4.15}
Suppose $n>1$. Then for
any given map $\nu\colon E_k\to BW_n$ the
composition
\[
\xymatrix@C=48pt{
\Sigma(\Omega G_k\wedge\Omega
G_k)\ar@{->}[r]^-{\enlarge{\Gamma_k}}&
E_k\ar@{->}[r]^-{\enlarge{\nu}}&
BW_n
}
\]
is null homotopic iff the composition
\[
\xymatrix@C=51pt{
\Sigma(\Omega G(0,k)\wedge \Omega G(0,k))\ar@{->}[r]^-{\enlarge{\Sigma(\Omega
f\wedge\Omega f)}}&
\Sigma(\Omega G_k\wedge\Omega G_k)
}\hspace*{-6pt}\xymatrix{\mbox{}\ar@{->}[r]^-{\enlarge{\Gamma_k}}&
E_k\ar@{->}[r]^-{\enlarge{\nu}}&
BW_n
}
\]
is null homotopic, where
\[
G(0,k)=P^{2n+1}\vee\bigvee\limits^k_{i=1}P^{2np^i}(p^{r+i-1})\vee
P^{2np^i+1}(p^{r+i-1})
\]
and the map $f\!\colon G(0,k)\to G_k$ is defined by the inclusion
of 
\[
P^{2n+1}=G_0\to G_k
\]
and the maps $\pi_kc(i)$ and
$\pi_ka(i)$ for $1\leqslant i\leqslant k$:
\begin{align*}
\makebox[53.5pt]{}
&\xymatrix{P^{2np^i}(p^{r+i-1})\ar@{->}[r]^-{\enlarge{a(i)}}&E_i\ar@{->}[r]^-{\enlarge{}}&
E_k\ar@{->}[r]^-{\enlarge{\pi_k}}&G_k}\\
&\xymatrix{P^{2np^i+1}(p^{r+i-1})\ar@{->}[r]^-{\enlarge{c(i)}}&E_i\ar@{->}[r]^-{\enlarge{}}&
E_k\ar@{->}[r]^-{\enlarge{\pi_k}}&G_k.}
\makebox[53.5pt]{}
\rlap{\qed}
\end{align*}
\end{Theorem}

Observe the diagram
\[
\xymatrix{
&E(0,k)\ar@{->}[r]\ar@{->}[d]&E_k\ar@{->}[d]\\
\Sigma(\Omega G(0,k)\wedge \Omega G(0,k))\ar@{->}[r]\ar@{->}[ru]^{\Gamma}&G(0,k)\ar@{->}[r]^-{f}&G_k
}
\]
where $\Gamma$ is $\Gamma_k$
composed with $\Sigma(\Omega f\wedge\Omega f)$. Since $G(0,k)$ is a wedge of
Moore
spaces, the components of~$\Gamma$ are $H$-space
based Whitehead products as defined by Neisendorfer~\cite{Nei10a}.
This will be studied in the next chapter.

\chapter{Simplification}\label{chap5}

In this chapter we work with the obstructions
obtained in Chapter~\ref{chap4}. These are $\bmod\, p^s$ homotopy
classes for $r\leqslant s\leqslant r+k$. In section~\ref{subsec5.1}, we define a
quotient
space $D_k$\index{$D_k$|LB} of~$G_k$ and a corresponding principal
fibration $J_k$ over~$D_k$. This has the property that
roughly half of the obstructions vanish in $J_k$;
we then seek a factorization of~$\nu_k$ through~$J_k$. In
section~\ref{subsec5.2} we introduce a congruence relation 
on homotopy classes, and show that we need only consider the
obstructions up to congruence homotopy. This leads
to a shorter list of obstructions.
Congruence homotopy will also be useful in
Chapter~\ref{chap6}, since the properties of relative
Whitehead products are simpler up to congruence
homotopy.

\section{Reduction}\label{subsec5.1}

The inductive hypothesis (\ref{theor6.7}) in the next chapter
is a strengthening of Proposition~\ref{prop2.12}, so in
particular, we will be assuming that
we have constructed a retraction $\nu_{k-1}\colon E_{k-1}\to BW_n$ such that
$\nu_{k-1}\Gamma_{k-1}$ is null homotopic.
In section~\ref{subsec4.1} we constructed classes $a(i)$, $c(i)$,
and $\beta_i$ for $i\leqslant k$ and in section~\ref{subsec4.2} we reduced the
constraints on the construction of~$\nu_k$ to a condition
involving the maps $a(i)$ and $c(i)$.
Some of
the material in this section and section~\ref{subsec6.2} can be found at {\ttfamily arXiv:0804.1896}. 

We now make a further simplification by burying
the classes $c(i)$ in the base space. Specifically, we
define a map
\[
c\index{$c$|LB}\colon C_k\index{$C_k$|LB}=\bigvee\limits^k_{i=1}P^{2np^i+1}(p^{r+i-1})\to E_k
\]
by the compositions
\[
\xymatrix{
P^{2np^i+1}(p^{r+i-1})\ar@{->}[r]^-{\enlarge{c(i)}}&
E_i\ar@{->}[r]&
E_k,
}
\]
and define\footnote{The spaces $D_k$ were first defined in \cite{Ani93}, but
were abandoned in~\cite{GT10}
as the related spaces $G_k$ have better properties. As we will see, the
spaces $D_k$ have smaller homotopy which is useful in our analysis.} $D_k$ by a cofibration
\[
\xymatrix{
C_k\ar@{->}[r]^-{\enlarge{\pi_kc}}&
G_k\ar@{->}[r]&
D_k.
}
\]
\begin{proposition}\label{prop5.1}
There is a homotopy commutative diagram
of cofibration sequences
\[
\xymatrix@C=33pt{
P^{2np^k+1}(p^{r+k-1})\ar@{->}[r]^-{\rho}&
P^{2np^k+1}(p^{r+k})\ar@{->}[r]^-{\sigma^{r+k-1}}&
P^{2np^k+1}(p)\\
C_k\ar@{->}[r]^-{\pi_kc}\ar@{->}[u]&
G_k\ar@{->}[r]\ar@{->}[u]^{\pi}&
D_k\ar@{->}[u]\\
C_{k-1}\ar@{->}[r]^-{\pi_{k-1}c}\ar@{->}[u]&G_{k-1}\ar@{->}[r]\ar@{->}[u]&D_{k-1}\ar@{->}[u] 
}
\]
and
\[
H_i(D_{k};Z_{(p)})=\begin{cases}
Z/p^r&\text{if}\ i=2n\\
Z/p&\text{if}\ i=2np^j,\ 1\leqslant j\leqslant k\\
0&\text{otherwise}.
\end{cases}
\]
\end{proposition}
\begin{proof}
The composition
\[
\xymatrix{
P^{2np^k+1}(p^{r+k-1})\ar@{->}[r]^-{\enlarge{c(k)}}&
E_k\ar@{->}[r]^-{\enlarge{\pi_k}}&
G_k\ar@{->}[r]^-{\enlarge{\pi}}&
P^{2np^k+1}(p^{r+k})
}
\]
is $\rho$ by~\ref{theor4.4}. The homology calculation is immediate.
\end{proof}

Since $\varphi_k\pi_k$ is null homotopic, we can extend $\varphi_k$
to a map 
\[
\varphi'_k\colon D_k\to S^{2n+1}\{p^r\}.
\]
Any such extension
defines a diagram of vertical fibration sequences:
\setcounter{equation}{1}
\begin{equation}\label{eq5.2}
\begin{split}
\xymatrix{
E_k\ar@{->}[r]^-{\tau_k}
\ar@{->}[d]_{\pi_k}&
J_k
\ar@{->}[r]^-{\eta_k}
\ar@{->}[d]_{\xi_k}
&
F_k
\ar@{->}[d]_{\sigma_k}
\\
G_k\ar@{->}[r]\ar@{->}[d]_{\varphi_k}&
D_k\ar@{=}[r]\ar@{->}[d]_{\varphi'_k}
&
D_k\ar@{->}[d]\\
S^{2n+1}\{p^r\}\ar@{->}[r]&
S^{2n+1}\{p^r\}\ar@{->}[r]&
S^{2n+1}
}
\end{split}
\end{equation}\addtocounter{Theorem}{1}
\index{tauk@$\tau_k$|LB}\index{$J_k$|LB}\index{etak@$\eta_k$|LB}\index{xik@$\xi_k$|LB}\index{$F_k$|LB}\index{sigmak@$\sigma_k$|LB}\index{varphiprimek@$\varphi'_k$|LB}
\begin{proposition}\label{prop5.3}
We can choose an extension $\varphi'_k$ of $\varphi'_{k-1}$
in such a way that the composition
\[
\xymatrix{
P^{2np^k+1}(p^{r+k-1})\ar@{->}[r]^-{\enlarge{c(k)}}&
E_k\ar@{->}[r]^-{\enlarge{\tau_k}}&
J_k
}
\]
is null homotopic.
\end{proposition}
\begin{proof}
We begin by defining $D$ by a pushout square:
\[
\xymatrix{
G_k\ar@{->}[r]&D\\
G_{k-1}\ar@{->}[u]\ar@{->}[r]&D_{k-1}\ar@{->}[u] 
}
\]
Using the lower right hand square in \ref{prop5.1} and the pushout
property, we see that there is a cofibration
\[
\xymatrix{
P^{2np^k+1}(p^{r+k-1})\ar@{->}[r]^-{\enlarge{\alpha}}&
D\ar@{->}[r]&
D_k
}
\]
where $\alpha$ is the composition:
\[
\xymatrix{
P^{2np^k+1}(p^{r+k-1})\ar@{->}[r]^-{\enlarge{c(k)}}&
E_k\ar@{->}[r]^-{\enlarge{\pi_k}}&
G_k\ar@{->}[r]&
D.
}
\]
We use the pushout property to construct $\varphi\colon D\to
S^{2n+1}\{p^r\}$
by $\varphi'_{k-1}$ on $D_{k-1}$ and $\varphi_k$ on $G_k$. We seek a map
$\varphi'_k$ 
 in the diagram
\[
\xymatrix{
E_k\ar@{->}[r]\ar@{->}[d]_{\pi_k}&J\ar@{->}[r]\ar@{->}[d]&J_k\ar@{->}[d]\\
G_k\ar@{->}[r]\ar@{->}[d]_{\varphi_k}&D\ar@{->}[r]\ar@{->}[d]_{\varphi}&D_k=D\cup_{\alpha}CP^{2np^k+1}(p^{r+k-1})\ar@{->}[d]_{\varphi'_k}\\
S^{2n+1}\{p^r\}\ar@{=}[r]^{}&S^{2n+1}\{p^r\}\ar@{=}[r]^{}&S^{2n+1}\{p^r\}
}
\]
We assert that we can choose $\varphi'_k$ so that the composition
\[
\xymatrix{
P^{2np^k+1}(p^{r+k-1})\ar@{->}[r]^-{\enlarge{c(k)}}&
E_k\ar@{->}[r]&
J\ar@{->}[r]&
J_k
}
\]
is null homotopic. Note that $\alpha$ is
homotopic to the composition
\[
\xymatrix{
P^{2np^k+1}(p^{r+k-1})\ar@{->}[r]^-{\enlarge{c(k)}}&
E_k\ar@{->}[r]&
J\ar@{->}[r]&
D.
}
\]
The assertion then follows from
\begin{lemma}\label{lem5.4}
Suppose $\xymatrix{J\ar@{->}[r]^-{\enlarge{\pi}}&D}$ is a principal fibration
induced by a map $\varphi\colon D\to S$. Suppose $c\colon Q\to J$ and
$D'$ is the mapping cone of $\pi c$. Then there is a map
$\varphi'\colon D'\to S$ with homotopy fiber $J'$ such that the composition
$\xymatrix{Q\ar@{->}[r]&J\ar@{->}[r]&J'}$
is null homotopic in the diagram:
\[
\xymatrix{Q\ar@{->}[r]^-{c}&J\ar@{->}[r]\ar@{->}[d]_{\pi}&J'\ar@{->}[d]\\
&D\ar@{->}[r]\ar@{->}[d]_{\varphi}&D'\ar@{->}[d]_{\varphi'}\\
&S\ar@{=}[r]^{}&S.  
}
\]
\end{lemma}
\begin{proof}
$J=\{(d,\omega)\in D\times PS\mid 
\omega(1)=\varphi(d)\}$ so
$c(q)$ has components $(c_1(q),c_2(q))$ where $c_1(q)\in D$, $c_2(q)\in
PS$
with $c_2(q)(0)=*$ and $c_2(q)(1)=\varphi(c_1(q))$. Write
$D'=D\cup_{\alpha}CQ$
with $0$ at the vertex of the cone and
$\alpha(q)=c_1(q)$. Now
define $\varphi'\colon D'\rightarrow  S$ by $\varphi'(d)=\varphi(d)$ for
$d\in D$ and $\varphi'(q,t)=c_2(q)(t)$.
This is well defined and we can define a homotopy
\[
H\colon Q\times I\to J'\subset D'\times PS
\]
by the formula
\[
H(q,t)=((q,t),c_2(q)_t)
\]
where $c_2(q)_t$ is the path defined as $c_2(q)_t(s)=c_1(q)(st)$.
\end{proof}

This proves the lemma and hence the proposition.
\end{proof}

Now assume $n>1$ and define
\setcounter{equation}{4}
\begin{equation}\label{eq5.5}
U_k\index{$U_k$|LB}=P^{2n+1}\vee \bigvee\limits^k_{i=1}P^{2np^i}(p^{r+i-1}),
\end{equation}
so $G(0,k)=U_k\vee C_k$,
and we have a homotopy commutative square
\[
\xymatrix{
G(0,k)\ar@{->}[r]\ar@{->}[d]&G_k\ar@{->}[d]\\
U_k\ar@{->}[r]^-{a}&D_k
}
\]
where the left hand vertical map is the projection and $a$
is defined on\linebreak[4]
$P^{2np^i}(p^{r+i-1})$ as the composition:
\[
\xymatrix{
p^{2np^i}(p^{r+i-1})\ar@{->}[r]^-{\enlarge{a(i)}}&
E_k\ar@{->}[r]^-{\enlarge{\tau_k}}&
J_k\ar@{->}[r]^-{\enlarge{\xi_k}}&
D_k.
}
\]
From this we construct homotopy commutative
diagram:
\begin{equation}\label{eq5.6}
\begin{split}
\xymatrix{
\Sigma(\Omega G(0,k)\wedge\Omega G(0,k))\ar@{->}[r]\ar@{->}[d]&
\Sigma(\Omega G_k\wedge\Omega
G_k)\ar@{->}[r]^-{\Gamma_k}\ar@{->}[d]&
E_k\ar@{->}[d]_{\tau_k}\\
\Sigma(\Omega U_k\wedge \Omega U_k)\ar@{->}[r]&
\Sigma(\Omega D_k\wedge \Omega D_k)\ar@{->}[r]^-{\overline{\Gamma}_k}&
J_k
}
\end{split}
\end{equation}
\addtocounter{Theorem}{2}
\begin{proposition}\label{prop5.7}
Suppose that $n>1$ and there is a retraction 
\[
\gamma_k\colon J_k\to BW_n
\]
such that the compositions
\[
\xymatrix@C=85pt{
U_k^{[j]}\ar@{->}[r]^-{\enlarge{\{a,\dots,a,\{a,a\}_{\times}\}_r}}&
J_k\ar@{->}[r]^-{\enlarge{\gamma_k}}&
BW_n
}
\]
are null homotopic for each $j\geqslant 2$. Then the
composition
\[
\xymatrix{
\Sigma(\Omega G_k\wedge \Omega
G_k)\ar@{->}[r]^-{\enlarge{\Gamma_k}}&
E_k\ar@{->}[r]^-{\enlarge{\tau_k}}&
J_k\ar@{->}[r]^-{\enlarge{\gamma_k}}&
BW_n
}
\]
is null homotopic.
\end{proposition}
\begin{proof}
By (\ref{eq5.6}) and \ref{theor4.15}, it suffices to show that the
composition
\[
\xymatrix{
\Sigma(\Omega U_k\wedge \Omega U_k)\ar@{->}[r]&
\Sigma(\Omega D_k\wedge \Omega D_k)\ar@{->}[r]^-{\enlarge{\overline{\Gamma}_k}}&
J_k\ar@{->}[r]^-{\enlarge{\gamma_k}}&
BW_n
}
\]
is null homotopic. Define $E(U_k)_k$ as a pullback:
\[
\xymatrix{
E(U_k)\ar@{->}[r]\ar@{->}[d]&J_k\ar@{->}[d]\\
U_k\ar@{->}[r]^-{a}&D_k 
}
\]
Then by naturality, it suffices to show that
the composition
\[
\xymatrix{
\Sigma(\Omega U_k\wedge \Omega
U_k)\ar@{->}[r]^-{\enlarge{\Gamma(U_k)}}&
E(U_k)\ar@{->}[r]&
J_k\ar@{->}[r]^-{\enlarge{\gamma_k}}&
BW_n
}
\]
is null homotopic. But since $U_k$ is a co-$H$ space,
we can apply \ref{theor3.22} to finish the proof.
\end{proof}

Now write  $U_k=\Sigma P_k$ where
\[
P_k\index{$P_k$|LB}=P^{2n}\vee \bigvee\limits^k_{i=1}P^{2np^i-1}(p^{r+i-1})
\]
so $U^{[j]}_k=\Sigma P_k\wedge\dots\wedge P_k$ by \ref{prop3.2}. 

Using the splitting of $P_k$ into $k+1$ Moore spaces,
we obtain
\begin{proposition}\label{prop5.8}
Suppose $n>1$. Then the map
\[
\{a,\dots,a,\{a,a\}_{\times}\}_r
\colon U_k^{[j]}=\Sigma
P_k^{(j)}\to J_k
\]
when restricted to one of the $(k+1)^j$ iterated smash
products of Moore spaces is an iterated external 
Whitehead product of the form
\[
\{x_1,\dots,x_{j-2},\{x_{j-1},x_j\}_{\times}\}_r
\]
where each $x_i$ is either $\xi_k\tau_ka(i)\colon P^{2np^i-1}(p^{r+i-1})\to D_k$
for $1\leqslant i\leqslant k$
or the inclusion $P^{2n+1}\to D_k$.\qed
\end{proposition}

By applying \ref{prop3.26}, we resolve these external Whitehead
products into internal Whitehead products.
\begin{Theorem}\label{theor5.9}
Suppose $n>1$. Then the restriction of the map
\[
\{a,\dots,a,\{a,a\}_{\times}\}_r
\colon \Sigma P_k\wedge\dots\wedge P_k\to J_k
\]
to any Moore space in any decomposition of $\Sigma P_k\wedge\dots\wedge
P_k$
is homotopic to a linear combination of weight $j$ iterated
internal $H$-space based Whitehead products
\[
[x_1,\dots,x_{j-2},[x_{j-1},x_j]_{\times}]_r
\]
where each $x_i$ is one of the following: $\xi_k\tau_ka(i)\rho^t$,
$\xi_k\tau_ka(i)\delta_t$, $\mu$, $\nu$ for $1\leqslant i\leqslant k$ and for appropriate
values of~$t$.
\end{Theorem}
\begin{proof}
This is an easy induction on $j$ using \ref{prop3.26}.
\end{proof}

\section{Congruence Homotopy Theory}\label{subsec5.2}

The results of section~\ref{subsec5.1}, and in particular~\ref{theor5.9},
indicate that the obstructions to constructing a
suitable retraction $\nu_k=\gamma_k\tau_k$ are $\text{mod}\, p^s$ homotopy
classes in~$J_k$ for $r\leqslant s\leqslant r+k$. These obstructions are iterated compositions of relative and $H$-space based Whitehead
products. However, they are not Whitehead products
of maps into $J_k$ in general. Since $BW_n$ is an 
$H$-space, any Whitehead products of classes in~$J_k$
will be annihilated by any such map~$\gamma_k$. We are
led to a coarser classification. We introduce a congruence
relation among homotopy
classes so that it is only necessary to annihilate a
representative of each congruence class. A remarkable
and useful feature is that the congruence
homotopy of~$J_k$ is a module over the symmetric algebra on $\pi_k(\Omega
D_k)$.
\begin{definition}\label{def7.1}
Two maps $f,g\colon X\to Y$ will be called
congruent (written $f\equiv g$) if $\Sigma f$ and $\Sigma g$ are homotopic
in~$[\Sigma X,\Sigma Y]$. We write $e[X,Y]$\index{$e[\ ,\ ]$|LB} for the set of
congruence classes of pointed maps: $X\to Y$
and
\[
e\pi_k\index{$e\pi_k$|LB}(Y;Z/p^s)=e[P^k(p^s),Y].
\]

Clearly congruence is an equivalence relation
and composition is well defined on congruence
classes. This defines the congruence homotopy
category. It is easy to prove
\end{definition}
\begin{proposition}\label{prop5.11}
Suppose $f\equiv
 g\colon X\to Y$ and $h\colon Y\to Z$\index{\{\ ,\ \}$_{\times}\equiv$@$\equiv$|LB}
where $Z$ is an $H$-space. $hf$ and $hg$ are homotopic.
\end{proposition}

Consequently, it is sufficient to classify the
iterated $H$-space based Whitehead products of~\ref{theor5.9} up to congruence. We will
also need
to consider congruence in a different way in
section~\ref{subsec6.3}. In constructing $\gamma_k$ we will make alterations
in dimensions where obstructions of level $k-1$ may
resurface. This is a delicate point
which has needed much attention. For this reason
we need to develop some deeper properties of 
congruence homotopy theory.
\begin{proposition}\label{prop5.12}
The inclusion $\xymatrix{Y_1\vee Y_2\ar@{->}[r]^-{\enlarge{i}}&Y_1\times Y_2}$ defines
a \textup{1--1} map $\iota_*\colon e[X,Y_1\vee Y_2]\to e[X,Y_1\times Y_2]$. Furthermore,
if $G$ is a co-$H$ space $e[G,X]$ is a Abelian group and
\[
e[G,Y_1\vee Y_2]\cong e[G,Y_1\times Y_2]\cong e[G,Y_1]\oplus e[G,Y_2].
\]
\end{proposition}
\begin{proof}
Suppose $f,g\colon X\to Y_1\vee Y_2$ and the compositions:
\begin{align*}
\xymatrix{\Sigma X\ar@{->}[r]^-{\enlarge{\Sigma f}}&\Sigma(Y_1\vee
Y_2)\ar@{->}[r]^-{\enlarge{\Sigma
i}}&\Sigma(Y_1\times Y_2)}\\
\xymatrix{\Sigma X\ar@{->}[r]^-{\enlarge{\Sigma g}}&\Sigma(Y_1\vee
Y_2)\ar@{->}[r]^-{\enlarge{\Sigma
i}}&\Sigma(Y_1\times Y_2)}
\end{align*}
are homotopic. Since $\Sigma i$ has a left homotopy inverse,
$\Sigma f$ and $\Sigma g$ are homotopic. The co-$H$ space structure
on~$G$ defines a multiplication on $[G,X]$ and the
map
\[
[G,X]\to [G,\Omega\Sigma X]
\]
is multiplicative. However $[G,\Omega\Sigma X]$ is a Abelian
group by a standard argument. Since $e[G,X]$
is a subgroup of $[G,\Omega\Sigma X]\cong [\Sigma G,\Sigma X]$, it also is
Abelian. Finally observe that the composition
\[
e[G,Y_1]\oplus e[G,Y_2]\to e[G,Y_1\vee Y_2]\to e[G,Y_1\times Y_2]\to
e[G,Y_1]\oplus e[G,Y_2]
\]
is the identity where the first and last maps are defined
by naturality. Thus the composition of the first two is \textup{1--1}.
But this composition is also onto since any element of
$e[G,Y_1\times Y_2]$ is represented by a map $G\to Y_1\times Y_2$ so all
these maps are isomorphisms.
\end{proof}
\begin{proposition}\label{prop5.13}
Suppose $G$ and $H$ are co-$H$ spaces.
Then composition defines a homomorphism:
\[
e[G,H]\otimes e[H,X]\to e[G,X].
\]
\end{proposition}
\begin{proof}
The only issue is the distributive law
\[
(\beta_1+\beta_2)\alpha\equiv \beta_1\alpha+\beta_2\alpha
\] 
for $\alpha\colon G\to H$ and $\beta_1,\beta_2\colon H\to X$. To prove this
we show that
the following diagram commutes up to congruence:
\[
\xymatrix@R=40pt@C=27pt{
G\ar@{->}[r]^-{\alpha}\ar@{->}[d]\ar@{}[dr]|{\equiv}&H\ar@{->}[d]\\
G\vee G\ar@{->}[r]^-{\alpha\vee\alpha}&H\vee H.
}
\]
This certainly commutes after the inclusion of $H\vee H\to H\times H$.
Thus it commutes up to congruence by~\ref{prop5.12}.
\end{proof}

This will be useful when $G$ and~$H$ are Moore spaces.
\begin{corollary}\label{cor5.14}
The category of co-$H$ spaces and
congruence classes of contiuous maps is an additive
category.
\end{corollary}
\begin{proof}
$G\vee H$ is both a product and co-product by~\ref{prop5.12}
and composition is bilinear by~\ref{prop5.13}.
\end{proof}

The following result will
be needed in \ref{prop5.19} and in section~\ref{subsec6.3}.
\begin{Theorem}\label{theor5.15}
Suppose $\varphi\colon \Sigma^2X\to P^{2m}(p^{r+t})$ has order $p^r$.
Then there are maps $\varphi_1\colon \Sigma^2X\to P^{2m}(p^r)$ and
$\varphi_2\colon \Sigma^2X\to S^{2m-1}$
such that
\[
\varphi\equiv\rho^t\varphi_1+\iota_{2m-1}\varphi_2.
\]
\end{Theorem}
\begin{proof}
According to \cite[11.1]{CMN79a} or \cite[1.2]{Gra99}, there
is a fibration sequence
\[
\xymatrix{
\Omega P^{2m}(p^{r+t})\ar@{->}[r]^-{\enlarge{\partial}}&
S^{2m-1}\{p^{r+t}\}\ar@{->}[r]&
W\ar@{->}[r]^-{\enlarge{\pi}}&
P^{2m}(p^{r+t})
}
\]
where $W$ is a $4m-3$ connected wedge of Moore spaces and
$\pi$ is an iterated Whitehead product on each factor.
In particular, $\pi$ is null congruent. A right homotopy
inverse for $\partial$ is constructed as follows. Given any
map $\theta\colon U\to V$, there is a natural map from
the fiber of~$\theta$ to the loop space on the cofiber:
\[
\Phi\colon F_{\theta}\to \Omega(V\cup_{\theta}CU).
\]
This defines a map $\Phi\colon S^{2m-1}\{p^{r+t}\}\to \Omega
P^{2m}(p^{r+t})$ and $\partial \Phi$ is a homology equivalence since
$S^{2m-1}\{p^{r+t}\}$ is atomic. This defines a splitting
of $\Omega P^{2m}(p^{r+t})$ and we have a direct sum decomposition
\[
\xymatrix{
[\Sigma^2 X,W]\oplus [\Sigma X,S^{2m-1}\{p^{r+t}\}]\ar@{->}[r]&
[\Sigma^2X,P^{2m}(p^{r+t})]\\
(\alpha,\beta)\ar@{<->}[r]&\varphi=\pi\alpha+\widetilde{\Phi\beta}
}
\]
where $\widetilde{\Phi\beta}$ is the adjoint of $\Phi\beta\colon \Sigma
X\to \Omega P^{2m}(p^{r+t})$.
Since $\varphi$ has order $p^r$, both $\alpha$ and $\beta$ have
order~$p^r$.
Since $\pi$ is null congruent, we have
\[
\varphi\equiv \widetilde{\Phi\beta}.
\]

Now consider the diagram of fibration sequences:
\[
\xymatrix{
F\ar@{->}[r]^-{f}\ar@{->}[d]&S^{2m-1}\{p^{r+t}\}\ar@{->}[r]^-{p^r}\ar@{->}[d]&S^{2m-1}\{p^{r+t}\}\ar@{->}[d]\\
S^{2m-1}\{p^r\}\ar@{->}[r]\ar@{->}[d]_{*\sim
p^{r+t}}&S^{2m-1}\ar@{->}[r]^-{p^r}\ar@{->}[d]_{p^{r+t}}&S^{2m-1}\ar@{->}[d]_{p^{r+t}}\\
S^{2m-1}\{p^r\}\ar@{->}[r]&S^{2m-1}\ar@{->}[r]^-{p^r}&S^{2m-1}.
}
\]
From this we see that $F\simeq S^{2m-1}\{p^r\}\times \Omega
S^{2m-1}\{p^r\}$. We choose a splitting of $F$ as follows: define a map
$\rho^t$
by the diagram of vertical fibration sequences:
\[
\xymatrix{
S^{2m-1}\{p^r\}\ar@{->}[r]^-{\rho^t}\ar@{->}[d]&
S^{2m-1}\{p^{r+t}\}\ar@{->}[d]\\
S^{2m-1}\ar@{=}[r]\ar@{->}[d]_{p^r}&
S^{2m-1}\ar@{->}[d]_{p^{r+t}}\\
S^{2m-1}\ar@{->}[r]^-{p^t}&
S^{2m-1}.
}
\]
The map $\rho^t$ factors through $f$ and defines a splitting. Thus
the composition
\[
\xymatrix{
S^{2m-1}\{p^r\}\times \Omega S^{2m-1}\{p^r\}\ar@{->}[r]^-<<<{\simeq}&
F\ar@{->}[r]^-{f}&S^{2m-1}\{p^{r+t}\}
}
\]
is homotopic to the map
\begin{multline*}
S^{2m-1}\{p^r\}\times \Omega
S^{2m-1}\{p^r\}\\
\xymatrix{\ar@{->}[r]^-{\enlarge{\rho^t\times\delta_t}}&
S^{2m-1}\{p^{r+t}\}\times
S^{2m-1}\{p^{r+t}\}\ar@{->}[r]^-{\enlarge{\mu}}&S^{2m-1}\{p^{r+t}\}}
\end{multline*}
where $\delta_t$ is the composition
\[
\xymatrix{
\Omega S^{2m-1}\{p^r\}\ar@{->}[r]^-{\enlarge{\Omega\pi}}&
\Omega S^{2m-1}\ar@{->}[r]^-{\enlarge{\iota}}&
S^{2m-1}\{p^{r+t}\}.
}
\]
Since $\beta$ has order $p^r$, it factors through $F$ and we
conclude that $\beta$ is homotopic to a composition:
\begin{multline*}
\xymatrix@C=37pt{
\Sigma X
\ar@{->}[r]&
S^{2m-1}\{p^r\}\times \Omega S^{2m-1}\{p^r\}
}\\
\xymatrix@C=37pt{
\ar@{->}[r]^-{\enlarge{\rho^t\times \delta_t}}&
S^{2m-1}\{p^{r+t}\}\times S^{2m-1}\{p^{r+t}\}
\ar@{->}[r]&S^{2m-1}\{p^{r+t}\}.
}
\end{multline*}
This map is homotopic to the sum of the two compositions
\begin{align*}
&\xymatrix{
\Sigma X\ar@{->}[r]&
S^{2m-1}\{p^r\}\ar@{->}[r]^-{\enlarge{\rho^t}}&
S^{2m-1}\{p^{r+t}\}
}\\
&\xymatrix{
\Sigma X\ar@{->}[r]&
\Omega S^{2m-1}\{p^r\}\ar@{->}[r]&
\Omega S^{2m-1}\ar@{->}[r]&
S^{2m-1}\{p^{r+t}\}.
}
\end{align*}
By the naturality of $\Phi$, $\Phi\beta$ is homotopic to the sum
of the maps
\begin{align*}
&\xymatrix@C=35pt{\Sigma X\ar@{->}[r]^-{\enlarge{\widetilde{\varphi}_1}}& \Omega
P^{2m}(p^r)\ar@{->}[r]^-{\enlarge{\Omega\rho^t}}&\Omega P^{2m}(p^{r+t})}\\
&\xymatrix@C=35pt{\Sigma X\ar@{->}[r]^-{\enlarge{\widetilde{\varphi}_2}}& 
\Omega S^{2m-1}\ar@{->}[r]^-{\enlarge{\Omega\iota_{2m-1}}}&\Omega
P^{2m}(p^{r+t}).}
\end{align*}
Thus $\varphi\equiv \widetilde{\phi\beta}$ which is homotopic
to $\rho^t\varphi_1+\iota_{2m-1}\varphi_2$.
\end{proof}

At this point we will examine relative Whitehead products in congruence
homotopy. This will be applied to determining
the congruence homotopy classes of the obstruction in Theorem~\ref{theor5.9}.

We define a $Z/p^r$ module $M_*(B)$ by
\[
M_m=\pi_{m+1}(B;Z/p^r)\approx \pi_m(\Omega B;Z/p^r).
\]
\setcounter{equation}{9}
\begin{proposition}\label{prop5.16}
The relative Whitehead product
induces a homomorphism:
\[
M_m\index{$M_m$|LB}\otimes e\pi_k(E;Z/p^r)\to e\pi_{m+k}(E;Z/p^r)
\]
which commutes with the Hurewicz homomorphism; i.e.,
the diagram
\[
\xymatrix{
M_m\otimes
e\pi_k(E;Z/p^r)\ar@{->}[r]\ar@{->}[d]&e\pi_{m+k}(E;Z/p^r)\ar@{->}[d]\\
H_m(\Omega B;Z/p^r)\otimes H_k(E;Z/p^r)\ar@{->}[r]^-<<<<{a_*}&H_{m+k}(E;Z/p^r)
}
\]
commutes where $a_*$ is induced by the principal action.
\end{proposition}
\begin{proof}
To see that the pairing is well defined, it suffices to show that if $\Sigma\delta\sim *$, then
$\Sigma[\alpha,\delta]_r\sim*$.
Recall that $[\alpha,\delta]_r$ is given in \ref{prop3.16} by the composition:
\begin{multline*}
\xymatrix@C=33pt{
P^{m+k}\ar@{->}[r]^-{\enlarge{\Delta}}&
P^m\wedge P^k\ar@{->}[r]^-{\enlarge{\theta}}&
P^{m}\ltimes
P^k\ar@{->}[r]^-<<<<{\enlarge{\widetilde{\alpha}\ltimes \delta}}&
\Omega B\ltimes E\ar@{->}[r]^-{\enlarge{\Gamma'}}&
E.
}
\end{multline*}
However, since there is a natural homeomorphism
\[
\Sigma(X\ltimes Y)\cong X\ltimes \Sigma Y,
\]
$\Sigma[\alpha,\delta]_r$ factors through the map
\[
\xymatrix@C=35pt{P^m\ltimes P^{k+1}\ar@{->}[r]^-{\enlarge{\widetilde{\alpha}\ltimes
\Sigma\delta}}&\Omega B\ltimes \Sigma E.}
\]
This map is null homotopic since it factors through
$\Omega B\ltimes *$ up to homotopy.
Clearly the Hurewicz map factors through
congruence homotopy and the homology
calculation follows from \ref{cor3.18}.
\end{proof}
\begin{definition}\label{def5.17}
$A_*(B)$\index{$A_*(\ )$|LB} is the graded symmetric
algebra generated by~$M_*(B)$\index{$M_*(\ )$|LB}.
\end{definition}
\begin{Theorem}\label{theor5.18}
The relative Whitehead product
induces the structure of a graded
differential $A_*(B)$ module on $e\pi_*(E;Z/p^r)$.
\end{Theorem}
\begin{proof}
By \ref{prop5.16}, there is an action of~$M_*(A)$ on
$e\pi_*(E;Z/p^r)$ and hence an action of the tensor
algebra. It suffices to show that
\[
[\alpha,[\beta,\gamma]_r]_r\equiv (-1)^{mn}[\beta,[\alpha,\gamma]_r]_r
\]
where $\alpha\in M_m$ and $\beta\in M_n$. However we have
\begin{align*}
[\alpha,[\beta,\gamma]_r]_r&=[\alpha,[\beta,\pi_*(\gamma)]]_{\times}\\
&=[[\alpha,\beta],\pi_*(\gamma)]_{\times}+(-1)^{mn}[\beta,[\alpha,\pi_*(\gamma)]]_{\times}
\end{align*}
by \ref{prop3.30}(c), (d) and \ref{prop3.31}(c). But
$[\alpha,\beta]=\pi_*[\alpha,\beta]_{\times}$ by \ref{prop3.30}(a),
so we have
\begin{align*}
[[\alpha,\beta],\pi_*(\gamma)]_{\times}&=[\pi_*([\alpha,\beta]_{\times}),\pi_*(\gamma)]_{\times}\\
&=[[\alpha,\beta]_{\times},\gamma]
\end{align*}
by \ref{prop3.30}(b). Since $[[\alpha,\beta]_{\times},\gamma]$ is a Whitehead
product of
classes in $\pi_*(E;Z/p^r)$, $[[\alpha,\beta]_{\times},\gamma]\equiv 0$.
Thus
\begin{align*}
\makebox[84pt]{}
[\alpha,[\beta,\gamma]_r]_r&\equiv(-1)^{mn}[\beta,[\alpha,\pi_*(\gamma)]]_{\times}\\
&=(-1)^{mn}[\beta,[\alpha,\gamma]_r]_r.
\makebox[84pt]{}
\rlap{\qed}
\end{align*}\noqed
\end{proof}

Since the action of $A_*(B)$ is associative, we will
also use the notation
\[
\alpha\cdot \gamma=[\alpha,\gamma]_r
\]
to simplify the notation and distinguish this
action from composition.
\begin{proposition}\label{prop5.19}
Suppose $\xi\colon H'\to H$, $\delta\colon H\to E$ and $\alpha\colon \Sigma
P\to B$. Then

\textup{(a)}\hspace*{0.5em}$\{\alpha,\delta\xi\}_r\equiv
\{\alpha,\delta\}_r(1\wedge\xi)\colon A\wedge H'\to E$.

\textup{(b)}\hspace*{0.5em}Suppose $P=P^m$, $H=P^k$ and $H'=P^{\ell}$. Then
there is
a splitting $P\wedge H\simeq P^{m+k}\vee P^{m+k-1}$ such that
\[
[\alpha,\delta\xi]_r\equiv[\alpha,\delta]_r(\Sigma^m\xi)\vee[\beta(\alpha),\delta]_r(\xi')
\]
where $\xi'\colon P^{m+\ell}\to P^{m+k-1}$.
\end{proposition}
\begin{remark*}
These relative Whitehead products are
not homotopic in general unless $\xi$ is a co-$H$ map.
\end{remark*}
\begin{proof}
Using \ref{prop3.16}, we construct the following diagram:
\[
\xymatrix@C=41pt{
P\wedge H\ar@{->}[r]^{\theta}&P\ltimes
H\ar@{->}[r]^{\widetilde{\alpha}\ltimes\delta}&\Omega B\ltimes E\ar@{->}[dr]^{\Gamma'}\ar@{=}[dd]&\\
&&&E\\
P\wedge H'\ar@{->}[uu]^{1\wedge\xi}\ar@{->}[r]^{\theta}
&P\ltimes H'\ar@{->}[r]^{\widetilde{\alpha}\ltimes \delta\xi}
\ar@{->}[uu]^{1\ltimes\xi}&\Omega B\ltimes E\ar@{->}[ur]_{\Gamma'}&
}
\]
The two right hand regions are commutative. We
claim that the left hand region commutes up
to congruence. Since $P\ltimes H$ is a co-$H$ space, there
is a homotopy equivalence
\[
P\ltimes H\simeq P\wedge H\vee H
\]
given by the sum of the map pinching $H$ to a point
and the projection onto~$H$. Since the composition
\[
\xymatrix{
P\wedge H\ar@{->}[r]^{\theta}&P\ltimes H\ar@{->}[r]&P\wedge H
}
\]
is a homotopy equivalence while the composition
\[
\xymatrix{
P\wedge H\ar@{->}[r]^{\theta}&P\ltimes H\ar@{->}[r]&H
}
\]
is null homotopic, both compositions in the left
hand square become homotopic when composed with the map
\[
\xymatrix{
P\ltimes H\ar@{->}[r]^->>>>{\simeq}&P\wedge H\vee H\ar@{->}[r]&(P\wedge H)\times
H.
}
\]
The result then follows from \ref{prop5.12}.

For part~$b$, we split $P^m\wedge P^k$ by the composition
\[
\xymatrix@C=45pt{
P^{m+k}\vee P^{m+k-1}\ar@{->}[r]^->>>>>>>{\Delta\vee \Delta}&P^m\wedge P^k\vee
P^{m-1}\wedge P^k\ar@{->}[r]^-<<<<<<<{1\vee \beta\wedge 1}&P^m\wedge P^k
}
\]
(compare to \ref{prop3.25}) $[\alpha,\delta\xi]_r$ is given by
precomposing
$\{\alpha,\delta\xi\}_r$ with $\Delta$. The composition
\[
\xymatrix@C=30pt{
P^{m+\ell}\ar@{->}[r]^->>>>{\Delta}&P^m\wedge P^{\ell}
\ar@{->}[r]^->>>>>{1\wedge \xi}&P^m\wedge P^k\simeq P^{m+k}\vee P^{m+k-1}
}
\]
is congruent 
to a map with components $\Sigma^m\xi$ and 
$\xi'\colon P^{m+k}\to P^{m+k-1}$.
The result then follows from the splitting.
\end{proof}

In case $n>1$, we now apply the $A_*(D_k)$ module structure to
the study of the congruence classes of the obstructions
in Theorem~\ref{theor5.9}. We will actually only consider the subalgebra
of~$A_*(D_k)$ generated by $\nu\index{nu@$\nu$|LB}\in M_{2n}=\pi_{2n+1}(D_{k}Z/p_r)$
and $\mu\index{mu@$\mu$|LB}=\beta\nu\in M_{2n-1}$. These elements generate
a subalgebra
\[
Z/p[\nu]\otimes\wedge(\mu)\subset A_*(D_k)
\]
and $e\pi_*(J_k;Z/p^r)$ is a module over this
algebra. We define classes
\begin{align*}
\overline{a(i)}
&=\tau_ka(i)\rho^{i-1}\colon P^{2np^i}\to J_k\\
\overline{b(i)}
&=\tau_ka(i)\delta_{i-1}\colon P^{2np^i-1}\to J_k\\
\overline{a(0)}&=[\nu,\mu]_{\times}\\
\overline{b(0)}&=[\mu,\mu]_{\times}
\end{align*}\index{$a(i)$@$\overline{a(i)}$|LB}\index{$b(i)$@$\overline{b(i)}$|LB}
for $1\leqslant i\leqslant k$.
\begin{Theorem}\label{theor5.20}
In case $n>1$, the collection of congruence classes
of the set of obstructions listed in Theorem~\ref{theor5.9} is
spanned, as a module over $Z/p[\nu]\otimes\wedge(\mu)$
by the classes $\overline{a(i)}$ and $\overline{b(i)}$ of weight $j\geqslant 2$
for $0\leqslant i\leqslant k$, where the weight of $\overline{a(i)}$ and
$\overline{b(i)}$ are
both one, except when $i=0$ where the weight is two.
\end{Theorem}
\begin{proof}
We first consider internal $H$-space based
Whitehead products of weight~$2$. Recall that
by~\ref{prop3.30}(b) $[\xi_k\gamma,\xi_k\delta]_{\times}=[\gamma,\delta]$ which is
null congruent, so we need only consider
weight~$2$ products $[x_1,x_2]_{\times}$ in which at least
one of $x_1,x_2$ is not in the image of~$\xi_k$. By \ref{prop3.31}(a), we 
will assume that $x_1=\mu$ or $\nu$. This gives the following
possibilities for weight~$2$.
\[
\overline{a(0)},\
\overline{b(0)},\
[\nu,\xi_k\overline{a(i)}]_{\times},\
[\mu,\xi_k\overline{a(i)}]_{\times},\
[\nu,\xi_k\overline{b(i)}]_{\times},\
[\mu,\xi_k\overline{b(i)}]_{\times}
\]
for $1\leqslant i\leqslant k$. Applying \ref{prop3.30}(e) again, we see that for $j>2$
the class of 
\[
[x_1,\dots,x_{j-2},[x_{j-1},x_j]_{\times}]_r
\]
is null congruent if $x_1$
is in the image ov~$\xi_k$. Thus each of $x_1,\dots,x_{j-2}$ must
be either $\mu$ or~$\nu$, 
and this class is in the $Z/p[\nu]\otimes\wedge(\mu)$
submodule generated by $[x_{j-1},x_j]_{\times}$.
\end{proof}

We will refer to the submodule
generated by $\overline{a(k)}$ and $\overline{b(k)}$ as the level~$k$
obstructions. In case $k=0$ we have some
simple relations:
\begin{proposition}\label{prop5.21}
$\mu\cdot\overline{b(0)}\equiv0$ and $\nu\cdot\overline{b(0)}\equiv
2\mu\cdot\overline{a(0)}$.
Consequently the submodule generated by $\overline{a(0)}$
 and $\overline{b(0)}$ has a basis consisting of $\nu^k\cdot\overline{b(0)}$ and
$\nu^k\cdot\overline{a(0)}$ for $k\geqslant 0$.
\end{proposition}
\begin{proof}
\begin{tabular}[t]{@{\hspace*{50pt}}l@{}}
$\mu\cdot[\mu,u]_{\times}\equiv [\mu,[\mu,u]_{\times}]_r\equiv [\mu,[\mu,u]]_{\times}\equiv 0.$\\
$\mu\cdot[\nu,\mu]_{\times}\equiv [\mu,[\nu,\mu]_{\times}]_r\equiv [\mu,[\nu,\mu]]_{\times}$.
\end{tabular}\\
Using \ref{prop3.31}(a) and~(c)
we get
\[
[\mu,[\nu,\mu]]_{\times}=[[\mu,\nu],\mu]_{\times}
+[\nu,[\mu,\mu]]_{\times}=-[\mu,[\nu,\mu]]_{\times}
+[\nu,[\mu,\mu]]_{\times},
\]
so $2\mu\cdot[\nu,\mu]_{\times}=[\nu,[\mu,u]]_{\times}\equiv \nu\cdot\nobreak[\mu,\mu]_{\times}$.
\end{proof}

Because of these relations we define\footnote{The class $x_k\colon P^{2nk}\to D_0=P^{2n+1}$ is the adjoint of the
similarly named class in~\cite{CMN79a}.}
$x_2=[\nu,\mu]_{\times}$
and $y_2=\frac{1}{2}[\mu,\mu]_{\times}$. Then
$x_k=\nu\cdot x_{k-1}=\nu^{k-2}\cdot\overline{a(0)}$ and $y_k=\mu\cdot x_{k-1}$.
\begin{proposition}\label{prop5.22}
In case $n>1$, the level $0$ congruence classes are
generated by $x_j$\index{$x_j$|LB} and $y_j$\index{$y_j$|LBmorespace} for $j\geqslant 2$ with the relations
$\mu\cdot x_k\equiv\nu\cdot y_k$ and $\nu\cdot y_k\equiv 0$. Furthermore $\beta x_j\equiv jy_j$
and $\beta y_j\equiv 0$.
\end{proposition}
\begin{proof}
$\mu\cdot x_k\equiv \nu^{k-2}\mu\cdot
x_2\equiv\frac{1}{2}\nu^{k-1}\cdot[\mu,\mu]_{\times}\equiv\nu^{k-1}\cdot y_2=\nu\cdot y_k$.
$\mu\cdot y_k\equiv\mu \nu^{k-2}\cdot y_2\equiv\frac{1}{2}\nu^{k-2}\mu\cdot
[\mu,\mu]_{\times}\equiv 0$ by \ref{prop5.21}. These
relations imply that the $x_k$ and $y_k$ are linear
generators. We will see in section~\ref{subsec6.1} that they 
are actually linearly independent. $\beta x_k=(k-2)\nu^{k-3}\cdot x_2
+\nu^{k-2}\cdot [\mu,\mu]_{\times}\equiv(k-2)\mu\cdot
x_{k-1}+2y_k\equiv ky_k$. 
$\beta y_k\equiv (k-2)\nu^{k-3}\mu\cdot [\mu,\mu]_{\times}\equiv 0$.
\end{proof}

\chapter{Constructing $\gamma_k$}\label{chap6}
In this chapter we assemble the material in the previous
chapters and construct a proof of Theorem~A. In \ref{subsec6.1} we
introduce the controlled extension theorem and apply this
in case $k=0$. This case is simpler since $D_0=G_0$, and serves as a model
for the more
complicated case when~$k>0$.
We recall a space $F_0$ (called $F$ in~\cite{GT10}) and construct
$\nu_0$ over the skeleta of~$F_0$ as in~\cite{GT10}. However we choose
$\nu_0$ so that it annihilates $x_m$ and~$y_m$ for each $m\geqslant 2$,
and
conclude that $\nu_0\Gamma_0$ is null homotopic.

In \ref{subsec6.2} we establish an inductive hypothesis and
derive important properties from the inductive assumption.
In particular, we construct $F_k$ and calculate its
homology. This allows for a secondary induction over
the skeleta of~$F_k$.

In \ref{subsec6.3} the calculations are made to construct
$\nu_k=\gamma_k\eta_k$ so that it annihilates the level $k$ obstructions.
At this point it is necessary to show that the level
$k-1$ obstructions do not reappear. This requires
a careful analysis of the congruence class of the level $k-1$
obstructions. Theorem~\ref{theor6.28} and Corollary~\ref{cor6.37} provide
the necessary decompositions, and the issue is
resolved by~\ref{lem6.40}. The induction is completed by~\ref{theor6.41}
and~\ref{theor6.43}

\section{Controlled Extension and the Case $k=0$}\label{subsec6.1}

In this section we will introduce the controlled
extension theorem and apply it to the simplest case: the
construction of a retraction map 
\[
\nu_0\colon E_0\to BW_n
\]
such that the composition
\[
\xymatrix{
\Sigma (\Omega G_0\wedge \Omega G_0)\ar@{->}[r]^-{\Gamma_0}&
E_0\ar@{->}[r]^-{\nu_0}&
BW_n
}
\]
is null homotopic. This case is considerably simpler than
the case $k>0$ and will serve as a model for the later
cases. The controlled
extension theorem is an enhancement of the extension
theorem (\ref{theor2.5}), and our construction of~$\nu_0$ is a
controlled version of the construction of $\nu_0=\nu^E$
of section~3 of~\cite{GT10}.
\begin{Theorem}[Controlled extension theorem]\label{theor6.1}
Suppose that
all spaces are localized at $p>2$ and we have a
diagram of principal fibrations induced by
a map $\varphi\colon B\cup_{\theta} e^m\to X$:
\[
\xymatrix{
\Omega X\ar@{=}[r]\ar@{->}[d]&
\Omega X\ar@{->}[d]\\
E_0\ar@{->}[r]\ar@{->}[d]&E\ar@{->}[d]_{\pi}\\
B\ar@{->}[r]&B\cup_{\theta}e^m.
}
\]
Suppose $\dim B<m$ and we are
given a map $\chi\colon P^m(p^s)\to E$ with $s\geqslant 1$
such that $\pi \chi\colon P^m(p^s)\to B\cup_{\theta}e^m$ induces an
isomorphism
in $\text{mod}\, p$ cohomology in dimension~$m$. Suppose also
that we are given a map $\gamma_0\colon E_0\to BW_n$. Then:
\begin{enumerate}
\item[(a)]
There is an extension of $\gamma_0$ to $\gamma'\colon E\to BW_n$.
\item[(b)]
Suppose also that we are given a map $u\colon P\to E$
and a subspace $P_0\subset P$ such that the composition
\[
\xymatrix{
P_0\ar@{->}[r]&P\ar@{->}[r]^-{u}&E\ar@{->}[r]^-{\gamma'}&BW_n
}
\]
is null homotopic and such that the quotient map
$q\colon P\to P/P_0$ factors up to homotopy
\[
\xymatrix{
P\ar@{->}[r]^-{u}&E\ar@{->}[r]^-{q'}&E/E_0\ar@{->}[r]^-{\xi}&P/P_0
}
\]
for some map $\xi$. Then there is an extension $\gamma$ of $\gamma_0$
such that $\gamma u\sim *$.
\end{enumerate} 
\end{Theorem}
\begin{proof}
Clearly the map $\pi\chi\colon (P^m(p^s),S^{m-1})\to (B\cup_{\theta}e^m,B)$
induces an isomorphism in homology, so the existence of $\gamma'$ follows from the extension theorem (\ref{theor2.5}).

To prove part (b), we suppose $\gamma'$ is given and we construct
$\gamma$ as the composition:
\[
\makebox[0pt]{}
\xymatrix@C=30pt{
E\ar@{->}[r]^-{\Delta}&E\!{}\times{}\!
E/E_0\ar@{->}[r]^-{\gamma'\!{}\times{}\!\xi}&BW_n\!{}\times{}\!P/P_0\ar@{->}[r]^-{1\!{}\times{}\!
\eta}&BW_n\!{}\times{}\!BW_n\ar@{->}[r]^-{\div}&BW_n
}
\makebox[0pt]{}
\]
where $\eta\colon P/P_0\to BW_n$ is defined by the null homotopy of
$\gamma'u\vert_{P_0}$ and $\div$ is the $H$-space division map. Clearly
$\gamma\vert_{E_0}\sim \gamma_0$. To study $\gamma\vert_P$, consider the diagram
\[
\xymatrix@C=30pt{
E\ar@{->}[r]^-{\Delta}&
E\!{}\times E/E_0\ar@{->}[r]^-{\gamma'\!{}\times{}\! \xi}&
BW_n\!{}\times{}\! P/P_0\ar@{->}[r]^-{1\!{}\times{}\! \eta}&
BW_n\!{}\times{}\! BW_n\ar@{->}[r]^-{\div}\ar@{=}[d]&
BW_n\ar@{=}[d]\\
P\ar@{->}[r]^-{\Delta}\ar@{->}[u]^{u}&
P\!{}\times{}\! P\ar@{->}[u]^{u\!{}\times{}\! q'u}\ar@{->}[r]^-{\eta
q\!{}\times{}\! q}&
BW_n\!{}\times{}\! P/P_0\ar@{->}[u]\ar@{->}[r]^-{1\!{}\times{}\!\eta}&
BW_n\!{}\times{}\! BW_n\ar@{->}[r]^-{\div}&
BW_n
}
\]
where the lower composition of the first 3 maps factors
through the diagonal map of~$BW_n$, so the lower
composition is null homotopic.
\end{proof}

Now we will apply this in the case $k=0$. Recall
the spaces (\ref{eq5.2}). In this case $D_0=G_0=P^{2n+1}$ and $J_0=E_0$:
\[
\xymatrix{
\Omega^2 S^{2n+1}\ar@{->}[r]&
E_0\ar@{->}[r]^-{\tau_0}\ar@{->}[d]&
F_0\ar@{->}[r]\ar@{->}[d]&
\Omega S^{2n+1}\ar@{->}[d]\\
&P^{2n+1}\ar@{=}[r]\ar@{->}[d]&
P^{2n+1}\ar@{->}[r]\ar@{->}[d]&
PS^{2n+1}\ar@{->}[d]\\
&S^{2n+1}\{p^r\}\ar@{->}[r]&
S^{2n+1}\ar@{->}[r]^-{p^r}&
S^{2n+1}.
}
\]
These spaces were introduced in~\cite{CMN79c} where $E_0$ is
called $E^{2n+1}(p^r)$ and $F_0$ is called $F^{2n+1}(p^r)$.
\begin{proposition}\label{prop6.2}
$H^i(F_0;Z_{(p)})=\begin{cases}
Z_{(p)}&\text{if}\ i=2mn\\
0&\text{otherwise}.
\end{cases}$
\end{proposition}
\begin{proof}
This is immediate from consideration of the
cohomology Serre spectral sequence of the middle
fibration, which is induced from the path space fibration
on the right. All differentials are controlled by the
path space fibration.
\end{proof}

We filter $F_0$ by setting $F_0(m)$\index{$F_0(m)$|LB} to be the $2mn$ skeleton
of~$F_0$ and define $E_0(m)$\index{$E_0(m)$|LB} to be the pullback over~$F_0(m)$
\[
\xymatrix{
\Omega^2S^{2n+1}\ar@{=}[r]\ar@{->}[d]&
\Omega^2S^{2n+1}\ar@{->}[d]\\
E_0(m)\ar@{->}[r]\ar@{->}[d]&
E_0\ar@{->}[d]_{\eta_0}\\
F_0(m)\ar@{->}[r]&F_0
}
\]
Since $F_0(1)=S^{2n}$, this fibration in case
$m=1$ is the fibration which defines $BW_n$ (see \ref{prop2.3})
\[
\xymatrix{
\Omega^2S^{2n+1}\ar@{->}[r]&S^{4n-1}\times
BW_n\ar@{->}[r]&S^{2n}\ar@{->}[r]&\Omega S^{2n+1}.
}
\]
We consequently define $\nu_0(1)\colon E_0(1)\to BW_n$ by
retracting onto $BW_n$. Clearly $\nu_0(1)y_2$
is null homotopic since $y_2\colon P^{4n-1}\to E_0$ and $BW_n$
is $2np-3$ connected. We will use \ref{theor6.1} to
construct $\nu_0(m)\colon E_0(m)\to BW_n$ such that $\nu_0(m)y_i$
and $\nu_0(m)x_{i-1}$ are null homotopic for $i\leqslant m+1$.
\begin{proposition}\label{prop6.3}
For each $m\geqslant 2$, there is a retraction
\[
\nu_0(m)\colon E_0(m)\to BW_n
\]
extending $\nu_0(m-1)$ such that
$\nu_0(m)_{*}$ annihilates the classes $x_m$ and $y_{m+1}$ from~\ref{prop5.22}.
\end{proposition}

The proof of this result will depend on two lemmas.
\begin{lemma}\label{lem6.4}
The composition
\[
\xymatrix{
P^{2nj}\ar@{->}[r]^-{x_j}&
E_0\ar@{->}[r]^-{\eta_0}&
F_0
}
\]
induces a cohomology epimorphism.
\end{lemma}
\begin{proof}
To study $\eta_0x_j$, we use the principal fibration:
\[
\xymatrix{
\Omega S^{2n+1}\ar@{->}[r]&
F_0\ar@{->}[r]&
P^{2n+1}
}
\]
Clearly $\mu\colon P^{2n}\to P^{2n+1}$ lifts to a map $x_1\colon
P^{2n}\to F_0$
which induces a cohomology epimorphism. Then
$x_j=[\nu,x_{j-1}]_r$ for each $j\geqslant 2$ so we can apply 
\ref{cor3.18}
to evaluate $x_j$ in cohomology.
\end{proof}
\begin{lemma}\label{lem6.5}
The composition
\begin{multline*}
\xymatrix{
P^{2n(j+1)-1}\ar@{->}[r]^-{y_{j+1}}&
E_0(j)\ar@{->}[r]&
E_0(j)/E_0(j-1)
}\\
\xymatrix{
\Omega^2S^{2n+1}
\ltimes
S^{2nj}
\ar@{->}[r]^-{\xi}&
S^{2nj}\vee S^{2n(j+1)-1}
}
\end{multline*}
induces an integral cohomology epimorphism.
\end{lemma}
\begin{proof}
Since $F_0(j)=F_0(j-1)\cup e^{2mj}$, 
\[
E_0(j)/E_0(j-1)\simeq 
\Omega^2S^{2n+1}
\ltimes
S^{2mj}
\]
by the clutching construction (\ref{eq2.2}). From the
homotopy commutative square
\[
\xymatrix{
E_0(j)\ar@{->}[r]\ar@{->}[d]&E_0(j)/E_0(j-1)\ar@{->}[d]\\
F_0(j)\ar@{->}[r]&S^{2nj}
}
\simeq S^{2nj}\ltimes \Omega^2S^{2n+1}
\]
we see that the composition
\begin{multline*}
\xymatrix{
P^{2nj}\ar@{->}[r]^-{x_j}&E_0(j)\ar@{->}[r]&E_0(j)/E_0(j-1)
}\\
\xymatrix{
\simeq
\Omega^2S^{2n+1}
\ltimes
S^{2nj}
\ar@{->}[r]^-{\xi}&
S^{2nj}\vee S^{2n(j+1)-1}}
\end{multline*}
is an integral cohomology epimorphism. Since the action of $\Omega^2S^{2n+1}$
on $E_0(j)$ corresponds with the action of $\Omega^2S^{2n+1}$ on
\[
E_0(j)/E_0(j-1)\simeq 
\Omega^2S^{2n+1}
\ltimes 
S^{2nj}
,
\]
we can apply \ref{cor3.18}
to see
that the composition in question
\[
\xymatrix{
P^{2n(j+1)-1}\ar@{->}[r]&
\Omega^2S^{2n+1} 
\ltimes 
S^{2nj}
}
\]
induces an epimorphism in integral cohomology.  
\end{proof}
\begin{proof}[Proof of \ref{prop6.3}]
We apply \ref{theor6.1} to the diagram
\[
\xymatrix{
\Omega^2S^{2n+1}\ar@{=}[r]\ar@{->}[d]&
\Omega^2S^{2n+1}\ar@{->}[d]\\
E_0(m-1)\ar@{->}[r]\ar@{->}[d]&
E_0(m)\ar@{->}[d]\\
F_0(m-1)\ar@{->}[r]&
F_0(m)&{\mbox{\hspace*{-31pt}$=F_0(m-1)\cup_{\theta}e^{2mn}$}}
}
\]
where $\theta$ is the attaching map  of the $2mn$ cell. Let
$\chi=x_m\colon P^{2mn}\to E(m)$. Choose an extension $\gamma'$ of
$\gamma_0=\nu_0(m-1)$. Let
$P=P^{2mn}\vee P^{2(m+1)n-1}$ and
\[
u\colon P^{2mn}\vee \xymatrix@C=45pt{P^{2(m+1)n-1}\ar@{->}[r]^-{\enlarge{x_m\vee
y_{m+1}}}&E_0(m);}
\]
let $P_0=S^{2mn-1}\vee S^{2(m+1)n-2}$. The composition
\[
\xymatrix{
P^{2mn-1}\ar@{->}[r]&S^{2mn-1}\ar@{->}[r]&P^{2mn}\ar@{->}[r]^-{x_m}&E_0(m)
}
\]
is $\beta x_m\equiv my_m$ by \ref{prop5.22}.
Since $\nu_0(m-1)y_m$ is null homotopic, the composition
\[
\xymatrix{
P^{2mn-1}\ar@{->}[r]^-{\beta}&P^{2mn}\ar@{->}[r]^-{x_m}&E_0(m)\ar@{->}[r]^-{\gamma'}&BW_n
}
\]
is null homotopic. $\beta$ factors: $P^{2mn-1}\to S^{2mn-1}\to P^{2mn}$, so
the composition
\[
\xymatrix{
S^{2mn-1}\ar@{->}[r]&P^{2mn}\ar@{->}[r]^-{x_m}&E_0(m)\ar@{->}[r]^-{\gamma'}&BW_n
}
\]
is divisible by $p^r$. However $p\cdot \pi_*(BW_n)=0$,
so this
composition is null homotopic. Similarly, since $\beta y_{m+1}\equiv 0$,
the composition
\[
\xymatrix@C=30pt{
S^{2(m+1)n-2}\ar@{->}[r]&
P^{2(m+1)n-1}\ar@{->}[r]^-{y_{m+1}}&
E_0(m)\ar@{->}[r]^-{\gamma'}&
BW_n
}
\]
is null homotopic. Thus the composition
\[
\xymatrix@C=52pt{
P_0\ar@{->}[r]&P\ar@{->}[r]^-{x_m\vee y_{m+1}}
&E_0(m)\ar@{->}[r]^-{\gamma'}&BW_n
}
\]
is null homotopic. However
\begin{multline*}
\hspace*{-12pt}\xymatrix@C=52pt{
P=P^{2mn}\!{}\vee{}\!P^{2(m+1)n-1}\ar@{->}[r]^-{x_m\!{}\vee{}\!y_{m+1}}&
E_0(m)\ar@{->}[r]&
}\\
\xymatrix{
E_0(m)/E_0(m-1)
\cong 
\Omega^2S^{2n+1}
\!{}\ltimes{}\!
S^{2mn}
}
\end{multline*}
induces an integral cohomology epimorphism by \ref{lem6.4} and \ref{lem6.5}. Let $\xi$
be the composition
\[
\xymatrix{
\Omega^2S^{2n+1}
\ltimes
S^{2mn}
\simeq S^{2mn}\vee S^{2mn}\wedge
\Omega^2S^{2n+1}\ar@{->}[r]&S^{2mn}\vee S^{2(m+1)n-1}
}
\]
where the last map is obtained by evaluation on the
double loop space. Clearly the conditions of~\ref{theor6.1} are
satisfied, so we can choose an extension $\nu_0(m)$ of $\nu_0(m-1)$ such
that $\nu_0(m)_*(x_m)=0$ and $\nu_0(m)_*(y_{m+1})=0$. This completes
the induction.  
\end{proof}
\begin{corollary}\label{cor6.6}
There is a retraction $\nu_0\colon E_0\to BW_n$ such that
the composition
\[
\xymatrix{
\Sigma(\Omega G_0\wedge \Omega
G_0)\ar@{->}[r]^-{\Gamma_0}&E_0\ar@{->}[r]^-{\nu_0}&BW_n
}
\]
is null homotopic.
\end{corollary}
\begin{proof}
By \ref{prop6.3} $(\nu_0)_*(x_m)=0$ and $(\nu_0)_*(y_m)=0$ for $m\geqslant 2$. By
\ref{theor5.9}
and \ref{prop5.22},
\[
\{\nu,\dots,\nu,\{\nu,\nu\}_{\times}\}_r
\colon\Sigma G_0\wedge\dots\wedge
\xymatrix{G_0\ar@{->}[r]&E_0\ar@{->}[r]^-{\enlarge{\nu_0}}&BW_n}
\]
is null homotopic for all $j\geqslant 2$. The conclusion follows
from~\ref{prop5.7}.
\end{proof}

\section{Preparation for Induction}\label{subsec6.2}

At this point we present the inductive hypothesis. As
pointed out in section~\ref{subsec2.4}, this will be stronger than 
Proposition~\ref{prop2.12} given there. This strengthening will be
a factorization of the map $\nu_k$ through the map $\tau_k\colon\allowbreak E_k\to\nobreak J_k$.
In case $k=0$, $E_k=J_k$ and $\tau_k$\index{tauk@$\tau_k$|LB} is the identity map, so this
alteration only applies when $k>0$.
\begin{InductiveHypothesis}\label{theor6.7}
There is a map $\gamma_k\colon J_k\to BW_n$
such that the composition
\[
\xymatrix{
\Omega G_k*\Omega
G_k\ar@{->}[r]^-<<<{\Gamma_k}
&E_k\ar@{->}[r]^{\tau_k}&J_k\ar@{->}[r]^->>>>{\gamma_k}
&BW_n
}
\index{Gammak@$\Gamma_{k}$|LB}
\index{gammak@$\gamma_{k}$|LB}
\]
is null homotopic and such that the compositions
\begin{gather*}
\xymatrix{
E^{2np^k-2}_{k-1}\ar@{->}[r]&E_{k-1}\ar@{->}[r]^{e_k}\index{$e_k$|LB}&E_k\ar@{->}[r]^{\tau_k}&J_k\ar@{->}[r]^->>>>{\gamma_k}&BW_n
}\\
\xymatrix{
E^{2np^k-1}_{k-1}\ar@{->}[r]&E_{k-1}\ar@{->}[r]^{\tau_{k-1}}&J_{k-1}\ar@{->}[r]^->>>>{\gamma_{k-1}}&BW_n
}
\end{gather*}
are homotopic, where $\gamma_0=\nu_0$ as constructed in~\ref{cor6.6}.
\end{InductiveHypothesis}

By \ref{prop2.3}, the composition
\[
\xymatrix{
\Omega^2S^{2n+1}\ar@{->}[r]&E_0(1)\ar@{->}[r]^{\nu_0(1)}&BW_n
}
\]
is homotopic to~$\nu$. Consequently \ref{theor6.7}
implies \ref{prop2.12} in case $k=0$.
We will assume that we have constructed $\gamma_i$ for $i<k$,
Having $\gamma_i$ we obtain
$\nu_i=\gamma_i\tau_i$ and construct $a(k)$, $c(k)$
and $\beta_k$ as in \ref{theor4.4} with $\nu_{k-1}\beta_k\sim *$. We then define $D_k$ and~$J_k$ and proceed to
construct $\gamma_k$. The construction is
completed with \ref{theor6.43}.

The procedure in section~\ref{subsec6.1} is a model
for the inductive step. To proceed, we will first need to
prove:\setcounter{equation}{7}
\begin{equation}\label{eq6.8}
H_i(F_k;Z_{(p)})\cong \begin{cases}
Z_{(p)}&\text{if}\ i=2mn\\
0&\text{otherwise}.
\end{cases}
\end{equation}
We will then use an inductive procedure over the skeleta
of~$F_k$ as in section~\ref{subsec6.1}. 
The proof of \ref{eq6.8} will be by induction on~$k$. The
case $k=0$ is \ref{prop6.2}. At this point we will assume \ref{eq6.8}
in case $k-1$.\addtocounter{Theorem}{1}
\begin{proposition}\label{prop6.9}
Let $W_{k-1}$ be the fiber of $\gamma_{k-1}$. Then we
have a homotopy commutative diagram of vertical
fibration sequences
\[
\xymatrix{
T\ar@{=}[r]\ar@{->}[d]&
T\ar@{->}[r]\ar@{->}[d]&
\Omega S^{2n+1}\ar@{->}[d]\\
R_{k-1}\ar@{->}[r]\ar@{->}[d]&
W_{k-1}\ar@{->}[r]\ar@{->}[d]&
F_{k-1}\ar@{->}[d]^-{\sigma_{k-1}}\\
G_{k-1}\ar@{->}[r]&
D_{k-1}\ar@{=}[r]&
D_{k-1} 
}
\]\index{$R_k$|LB}\index{$W_k$|LB}
and two diagrams of fibration sequences
\begin{gather*}
\xymatrix{
S^{2n-1}\ar@{->}[r]\ar@{->}[d]&
\Omega^2S^{2n+1}\ar@{->}[r]\ar@{->}[d]&
BW_n\ar@{=}[d]\\
W_{k-1}\ar@{->}[r]\ar@{->}[d]&
J_{k-1}\ar@{->}[r]^-{\gamma_{k-1}}\ar@{->}[d]&
BW_n\\
F_{k-1}\ar@{=}[r]&
F_{k-1}
}\\
\mbox{}\\
\mbox{}\\
\xymatrix{
S^{2n-1}\ar@{->}[r]\ar@{->}[d]&
T\ar@{->}[r]\ar@{->}[d]&
\Omega S^{2n+1}\\
W_{k-1}\ar@{=}[r]\ar@{->}[d]&
W_{k-1}\ar@{->}[d]&\\
F_{k-1}\ar@{->}[r]&
D_{k-1}\ar@{->}[r]&S^{2n+1}
}
\end{gather*}
\end{proposition}
\begin{proof}
We define $\nu_{k-1}$ to be the composition
\[
\xymatrix{
E_{k-1}\ar@{->}[r]^-{\tau_{k-1}}&
J_{k-1}\ar@{->}[r]^-{\gamma_{k-1}}&
BW_n.
}
\]
From this it follows that we have a commutative
diagram of fibration sequences
\[
\xymatrix{
R_{k-1}\ar@{->}[r]\ar@{->}[d]&
E_{k-1}\ar@{->}[r]^-{\nu_{k-1}}\ar@{->}[d]_{\tau_{k-1}}&
BW_n\ar@{=}[d]\\
W_{k-1}\ar@{->}[r]&
J_{k-1}\ar@{->}[r]^-{\gamma_{k-1}}&
BW_n.
}
\]
Consequently the square
\[
\xymatrix{
R_{k-1}\ar@{->}[r]\ar@{->}[d]&
W_{k-1}\ar@{->}[d]\\
G_{k-1}\ar@{->}[r]&
D_{k-1}
}
\]
is the composition of two pullback
squares
\[
\xymatrix{
R_{k-1}\ar@{->}[r]\ar@{->}[d]&
E_{k-1}\ar@{->}[r]^-{\pi_{k-1}}\ar@{->}[d]_{\tau_{k-1}}&
G_{k-1}\ar@{->}[d]\\
W_{k-1}\ar@{->}[r]&
J_{k-1}\ar@{->}[r]^-{\xi_{k-1}}&
D_{k-1},
}
\]
so it is a pullback and the first diagram commutes up to
homotopy. The second diagram follows from the
definition of~$W_{k-1}$ and the third is a combination of the first
two.
\end{proof}
\begin{proposition}\label{prop6.10}
$\Omega F_{k-1}\simeq S^{2n-1}\times \Omega W_{k-1}$.
\end{proposition}
\begin{proof}
Extending the third diagram of~\ref{prop6.9} to the left yields a diagram
\[
\xymatrix{
\Omega^2 S^{2n+1}\ar@{->}[r]\ar@{->}[d]&
S^{2n-1}\ar@{->}[d]\\
{*}\ar@{->}[r]\ar@{->}[d]&
W_{k-1}\ar@{->}[d]\\
\Omega S^{2n+1}\ar@{->}[r]&
F_{k-1};
}
\]
both of the horizontal maps have degree $p^r$ in dimension $2n$, so $W_{k-1}$
is $4n-2$ connected and the map $S^{2n-1}\to W_{k-1}$ is null homotopic.
From this it follows that $\Omega F_{k-1}\simeq S^{2n-1}\times \Omega
W_{k-1}$.
\end{proof}
\begin{proposition}\label{prop6.11}
The homomorphism 
\[
H^*(F_{k-1};Z_{(p)})\to H^*(W_{k-1};Z_{(p)})
\]
is
onto and
\[
H^j(W_{k-1};Z_{(p)})=\begin{cases}
Z_{(p)}/p^{r+s-1}&\text{if}\ j=2np^s\quad 0<s<k\\
Z_{(p)}/ip^r&\text{if}\ j=2ni,\ \text{otherwise}\\ 
0&\text{otherwise}.
\end{cases}
\]
\end{proposition}
\begin{proof}
Consider the Serre spectral sequence for the $p$-local
homology of the fibration
\[
\xymatrix{
\Omega S^{2n+1}\ar@{->}[r]^-{\delta_{k-1}}&
F_{k-1}\ar@{->}[r]^-{\sigma_{k-1}}&
D_{k-1}
}
\]
Since $E^2_{p,q}$ is only nonzero when $p$ and $q$ are divisible by
$2n$, $E^2_{p,q}=E^{\infty}_{p,q}$. We assume the result~(\ref{eq6.8}) for the case
$k-1$ by
induction. Since $E^{\infty}_{p,q}$ has finite order when $p>0$ and
$H_*(F_{k-1};Z_{(p)})$ is free, all extensions are
nontrivial.
Let $u_i\in H^{2ni}(\Omega S^{2n+1})$
be the generator dual to the $i^{\text{th}}$ power of a chosen
fixed generator in~$H_{2n}(\Omega S^{2n+1})$, so
\[
u_iu_j=\left(\begin{matrix}
i+j\\
i
\end{matrix}\right)u_{i+j}.
\]
Using the nontrivial extensions in the Serre spectral sequence, we can
choose generators $e_i\in H^{2ni}(F_{k-1};Z_{(p)})$ so that
\[
(\delta_{k-1})^*(e_i)=\begin{cases}
p^{r+d-1}u_i&\text{if}\ p^{d-1}\leqslant i <p^d\quad d<k\\
p^{r+k-1}u_i&\text{if}\ i\geqslant p^{k-1}.
\end{cases}
\]
Since $(\delta_{k-1})^*$ is a monomorphism, it is easy to check that
\[
e_1e_{i-p}\begin{cases}
ip^{r-1}e_i&\text{if}\ i=p^s\quad 0<s<k\\
ip^re_i&\text{otherwise}.
\end{cases}
\]
It now follows from the $p$-local cohomology
Serre spectral sequence for the fibration
\[
\xymatrix{
S^{2n-1}\ar@{->}[r]&W_{k-1}\ar@{->}[r]&F_{k-1}
}
\]
that 
\[
d_{2n}(e_{i-1}\otimes u)=\begin{cases}
ip^{r-1}e_i&\text{if}\ i=p^s\quad 0<s<k\\
ip^re_i&\text{otherwise}.
\end{cases}
\]
From this we can read off the cohomology of $W_{k-1}$.
\end{proof}
\begin{proposition}\label{prop6.12}
The homomorphism
\[
\xymatrix{
H^{2np^k}(W_{k-1};Z_{(p)})\ar@{->}[r]&H^{2np^k}(T;Z_{(p)})
}
\]
is nontrivial of order~$p$.
\end{proposition}
\begin{proof}
From \ref{prop6.9} we have a homotopy commutative square
\[
\xymatrix{
T\ar@{->}[r]\ar@{->}[d]&\Omega S^{2n+1}\ar@{->}[d]_{\delta_{k-1}}\\
W_{k-1}\ar@{->}[r]&F_{k-1}
}
\]
to which we apply cohomology
\[
\xymatrix{
H^{2np^k}(T)&
H^{2np^k}(\Omega S^{2n+1})\ar@{->}[l]\\
H^{2np^k}(W_{k-1})\ar@{->}[u]&
H^{2np^k}(F_{k-1})\ar@{->}[u]^{\delta^*_{k-1}}
\ar@{->}[l]
}
\]
which we evaluate
\[
\xymatrix{
Z/ p^{r+k}&
Z_{(p)}\ar@{->}[l]\\
Z/ p^{r+k}\ar@{->}[u]&
Z_{(p)}\ar@{->}[l]\ar@{->}[u]^{p^{r+k-1}} 
}
\]
where the two horizontal arrows are epimorphisms.
The result follows.
\end{proof}
\begin{proposition}\label{prop6.13}
The map $T\to R_{k-1}$ extends to a map
\[
T/ T^{2np^k-2}\to R_{k-1}
\]
such that the composition
\[
\xymatrix{
P^{2np^k}(p^{r+k})=T^{2np^k}/ T^{2np^k-2}\ar@{->}[r]&
T/ T^{2np^k-2}\ar@{->}[r]&
R_{k-1}\ar@{->}[r]&
R_k
}
\]
is null homotopic.
\end{proposition}
\begin{proof}
Since the fibration
\[
\xymatrix{
T\ar@{->}[r]&R_{k-1}\ar@{->}[r]&G_{k-1}
}
\]
is induced from the fibration over $G_k$, we have a homotopy
commutative square
\[
\xymatrix{
T/\Omega G_{k-1}\ar@{->}[r]\ar@{->}[d]&
R_{k-1}\ar@{->}[d]\\
T/\Omega G_{k}\ar@{->}[r]&
R_k.
}
\]
By \ref{theor2.14}(a) the inclusion $\xymatrix{T^{2np^k-2}\ar@{->}[r]&T}$ factors through
$\Omega G_{k-1}$,
this gives a homotopy commutative square
\[
\xymatrix{
T/T^{2np^k-2}\ar@{->}[r]\ar@{->}[d]&
R_{k-1}\ar@{->}[d]\\
T/T^{2np^{k+1}-2}\ar@{->}[r]&
R_k
}
\]
The result follows by restriction to $T^{2np^k}/T^{2np^k-2}$.
\end{proof}
\begin{proposition}\label{prop6.14}
Let $\widetilde{\alpha_k}\index{alphaktilde@$\widetilde{\alpha_k}$|LB}\colon P^{2np^k}(p^{r+k})\to R_{k-1}$ be the
composition of the first two maps in \ref{prop6.13}.
Then the composition
\[
\xymatrix{
P^{2np^k}(p^{r+k})\ar@{->}[r]^-{\widetilde{\alpha}_k}&
R_{k-1}\ar@{->}[r]&
W_{k-1}
}
\]
is nonzero in $p$ local cohomology.
\end{proposition}
\begin{proof}
This follows from \ref{prop6.12} using the diagram
\[
\xymatrix{
P^{2np^k}(p^{r+k})\ar@{->}[r]\ar@{->}[dr]^-{\widetilde{\alpha}_k}&T/T^{2np^k-2}
\ar@{->}[d]&T\ar@{->}[d]
\ar@{->}[l]
\\
&R_{k-1}\ar@{->}[r]&W_{k-1}
}
\]
where the three spaces on the top have isomorphic cohomology in
dimension~$2np^k$.
\end{proof}
\begin{proof}[Proof of \ref{eq6.8}]
We assume the result for $F_{k-1}$ by induction. Since $F_k$
is the total space of a principal fibration over $D_k=D_{k-1}\cup
CP^{2np^k}(p)$
whose restriction to $D_{k-1}$ is $F_{k-1}$, we have by \ref{eq2.2}
\[
F_k/F_{k-1}=P^{2np^k+1}(p)\rtimes \Omega S^{2n+1};
\]
and consequently we have a short exact sequence
\[
\xymatrix{
0\ar@{->}[r]&H_{2nm}(F_{k-1};Z_{(p)})\ar@{->}[r]&H_{2nm}(F_k;Z_{(p)})\ar@{->}[r]&Z/p\ar@{->}[r]&0
}
\]
for $m\geqslant p^k$, while 
\[
H_{2nm}(F_{k-1};Z_{(p)})\simeq H_{2nm}(F_k;Z_{(p)})
\]
for
$m<p^k$. We
will prove that the extension is nontrivial.
It suffices to show that $H_{2np^k}(F_k;Z_{(p)})\simeq Z_{(p)}$ since
the module action of $H_*(\Omega S^{2n+1};Z_{(p)})$ on both $H_*(F_{k-1};Z_{(p)})$ and
$H_*(F_k;Z_{(p)})$ implies the result for all $m > p^k$. If this failed
we would conclude that $H_{2np^k}(F_k;Z_{(p)})\cong Z_{(p)}\oplus Z/p$. This
would imply that the homomorphism
\[
\xymatrix{
H^{2np^k}(F_k;Z_{(p)})\ar@{->>}[r]&H^{2np^k}(F_{k-1};Z_{(p)})
}
\]
is onto. We will show that this is impossible. Suppose then
that this homomorphism is onto and consider the homotopy
commutative diagram
\[
\xymatrix@R=60pt@C=38pt{
&&&&\Omega S^{2n+1}\ar@{->}[d]_{\delta_k}\\
P^{2np^k}(p^{r+k})\ar@{->}[rrrru]^-{L}\ar@{->}[r]^->>>>{\widetilde{\alpha}_k}&
R_{k-1}\ar@{->}[r]\ar@{->}[d]&
W_{k-1}\ar@{->}[r]\ar@{->}[d]&
F_{k-1}\ar@{->}[r]\ar@{->}[d]_{\sigma_{k-1}}&
F_k\ar@{->}[d]_{\sigma_k}\\
&G_{k-1}\ar@{->}[r]&
D_{k-1}\ar@{=}[r]&
D_{k-1}\ar@{->}[r]&
D_k.
}
\]
The map $L$ exists since the composition into $D_k$ factors
through the composition
\[
\xymatrix{
R_{k-1}\ar@{->}[r]&R_k\ar@{->}[r]&G_k\ar@{->}[r]&D_k;
}
\]
thus this composition is null homotopic by \ref{prop6.13}. Now
if 
\[
\xymatrix{H^{2np^k}(F_k;Z_{(p)})\ar@{->}[r]& H^{2np^k}(F_{k-1};Z_{(p)})}
\]
is onto then the entire horizontal
composition
\[
\xymatrix{
H^{2np^k}(F_k;Z_{(p)})\ar@{->}[r]&H^{2np^k}(P^{2np^k}(p^{r+k});Z_{(p)})
}
\]
is nonzero by \ref{prop6.11} and \ref{prop6.14}. But $\delta_k$ factors:
\[
\xymatrix{
\Omega S^{2n+1}\ar@{->}[r]^-{\delta_{k-1}}&F_{k-1}\ar@{->}[r]&F_k
}
\]
and $(\delta_{k-1})^*\colon H^{2np^k}(F_{k-1};Z_{(p)})\to H^{2np^k}(\Omega
S^{2n+1};Z_{(p)})$ is
divisible by $p^{r+k-1}$.
Consequently 
the image of~$\delta^*_k$ is divisible by $p^{r+k-1}$; since
$L^*\delta_k^*$ is nonzero 
and 
\[
H^{2np^k}(P^{2np^k}(p^{r+k});Z_{(p)})\cong Z/p^{r+k},
\]
we conclude that $L^*$
is onto.
This is impossible for then the composition
\[
\xymatrix{
P^{2np^k}(p^{r+k})\ar@{->}[r]^-{L}&\Omega
S^{2n+1}\ar@{->}[r]^-{H_{p^{k-1}}}&\Omega S^{2np^{k-1}+1}
}
\]
would be onto, where $H_{p^{k-1}}$ is the James Hopf invariant. But there is never a map
\[
P^{2mp}(p^{r+k})\to \Omega S^{2m+1}
\]
which is onto in cohomology when $r+k>1$ since the adjoint
\[
\xymatrix{
P^{2np-1}(p^{r+k})\ar@{->}[r]&\Omega^2 S^{2m+1}
}
\]
would also be onto. Such a map would not commute with
the Bockstein. Consequently the extension is nontrivial and the cohomology is free.
\end{proof}
\begin{corollary}\label{cor6.15}
The induced homomorphism
\[
\xymatrix{
H_{2ni}(F_{k-1};Z_{(p)})\ar@{->}[r]&H_{2ni}(F_k;Z_{(p)})
}
\]
is an isomorphism when $i<p^k$ and has degree $p$ if $i\geqslant p^k$.
Furthermore the principal action defines an isomorphism
\begin{gather*}
\xymatrix{
H_{2ni}(\Omega S^{2n+1};Z_{(p)})\otimes H_{2np^k}(F_k;Z_{(p)})\ar@{->}[r]&
H_{2np^k+2ni}(F_k;Z_{(p)})
}\\*[-18pt]
\makebox[\textwidth]{\hfill\qed}
\end{gather*}
\end{corollary}

This completes the first task of this section. Our
second task will be to give a sharper
understanding of the spaces $R_{k-1}$ and, in particular,
$W_{k-1}$. Recall that by \ref{theor2.14}(a), $G_{k-1}$ is a retract of $\Sigma T^{2np^k-2}$.
Consequently we have a sequence of induced fibrations
from \ref{prop6.9}
\[
\xymatrix{
T\ar@{=}[r]\ar@{->}[d]&T\ar@{=}[r]\ar@{->}[d]&T\ar@{->}[d]\\
R_{k-1}\ar@{->}[r]\ar@{->}[d]&Q_{k-1}\ar@{->}[r]\ar@{->}[d]&R_{k-1}\ar@{->}[d]\\
G_{k-1}\ar@{->}[r]&\Sigma T^{2np^k-2}\ar@{->}[r]&G_{k-1}
}
\]
from which we see that $R_{k-1}$
is a retract of $Q_{k-1}$. Using the clutching construction (\ref{prop2.1}), we
see that $Q_{k-1}$ is homotopy equivalent to a pushout with $E=Q_{k-1}$ and
$E_0=T=F$
\[
\xymatrix{
T\ar@{->}[r]&Q_{k-1}\\
T \times T^{2np^k-2}\ar@{->}[r]\ar@{->}[u]^{\varphi}&T\times CT^{2np^k-2}\ar@{->}[u] 
}
\]
where $\varphi$ is the restriction of the action map:
\[
\xymatrix{
T\times T^{2np^k-2}\ar@{->}[r]&T\times\Omega G_{k-1}\ar@{->}[r]^-{a}&T.
}
\]
Since the inclusion $T^{2np^k-2}\to T$ factors through $\Omega G_{k-1}$, the composition
\[
\xymatrix{
T^{2np^k-2}\ar@{->}[r]&T\ar@{->}[r]&Q_{k-1}
}
\]
is null homotopic; by applying \ref{eq2.2} we have an equivalence 
$Q_{k-1}/T^{2np^k-2}\simeq Q_{k-1}\vee \Sigma T^{2np^k-2}$.
However, from the pushout diagram, we have
\[
Q_{k-1}/T\simeq T\ltimes \Sigma T^{2np^k-2}.
\]
Restricting to the $2np^k-2$ skeleton, we get
\[
Q^{2np^k-2}_{k-1}\vee\Sigma T^{2np^k-2}\simeq (T\ltimes \Sigma T^{2np^k-2})^{2np^k-2}
\]
so $Q^{2np^k-2}_{k-1}\simeq \big(T\wedge\Sigma T^{2np^k-2}
\big)^{2np^k-2}$. Now $T\wedge \Sigma T$ is a wedge
of Moore spaces by \ref{theor2.14}(d) in section~\ref{subsec2.3} and only has cells in dimensions congruent
to $-1$, $0$, or~$1$ $\text{mod}\, 2n$. Consequently $Q^{2np^k-2}_{k-1}$ is
a wedge of
Moore spaces, and the largest exponent is the
same as the largest exponent in $\Sigma T^{2np^k-2}$, which is $p^{r+k-1}$.
Since $R_{k-1}$ is a retract of $Q_{k-1}$, we have proved
\begin{proposition}\label{prop6.16}
$R^{2np^k-2}_{k-1}$ is a wedge of $\text{mod}\, p^s$ Moore spaces $P^m(p^s)$ for
$r\leqslant s<r+k$.
\end{proposition}
\begin{remark*}
There are no spheres in this wedge as there are no
Moore spaces in $\Sigma T^{2np^k-2}\wedge T$ of dimension $2np^k-1$.
\end{remark*}
\begin{proposition}\label{prop6.17}
The homomorphism in integral homology
\[
\xymatrix{
H_i(R_{k-1};Z_{(p)})\ar@{->}[r]&H_i(W_{k-1};Z_{(p)})
}
\]
is onto for all $i$ and split for $i<2np^k-1$.
\end{proposition}
\begin{proof}
Since $D_{k-1}$ is the mapping cone of the composition
\[
\xymatrix{
C_{k-1}=\bigvee\limits^{k-1}_{i=1}P^{2np^k-1}(p^{r+i-1})\ar@{->}[r]^-{c}&E_{k-1}\ar@{->}[r]^-{\pi_{k-1}}&G_{k-1},
}
\]
we apply \ref{eq2.2} to the fibrations in~\ref{prop6.9}
\[
\xymatrix{
T\ar@{=}[r]\ar@{->}[d]&T\ar@{->}[d]\\
R_{k-1}\ar@{->}[r]\ar@{->}[d]&
W_{k-1}\ar@{->}[d]\\
G_{k-1}\ar@{->}[r]&
D_{k-1}
}
\]
and we can then describe $W_{k-1}$ by a pushout diagram
\[
\xymatrix{
T\times C(C_{k-1})\ar@{->}[r]&W_{k-1}\\
T\times C_{k-1}\ar@{->}[u]\ar@{->}[r]&R_{k-1}.\ar@{->}[u] 
}
\]
This leads to a long exact sequence
\begin{multline*}
\xymatrix{
\dots\ar@{->}[r]&
\widetilde{H}_i(T\times C_{k-1};Z_{(p)})\ar@{->}[r]&
\widetilde{H}_i(R_{k-1};Z_{(p)})\otimes \widetilde{H}_i(T;Z_{(p)})}\\
\xymatrix{
\ar@{->}[r]&
\widetilde{H}_i(W_{k-1};Z_{(p)})\ar@{->}[r]&
\widetilde{H}_{i-1}(T\times C_{k-1};Z_{(p)}).
}
\end{multline*}
We assert that the homomorphism
\[
\xymatrix{
\widetilde{H}_i(W_{k-1};Z_{(p)})\ar@{->}[r]&
\widetilde{H}_{i-1}(T\times C_{k-1};Z_{(p)})
}
\]
is trivial. By~\ref{prop6.11}, $\widetilde{H}_i(W_{k-1};Z_{(p)})$ is only nontrivial
when $i=2sn-1$
for some $s\geqslant 2$. But $H_{2sn-2}(T\times C_{k-1};Z_{(p)})=0$ since there are
no cells
in these dimensions. Now since $\pi_2\colon T\times C_{k-1}\to T$ is onto
in
homology, we conclude that
$\xymatrix{H_i(R_{k-1};Z_{(p)})\ar@{->}[r]&H_i(W_{k-1};Z_{(p)})}$ is onto.
To show that this is split when $i<2np^k-1$, we note that since
$H_i(W_{k-1};Z_{(p)})$ is cyclic by~\ref{prop6.11}, it suffices to show that the
exponent of $H_i(R_{k-1};Z_{(p)})$ is not larger than the exponent of
$H_i(W_{n-1};Z_{(p)})$ for $i<2np^k-1$. By~\ref{prop6.11}, we have
\[
\exp\left(H_{i-1}(W_{k-1};Z_{(p)})\right)=
\begin{cases}
r+\nu_p(i)& i\neq p^s\\
r+s-1&i=p^s\quad 0<s<k.
\end{cases}
\]
But
\begin{align*}
\exp\left(H_{i-1}(R_{k-1};Z_{(p)})\right)&\leqslant\exp\left(H_{i-1}(Q_{k-1};Z_{(p)})\right)\\
&\leqslant
\exp\left(H_{i-1}\left(T\wedge\Sigma T^{2np^k-2};Z_{(p)}\right)\right)
\end{align*}
when $i-1\leqslant 2np^k-2$. However 
\begin{align*}
\makebox[92.5pt]{}
ip^rH_{2ni-1}(\Sigma T\wedge T;Z_{(p)})&=0\\
\intertext{and}
p^{r+s-1}H_{2np^s-1}(\Sigma T\wedge T;Z_{(p)})&=0.
\makebox[92.5pt]{}
\rlap{\qed}
\end{align*}\noqed
\end{proof}
\begin{proposition}\label{prop6.18}
$W^{2np^k-2}_{k-1}$ is a wedge of Moore spaces.
\end{proposition}
\begin{proof}
Since $H_{i-1}\big(R^{2np^k-2}_{k-1};Z_{(p)}\big)\to
H_{i-1}\big(W^{2np^k-2}_{k-1};Z_{(p)}\big)$
is split onto,
we can find a Moore space in the decomposition of $R^{2np^k-2}_{k-1}$
for each $i$ representing a given generator. This constructs a
subcomplex of $R^{2np^k-2}_{k-1}$ which is homotopy equivalent to
$W^{2np^k-2}_{k-1}$.
\end{proof}
\begin{corollary}\label{cor6.19}
$W^{2np^k-2}_{k-1}\simeq\bigvee\limits_{i=2}^{p^k-1}P^{2ni}(p^{r+n_i})$ where
\[
n_i=\begin{cases}
\nu_p(i)&\text{if}\ i\neq p^s,\quad 0<s<k\\
s-1&\textit{if}\ i=p^s\quad 0<s<k.
\end{cases}
\]
\end{corollary}
\section{The Inductive Construction}\label{subsec6.3}

In this section we perform the inductive step of
constructing a retraction $\gamma_k\colon J_k\to BW_n$ for $k\geqslant 1$.
As in the proof of~\ref{prop6.3}, we will apply \ref{theor6.1} to the
fibration
\[
\xymatrix{
\Omega^2S^{2n+1}\ar@{->}[r]&J_k\ar@{->}[r]&F_k
}
\]
and do an induction over the cells of~$F_k$. At each
stage in this secondary induction we will make choices
to eliminate the obstructions from~\ref{theor5.20}. 

We will construct a map $\gamma_k\colon J_k\to BW_n$ which will
annihilate the level $k$ obstructions. However, $\gamma_{k-1}$ is
not homotopic to the composition
\[
\xymatrix{
J_{k-1}\ar@{->}[r]&J_k\ar@{->}[r]^->>>>{\gamma_k}&BW_n
}
\]
so we will need an extra argument to show that $\gamma_k$
annihilates the obstructions of level less than~$k$.
This is accomplished by some general results
(\ref{theor6.28} and \ref{cor6.37}) which decompose certain
relative Whitehead products. This is applied in 
\ref{lem6.40} to control the obstructions of a lower level.

We presume that $\gamma_{k-1}$
has been constructed such that the composition
\[
\xymatrix{
\Sigma (\Omega G_{k-1}\wedge\Omega G_{k-1})\ar@{->}[r]^-{\Gamma_{k-1}}&
E_{k-1}\ar@{->}[r]^-{\tau_{k-1}}&
J_{k-1}\ar@{->}[r]^-{\gamma_{k-1}}&
BW_n
}
\]
is null homotopic. This defines the fiber $R_{k-1}$ of
$\nu_{k-1}=\gamma_{k-1}\tau_{k-1}$
and we construct $\beta_k$, $a(k)$ and $c(k)$ in accordance with~\ref{theor4.4},
and
$D_k$, $J_k$ and $F_k$ as in \ref{eq5.2}.

We next construct a modification of~\ref{theor4.4} in this context.
\begin{proposition}\label{prop6.20}
There is a homotopy commutative ladder
of cofibration sequences:
\[
\xymatrix{
P^{2np^k}(p)\ar@{->}[r]\ar@{<->}[dddd]_{=}&
P^{2np^k}(p^{r+k})\ar@{->}[r]^-{\sigma}\ar@{->}[d]_{\beta_k}&
P^{2np^k}(p^{r+k-1})\ar@{->}[r]\ar@{->}[d]_{a(k)}&
P^{2np^k+1}(p)\ar@{<->}[dddd]_{=}\\
&
E_{k-1}\ar@{->}[d]_{\tau_{k-1}}&
E_k\ar@{->}[d]_{\tau_k}&\\
&
J_{k-1}\ar@{->}[r]^-{\iota}\ar@{->}[d]_{\eta_{k-1}}&
J_k\ar@{->}[d]_{\eta_k}&\\
&
F_{k-1}\ar@{->}[r]\ar@{->}[d]_{\sigma_{k-1}}&
F_k\ar@{->}[d]_{\sigma_k}&\\
P^{2np^k}(p)\ar@{->}[r]&
D_{k-1}\ar@{->}[r]&
D_k\ar@{->}[r]&
P^{2np^k+1}(p)\\
}
\]
\end{proposition}
\begin{proof}
The upper central square commutes up to homotopy by~\ref{theor4.4}
and~\ref{prop5.3}
and the lower central squares follow from~(\ref{eq5.2}). By a
cohomology calculation, the right hand square commutes up to
homotopy. For the left hand region, observe that the $2np^k$
skeleton of the fiber of the inclusion of~$D_{k-1}$ into~$D_k$ is
homotopy equivalent to $P^{2np^k}(p)$; a standard argument
with cofibration sequences shows that the left hand
vertical map can be taken to be the identity.
\end{proof}
\begin{corollary}\label{cor6.21}
The compositions
\begin{align*}
&\xymatrix{
P^{2np^k}(p^{r+k-1})\ar@{->}[r]^-{a(k)}&E_k\ar@{->}[r]^-{\tau_k}&J_k\ar@{->}[r]^-{\eta_k}&F_k
}\\
&\xymatrix{
P^{2np^k}(p^{r+k-1})\ar@{->}[r]^-{\beta_k}&E_{k-1}\ar@{->}[r]^-{\tau_{k-1}}&J_{k-1}\ar@{->}[r]^-{\eta_{k-1}}&F_{k-1}
}
\end{align*}
induce integral cohomology epimorphisms.
\end{corollary}
\begin{proof}
The first composition is handled by applying
integral cohomology to the right hand region of~\ref{prop6.20}.
For the second composition we consider the upper two
parts of the middle region. The map $\sigma$
has degree~$p$ in~$H^{2np^k}$ as does the map $F_{k-1}\to F_k$
by~\ref{cor6.15}.
Since $r+k\geqslant 2$, this is enough to imply the result.
\end{proof}
\begin{proposition}\label{prop6.22}
$W^{2np^k}_{k-1}$ is a wedge of Moore spaces.
\end{proposition}
\begin{proof}
By \ref{prop6.18}, $W^{2np^k-2}_{k-1}$ is a wedge of Moore spaces.
By~\ref{prop6.11},
it suffices to show that $\beta_k$ factors through $W_{k-1}$:
\[
\xymatrix{
P^{2np^k}(p^{r+k})\ar@{->}[r]^-<<<{\enlarge{\beta_k}} &W^{2np^k}_{k-1} 
}
\]
is an epimorphism. But by \ref{cor6.21}, the composition
\[
\xymatrix{
P^{2np^k}(p^{r+k})\ar@{->}[r]^-{\beta_k}&W^{2np^k}_{k-1}\ar@{->}[r]&F_{k-1}
}
\]
induces an epimorphism in $p$-local cohomology.
\end{proof}

We now filter $F_k$ by skeleta and apply \ref{eq6.8}. As in section~\ref{subsec6.1}, let $F_k(m)$
be the $2mn$ skeleton of~$F_k$, so
\[
F_k(m)=F_k(m-1)\cup e^{2mn}.
\]
Let $J_k(m)$ be the pullback of $J_k$ to $F_k(m)$, so we have
a map of principal fibrations\setcounter{equation}{22}
\begin{equation}\label{eq6.23}
\begin{split}
\xymatrix{
\Omega^2S^{2n+1}\ar@{=}[r]\ar@{->}[d]&
\Omega^2S^{2n+1}\ar@{->}[d]\\
J_k(m-1)\ar@{->}[r]\ar@{->}[d]&J_k(m)\ar@{->}[d]\\
F_k(m-1)\ar@{->}[r]&F_k(m)
}
\end{split}
\end{equation}
and using the clutching construction (\ref{prop2.1}) we see that
\[
J_k(m)/J_k(m-1)\simeq \Omega^2S^{2n+1}\ltimes S^{2mn}
\]
The obstructions that we need to consider at level~$k$ are
the elements $\nu^i\cdot a(k)$ and $\mu\nu^{i-1}\cdot a(k)$ for $i\geqslant
1$ where
$\nu^i$ and $\mu\nu^{i-1}$ generate $Z/p(\nu)\otimes \wedge(\mu)\subset
A_*(D_k)$. (See
\ref{theor5.18} and~\ref{theor5.20}).
\addtocounter{Theorem}{1}
\begin{proposition}\label{prop6.24}
The compositions
\begin{align*}
&
\xymatrix@C=30pt{
P^{2np^k+2ni}\ar@{->}[r]^-{\nu^i\cdot\overline{a(k)}}&J_k\ar@{->}[r]^-{\eta_k}&F_k
}
\\
&
\xymatrix@C=46pt{
P^{2np^k+2ni-1}\ar@{->}[r]^-{\mu\nu^{i-1}\cdot\overline{a(k)}}&J_k(p^k+i-1)\ar@{->}[r]^-q&
\Omega^2S^{2n+1}
\ltimes S^{2(p^k+i-1)n}
}
\end{align*}
induce integral cohomology epimorphisms where $q$ is the
quotient map.
\end{proposition}
\begin{proof}
The first composition is evaluated by \ref{cor6.21} when $i=0$. In case
$i>0$, we use induction on~$i$. We apply \ref{prop3.11}(d) to the diagram
\[
\xymatrix@C=38pt{
J_k\ar@{->}[r]^-{\eta_k}\ar@{->}[d]&F_k\ar@{->}[d]\\
D_k\ar@{=}[r]\ar@{->}[d]_{\varphi'_k}&D_k\ar@{->}[d]\\
S^{2n+1}\{p^r\}\ar@{->}[r]&S^{2n+1}
}
\]
to see that
$\eta_k(\nu^i\cdot\overline{a(k)})\equiv [\nu,\eta_k\nu^{i-1}\cdot\overline{a(k)}]_r$. The
result then
follows from \ref{cor3.18} and~\ref{cor6.15}. The second composition is evaluated
by using 
\ref{cor3.18}
again since
$\mu\nu^i\cdot \overline{a(k)}\equiv [\mu,\nu^{i-1}\overline{a(k)}]_r$.
\end{proof}
\begin{corollary}\label{cor6.25}
The composition
\[
\xymatrix@C=38pt{
P^{2np^j+2ni}\ar@{->}[r]^-{\nu^i\cdot\overline{a(j)}}&W_j\ar@{->}[r]&W_{k-1}
}
\]
induces an integral cohomology epimorphism when $0\leqslant i<p^{j+1}-p^j$ and $j<k$; likewise
the composition
\[
\xymatrix@C=38pt{
P^{2np^j+2n(i+1)-1}\ar@{->}[r]^-{\mu\nu^i\cdot\overline{a(j)}}&
W_j\ar@{->}[r]&W_{k-1} 
}
\]
is nonzero in $\text{mod}\, p$ cohomology in dimension $2np^j+2n(i+1)-1$.
\end{corollary}
\begin{proof}
By the induction hypothesis,
$\nu^i\cdot\overline{a(j)}$ and $\mu\nu^{i-1}\cdot\overline{a(j)}$ are in the kernel
of $\gamma_j$ for $j<k$, so
they factor through~$W_j$. We then construct the diagram 
\[
\xymatrix@C=33pt{
P^{2np^j+2ni}\ar@{->}[r]^-{\nu^i\cdot\overline{a(j)}}&W_j\ar@{->}[r]\ar@{->}[d]&W_{k-1}\ar@{->}[d]\\
&F_j\ar@{->}[r]&F_{k-1}
}
\]
when $i<p^{j+1}-p^j$. Since the map $F_j\to F_{k-1}$ induces an
isomorphism in cohomology in dimensions less than
$2np^{j+1}$, the first result follows from~\ref{prop6.24}. The second
result follows directly from 
the first
 since there is a
map of fibrations
\begin{align*}
&\makebox[112pt]{} 
\xymatrix@C=30pt{
S^{2n-1}\ar@{->}[r]\ar@{->}[d]&\Omega^2S^{2n+1}\ar@{->}[d]\\
W_{k-1}\ar@{->}[r]\ar@{->}[d]&J_{k-1}\ar@{->}[d]\\
F_{k-1}\ar@{=}[r]&F_{k-1}
}\\[-17pt] 
&\makebox[323pt]{}
\rlap{\qed} 
\end{align*}\noqed
\end{proof}

At this point we introduce a simplified notation
analogous to the notation in case $k=0$. We define
$\text{mod}\, p^r$ homotopy classes
\begin{align*}
&x_i(k)\index{$x_i(k)$|LB}\colon P^{2ni}\to J_k\\
&y_i(k)\index{$y_i(k)$|LB}\colon P^{2ni-1}\to J_k
\end{align*}
for $i\geqslant 2$ by the formulas\setcounter{equation}{25}
\begin{align}\label{eq6.26}
\begin{split}
x_i(k)&=\begin{cases}
x_i&\text{if}\ k=0\\
\iota x_i(k-1)&\text{if}\ i<p^k\\
\nu^{i-p^k}\cdot\overline{a(k)}&\text{if}\ i\geqslant p^k
\end{cases}\\
y_i(k)&=\mu \cdot x_{i-1}(k).
\end{split}
\end{align}
Consequently, if $p^j\leqslant i<p^{j+1}\leqslant p^k$, $x_i(k)=x_i(j)$.

We will often not distinguish between $x_i(j)\colon P^{2ni}\to J_j$ and
its composition with $J_j\to J_k$ for $k\geqslant j$. However
\[
\nu\cdot x_i(k)=\begin{cases}
x_{i+1}(k)&\text{if}\ i\neq p^t-1\quad t<k\\
x_{i+1}(t-1)&\text{if}\ i= p^t-1\quad t<k
\end{cases}
\]
We will write $\overline{x}_i(k)$ for $x_i(\ell)$ with $\ell$ unspecified
but $\ell\leqslant k$, so $\nu^d\cdot x_i(k)
=\overline{x}_{i+d}(k)$ and similarly for $\overline{y}_i(d)$.\addtocounter{Theorem}{1}
\begin{corollary}\label{cor6.27}
The compositions
\begin{gather*}
\xymatrix{
P^{2ni}\ar@{->}[r]^-{x_i(k)}&J_k\ar@{->}[r]&F_k
}
\\
\xymatrix{
P^{2ni-1}\ar@{->}[r]^-{y_i(k)}&J_k(i-1)\ar@{->}[r]&
J_k(i-1)/J_k(i-2)
\simeq \Omega^2S^{2n+1}\ltimes S^{2n(i-1)}
}
\end{gather*}
induce integral cohomology epimorphisms for all $i\geqslant 2$.
\end{corollary}
\begin{proof}
This follows from \ref{cor6.25} and \ref{prop6.11}.
\end{proof}

The main technical tool in relating the $x_i(k)$ and $y_i(k)$
with $x_i(k-1)$ and $y_i(k-1)$ will be the following
\begin{Theorem}\label{theor6.28}
Suppose $\zeta \in \pi_{m+1}(D_{k-1};Z/p^r)\simeq \pi_m(\Omega
D_{k-1};Z/p^r)$ with
$m>0$ and the composition
\[
\xymatrix{
\Sigma^2X\ar@{->}[r]^-{\varphi}&W_{k-1}^{2np^k-2}\ar@{->}[r]&J_{k-1}
}
\]
has order $p^r$. Then $\{\zeta,\varphi\}_r\colon P^{m+2}\wedge X\to J_{k-1}$
is congruent to the
sum
\[
\{\zeta,\varphi\}_r\equiv\sum^{p^k-1}_{i=2}\{\zeta,
y_i(k-1)\}_r\alpha_i+\{\zeta, x_i(k-1)\}_r\beta_i
\]
where $\alpha_i\colon P^m\wedge\Sigma^2X\to P^m\wedge P^{2ni-1}$ and
$\beta_i\colon P^m\wedge\Sigma^2X\to P^m\wedge P^{2ni}$.
\end{Theorem}

There are several steps in the proof of~\ref{theor6.28}.  
Under the inductive hypothesis, $x_i(k-1)$ and $y_i(k-1)$
factor through $W_{k-1}$.
\begin{proposition}\label{prop6.29}
The map
\[
\Xi\colon \bigvee\limits^{p^{k-1}}_{i=2}P^{2ni-1}\vee
\xymatrix@C=92pt{P^{2ni}\ar@{->}[r]^-{\enlarge{y_i(k-1)\vee x_i(k-1)}}&
W^{2np^k-2}_{k-1}}
\]
induces a monomorphism $\text{mod}\, p$ cohomology.
\end{proposition}
\begin{proof}
$H^m(W^{2np^k-2}_{k-1})$ is trivial unless $m=2ni$ or
$m=2ni-1$ for $2\leqslant i< p^k$, in which case it is $Z/p$ by~\ref{cor6.19}.
Each of these classes is nontrivial under
either $x_i(k-1)$ or $y_i(k-1)$.
\end{proof}

We seek to compare the maps $x_i(k-1)$ and $y_i(k-1)$ to a natural
basis for $W^{2np^k-2}_{k-1}$. Choose maps $e_i\colon P^{2ni}(p^{r+n_i})\to W^{2np^k-2}_{k-1}$
for $2\leqslant i<p^k-1$ which define the splitting of~\ref{cor6.19}.
\[
e\colon
\bigvee\limits^{p^k-1}_{i=2}\xymatrix{P^{2ni}(p^{r+n_i})\ar@{->}[r]^-{\enlarge{\cong}}&W^{2np^k-2}_{k-1}} 
\]
where $n_i=\nu_p(i)$ if $i\neq p^s$ and $n_i=s-1$ if $i=p^s$.
Now define a map
\[
\Lambda\colon \bigvee\limits^{p^k-1}_{i=2}\xymatrix{P^{2ni}\vee
S^{2ni-1}\ar@{->}[r]^-{}&W^{2np^k-2}_{k-1}}
\]
with components
$e_i\rho^{n_i}\colon P^{2ni}\to W_{k-1}$ and
$e_i\iota_{2ni-1}\colon S^{2ni-1}\to W_{k-1}$.
\begin{proposition}\label{prop6.30}
Suppose $\varphi\colon \Sigma^2 X\to W^{2np^k-2}_{k-1}$ has
order $p^r$. Then there is a congruence
\[
\varphi\equiv\sum^{p^k-1}_{i=2}e_i\rho^{n_i}\alpha_i+e_ii_{2ni-1}\beta_i
\]
for some maps $\alpha_i\colon \Sigma^2X\to p^{2ni}$ and $\beta_i\colon \Sigma^2X\to S^{2ni-1}$.
\end{proposition}
\begin{proof}
This follows directly from \ref{prop5.12} and \ref{theor5.15}, since $W^{2np^k-2}_{k-1}$ is a wedge of even dimensional Moore spaces by~\ref{cor6.19}.
\end{proof}
\begin{corollary}\label{cor6.31}
If $n>1$, there is a diagram
\[
\xymatrix{
\bigvee\limits^{p^k-1}_{i=2}P^{2ni}\vee P^{2ni-1}\ar@{->}[dd]_{F}
\ar@{->}[dr]^-{\Xi}&\\
\ar@{}[r]|(.3){\equiv}
&W^{2np^k-2}_{k-1}\\
\bigvee\limits^{p^k-1}_{i=2}P^{2ni}\vee S^{2ni-1}\ar@{->}[ru]^{\Lambda}&
}
\]
which commutes up to congruence, for some map $F$.
\end{corollary}
\begin{proof}
If $n>1$, $\bigvee\limits^{p^k-1}_{i=2}P^{2ni}\vee P^{2ni-1}$ is a double suspension
whose identity map has order $p^r$. Thus \ref{cor6.31} follows
from \ref{prop6.30}.
\end{proof}

In particular, we obtain a congruence formula by
restricting \ref{cor6.31} to $P^{2ni}$:
\[
x_i(k-1)\equiv e_i\rho^{n_i}+\sum_{2\leqslant j<i}e_j\rho^{n_j}\alpha_j+e_j\iota_{2nj-1}\beta_j
\]
for some maps $\alpha_j\colon P^{2ni}\to P^{2nj}$ and $\beta_j\colon P^{2ni}\to S^{2nj-1}$.
Actually, the coefficient of $e_i\rho^{n_i}$ in this formula is a
unit by a cohomology calculation. We can safely 
assume it is the identity by adjusting the basis $\{e_i\}$.
We intend to use this formula to replace the term $e_i\rho^{n_i}$
in \ref{prop6.30} by $x_i(k-1)$ plus lower dimensional terms. This is a matter of linear substitutions,
and we explain this more clearly in a general context.
Observe that all the spaces in these formulas are
co-$H$ spaces and \ref{cor5.14} applies.
\begin{lemma}\label{lem6.32}
In an additive category, the formulas
\begin{gather*}
x=\sum^{N}_{i=1}a_i\varphi_i+b_i\theta_i\\
x_i=a_i+\sum^{i-1}_{j=1}a_j\varphi_{ij}+b_j\theta_{ij}
\end{gather*}
imply that there is a formula:
\[
x=\sum^N_{i=1}x_i\overline{\varphi}_i+b_i\overline{\theta}_i.
\]
\end{lemma}
\begin{proof}
Use downward induction beginning with replacing
$a_N$ with $x_N$.
\end{proof}

Comparing \ref{prop6.30} with the formula for $x_i(k-1)$ above
and applying \ref{lem6.32}, we get
\begin{corollary}\label{cor6.33}
Suppose $\varphi\colon\Sigma^{2}X\to W^{2np^k-2}_{k-1}$ has order
$p^r$ and $n>1$. Then $\varphi$ is congruent to a sum
\[
\sum^{p^k-1}_{i=2}x_i(k-1)\overline{\varphi}_i +e_i\iota_{2ni-1}\overline{\theta}_i
\]
where $\overline{\varphi}_i\colon \Sigma^2X\to P^{2ni}$ and $\overline{\theta}_i\colon\Sigma^2X\to S^{2ni-1}$.\qed
\end{corollary}
\begin{proof}[Proof of \ref{theor6.28}]
We apply \ref{prop5.19}(a) to \ref{cor6.33} to obtain a
congruence
\[
\{\zeta,\varphi\}_r\equiv
\sum^{p^k-1}_{i=2}
\{\zeta,x_i(k-1)\}_r\Sigma^m\overline{\varphi}_i+
\{\zeta,e_i\iota_{2ni-1}\}_r
\Sigma^m\overline{\theta}_i.
\]
But by \ref{prop3.29},
\[
\{\zeta,e_i\iota_{2ni-1}\}_r=\{\zeta,e_i\iota_{2ni-1}\pi_{2ni-1}\}_r=\{\zeta,d_i\}_r
\]
where 
\[
d_i=e_i\iota_{2ni-1}\pi_{2ni-1}=e_i\beta p^{n_i}\colon P^{2ni-1}\to
P^{2ni}(p^{r+n_i}).
\]
Substituting we~get\setcounter{equation}{33}
\begin{equation}\label{eq6.34}
\{\zeta,\varphi\}_r\equiv
\sum^{p^k-1}_{i=2}\{\zeta,x_i(k-1)\}_r\Sigma^m\overline{\varphi}_i+\{\zeta,d_i\}_r\Sigma^m\overline{\theta}_i.
\end{equation}

We also apply \ref{cor6.33} with $\varphi=y_i(k-1)$ to get
\[
y_i(k-1)\equiv d_i+\sum_{2\leqslant j<i} x_j(k-1)\overline{\varphi}'_j+e_je_{2nj-1}\overline{\theta}'_j
\]
and apply \ref{prop5.19}, we get
\begin{equation}\label{eq6.35}
\{\zeta,y_i(k-1)\}_r\equiv\{\zeta,d_i\}_r+\sum_{2\leqslant
j<i}\{\zeta,x_j(k-1)\}_r\overline{\varphi}''_j+\{\zeta,d_j\}\overline{\varphi}''_j.
\end{equation}

We now apply \ref{lem6.32} to \ref{eq6.34} and \ref{eq6.35} with
$x_i=\{\zeta,y_i(k-1)\}_r$, $a_j=\{\zeta,d_j\}_r$
and $b_j=\{\zeta,x_j(k-1)\}_r$ to obtain\setcounter{equation}{27}
\begin{equation}\label{eq6.28}
\makebox[20pt]{}\{\zeta,\varphi\}_r\equiv\sum^{p^k-1}_{i=2}\{\zeta,y_i-(k-1)\}_r\alpha_i+\{\zeta,x_i(k-1)\}_r\beta_i.
\makebox[20pt]{}
\rlap{\qed} 
\end{equation}\setcounter{equation}{35}\noqed
\end{proof}
In case $\Sigma^2X=P^{\ell}$ we can precompose with $\Delta$
\[
\xymatrix@C=38pt{
P^{m+\ell}\ar@{->}[r]&P^m\wedge P^{\ell}\ar@{->}[r]^-<<<<<{\{\zeta,\varphi\}_r}&J_{k-1}
}
\]
to obtain
\[
[\zeta,\varphi]_r\equiv\sum^{p^k-1}_{i=2}\{\zeta,y_i(k-1)\}_r\alpha'_i+\{\zeta,x_i(k-1)\}_r\beta'_i
\]
and apply \ref{prop5.19}(b) to obtain\setcounter{Theorem}{35}
\begin{corollary}\label{cor6.36}
Suppose $\zeta\colon P^m\to \Omega D_{k-1}$
and
$\varphi\colon P^{\ell}\to W^{2np^k-2}_{k-1}$. Then $[\zeta,\varphi]_r$ is
congruent to a sum
\[
\sum^{p^k-1}_{i=2}[\zeta,y_i(k-1)]_r\alpha_i+[\beta(\zeta),y_i(k-1)]_r\beta_i+[\zeta,x_i(k-1)]_r\gamma_i+[\beta(\zeta),x_i(k-1)]_r\delta_i
\]
for some maps 
\begin{align*}
&\alpha_i\colon P^{m+\ell}\to P^{m+2ni-\ell},\\
&\beta_i\colon P^{m+\ell}\to P^{m+2ni-2},\\
&\gamma_i\colon P^{m+\ell}\to P^{m+2ni}\\
\intertext{and}
&\delta\colon P^{m+\ell}\to P^{m+2ni-1}.
\end{align*}
\end{corollary}
\begin{corollary}\label{cor6.37}
Suppose $\varphi\colon P^{\ell}\to W_{k-1}^{2np^k-2}$. Then
\[
\nu^d\cdot \varphi\equiv\sum^{p^k-1}_{i=2}
x_{i+d}(k-1)\alpha_i+
y_{i+d}(k-1)\beta_i
\]
and
$\mu\nu^{d-1}\cdot\varphi\equiv\sum\limits^{p^k-1}_{i=2}y_{i+d}(k-1)\gamma_i$.
\end{corollary}
\begin{proof}
In case $d=1$ we apply \ref{cor6.36}. The formula simplifies
since $\beta(\nu)=\mu$ and $\mu\cdot y_i(k-1)\equiv 0$, while $\mu\cdot
x_i(k-1)=y_i(k-1)$
and $\nu\cdot x_i(k-1)=x_{i+1}(k-1)$. In case $d>1$ we apply
induction and \ref{prop5.19}(b) with $\alpha =\nu$ and either 
\begin{align*}
\delta&=x_{i+d-1}(k-1)\ \text{and}\ \xi=\alpha_{i-1}\\
\intertext{or}
\delta&=y_{i+d-1}(k-1)\ \text{and}\ \xi=\beta_{i-1}.
\end{align*}
\end{proof}

We will use \ref{cor6.37} to compare the obstructions at
adjacent levels. Recall (\ref{theor4.4}), the map
\[
\xymatrix{P^{2np^k}(p^{r+k})\ar@{->}[r]^-{\enlarge{\beta_k}}&R_{k-1}\ar@{->}[r]^-{\enlarge{}}&W_{k-1};}
\]
$\beta_k$ induces a cohomology epimorphism by \ref{cor6.21}. We apply
\ref{theor6.28} where $\varphi$ is one of the two
maps:\setcounter{equation}{37}
\begin{align}\label{eq6.38}
\begin{split}  
\Delta_1&=\beta_k\rho^k-x_{p^k}(k-1)\colon P^{2np^k}\to W^{2np^k-2}_{k-1}\\
\Delta_2&=\beta_k\delta_k-y_{p^k}(k-1)\colon P^{2np^k-1}\to W^{2np^k-2}_{k-1}.
\end{split}  
\end{align}
The maps $\Delta_1$ and $\Delta_2$ are uniquely defined as maps to
$W_{k-1}$ since each term lies in the kernel of
$\gamma_{k-1}$. The fact that they factor through $W_{k-1}^{2np^k-2}$
follows from \ref{prop6.22} and \ref{prop6.24} (In the case $k=1$
apply \ref{lem6.4} and \ref{lem6.5} in place of \ref{prop6.24}). Note that
by \ref{prop6.20} $\iota \beta_k=a(k)\sigma$ where $\iota\colon J_{k-1}\to J_k$, so
\begin{align*}
\iota\beta_k\rho^k&=a(k)\sigma\rho^k=pa(k) \rho^{k-1}=p\overline{a(k)}\\
\iota\beta_k\delta_k&=a(k)\sigma\delta_k=a(k) \delta_{k-1}=\overline{b(k)}.
\end{align*}
Thus we have
\begin{align}\label{eq6.39}
\begin{split}  
\iota\Delta_1&=p\overline{a(k)}-\iota x_{p^k}(k-1)\\
\iota\Delta_2&=\overline{b(k)}-\iota y_{p^k}(k-1)
\end{split}  
\end{align}\addtocounter{Theorem}{3}

We will filter $J_k$ by spaces $J_k(m)$ as in \ref{eq6.23} and
construct maps
\[
\gamma_k(m)\colon J_k(m)\to BW_n
\]
by induction first on $k$ and then on $m$. By
design the map $\gamma_k(m)$ will annihilate the classes
$x_m(k)$ and $y_{m+1}(k)$. However $\gamma_k(m)$ will not be an extension
of
$\gamma_{k-1}(m)$ and we need to know that the classes
$x_m(k-1)$ and $y_{m+1}(k-1)$ are also in the kernel. For this
purpose we establish the following\setcounter{Theorem}{39}
\begin{lemma}\label{lem6.40}
Suppose $m\geqslant p^k$ and we have constructed
\[
\gamma'\colon J_k(m)\to BW_n
\]
such that the kernel of $(\gamma')_*$ contains the classes
$x_i(s)$ and 
$y_{i+1}(s)$ for $i<m$ and $s<k$. Then the kernel of
$(\gamma')_*$ also contains the classes
\begin{gather*}
x_m(k-1)\\
y_{m+1}(k-1)\\
\nu^j\cdot \overline{b(k)}\\
\mu\nu^j\cdot \overline{b(k)}
\end{gather*}
when $j+p^k\leqslant m+1$.
\end{lemma}
\begin{proof}
Let $d=m-p^k$. Then we have
\begin{align*}
px_m(k)&=p(\nu^d\cdot \overline{a(k)})\\
&\equiv \nu^d\cdot(\iota x_{p^k}(k-1)+\iota\Delta_1)\\
&\equiv\iota x_m(k-1)+\iota\nu^d\cdot \Delta_1\\
&\equiv\iota
x_m(k-1)+\iota\left(\sum^{p^k-1}_{i=2}\overline{x}_{i+d}(k-1)\alpha_i+
\overline{y}_{i+d}(k-1)\beta_i\right)
\end{align*}
by \ref{eq6.26},
\ref{eq6.39}
and \ref{cor6.37}.
Since the identity map of~$BW_n$
has order $p$ and each of the classes $\overline{x}_{i+d}(k-1)$ and
$\overline{y}_{i+d}$ is equal to $x_{i+d}(s)$ and $y_{i+d}(s)$
respectively with $s<k$, we can conclude that
$(\gamma')_*(x_m(k-1))=0$. Similarly,
\begin{align*}
py_{m+1}(k)&=p\left(\mu\nu^d\cdot\overline{a(k)}\right)\\
&\equiv\mu\nu^d\cdot (\iota x_{p^k}(k-1)+\iota\Delta_1)\\
&\equiv\iota y_{m+1}(k-1)+\iota\mu\nu^d\cdot \Delta_1\\
&\equiv\iota
y_{m+1}(k-1)+\sum^{p^k-1}_{i=2}\overline{y}_{i+d+1}(k-1)\gamma_i
\end{align*}
by \ref{cor6.37}. Consequently $(\gamma')_*(y_{m+1}(k-1))=0$.
Likewise, by \ref{eq6.39} we obtain
\begin{align*}
\nu^j\cdot \overline{b(k)}&\equiv\nu^j\cdot(\iota
y_{p^k}(k-1)+\iota\Delta_2)\\
&\equiv \iota
y_{p^k+j}(k-1)+\Sigma\left(\sum^{p^k-1}_{i=2}\overline{x}_{i+j}(k-1)\alpha_i+\overline{y}_{i+j}(k-1)\beta_i\right)
\end{align*}
which lies in $\ker(\gamma')_*$ when $p^k+j\leqslant m+1$.
Similarly
\[
\mu\nu^j\cdot\overline{b(k)}\equiv\iota\left(\sum^{p^k-1}_{i=2}\overline{y}_{i+j+1}(k-1)\gamma_i\right)
\]
which is in
$\ker(\gamma')_*$ when $p^k+j\leqslant m+1$.
\end{proof}
\begin{Theorem}\label{theor6.41}
There are maps $\gamma_k\colon J_k\to BW_n$ such that $\gamma_k$
restricted to $J_k(p^k-1)=J_{k-1}(p^k-1)$ is homotopic to the
restriction of $\gamma_{k-1}$ and such that the compositions
\begin{gather*}
\xymatrix{
P^{2np^k}(p^{r+k-1})\ar@{->}[r]^-<<<{a(k)}&E_k\ar@{->}[r]^{\tau_k}&J_k\ar@{->}[r]^->>>>>{\gamma_k}&BW_n\qquad(k\geqslant 1)\\
P^{2ni}\ar@{->}[r]^{x_i(s)}&J_s\ar@{->}[r]&J_k\ar@{->}[r]^{\gamma_k}&BW_n\\
P^{2ni-1}\ar@{->}[r]^{y_i(s)}&J_s\ar@{->}[r]&J_k\ar@{->}[r]^{\gamma_k}&BW_n
}
\end{gather*}
are null homotopic for $i\geqslant 2$ and $s\leqslant k$.
\end{Theorem}
\begin{proof}
Recall (\ref{eq6.23}) that $F_k=\bigcup F_k(m)$ where
\[
F_k(m)=F_{k-1}(m)\cup e^{2mn}
\]
and $F_k(m)=F_{k-1}(m)$ when $n<p^k$. We have induced
principal fibrations:
\[
\xymatrix{
\Omega^2S^{2n+1}\ar@{=}[r]\ar@{->}[d]&
\Omega^2S^{2n+1}\ar@{->}[d]\\
J_k(m-1)\ar@{->}[r]\ar@{->}[d]&J_k(m)\ar@{->}[d]\\
F_k(m-1)\ar@{->}[r]&F_k(m)
}
\]
We will proceed by induction first on~$k$ and then on~$m$.
The result follows from \ref{prop6.3} when $k=0$. Suppose we
have constructed $\gamma_{k-1}$ and $\gamma_k(m-1)$ is defined
agreeing with $\gamma_{k-1}$ on $J_k(p^k-1)$ and such that
the classes $x_i(s)$ and $y_{i+1}(s)$ are in the kernel of $\gamma_k(m-1)$
when $i<m$ and
$s\leqslant k$. By \ref{prop6.24} the composition
\[
\xymatrix@C=42pt{
P^{2mn}\ar@{->}[r]^->>>>>>>{x_m(k)}&J_k(m)\ar@{->}[r]^{\eta_k}&F_k(m)
}
\]
induces an isomorphism in $\text{mod}\, p$ cohomology in dimension
$2mn$. We apply~\ref{theor6.1} to construct an extension $\gamma'\colon
J_k(m)\to BW_n$
of $\gamma_k(m-1)$.

We first consider the case $m=p^k$. Since $\gamma'$ extends
$\gamma_k(m-1)=\gamma_{k-1}(m-\nobreak 1)$, the classes $x_i(s)$ and $y_{i+1}(s)$
are in the
kernel of $\gamma'$ when $s<k$ and $i<p^k$. By \ref{lem6.40}, the kernel
of $\gamma'$ also contains the classes $x_{p^k}(k-1)$, $y_{p^k+1}(k-1)$,
$\nu\cdot \overline{b(k)}$,
$\mu\cdot\overline{b(k)}$ and $\mu\nu\cdot \overline{b(k)}$.

Now let $P=P^{2np^k}(p^{r+k-1})\vee P^{2n(p^k+1)-1}$ and $u\colon P\to
J_k(p^k)$
be given by $a(k)\vee \mu\cdot\overline{a(k)}$. Let $P_0=S^{2np^k-1}\vee
S^{2n(p^k+1)-2}\subset P$.
We next show that the composition\setcounter{equation}{41}
\begin{equation}\label{eq6.42}
\xymatrix{
P_0\ar@{->}[r]&P\ar@{->}[r]^->>>>{u}&J_k(p^k)\ar@{->}[r]^{\gamma'}&BW_n
}
\end{equation}
is null homotopic. Since
$a(k)\beta=a(k)\sigma\beta\rho=\iota\beta_k\beta\rho$
and $\beta_k$ is in the kernel of $\gamma_{k-1}$, $\gamma'(a(k)\beta)=0$.
This implies
that the composition
\[
\xymatrix@C=30pt{
S^{2np^k-1}\ar@{->}[r]&P^{2np^k}(p^{r+k-1})\ar@{->}[r]^-<<<<{a(k)}&J_k(p^k)\ar@{->}[r]^{\gamma'}&BW_n
}
\]
is divisible by $p^{r+k-1}\geqslant p$. Since $p\cdot \pi_*(BW_n)=0$, this
composition is null homotopic. Similarly
\[
\left(\mu\cdot
a(k)\right)\beta =-\mu\cdot\left[\left(a(k)\rho^{k-1}\right)\beta\right]=-p^{k-1}\mu\cdot
b(k)
\]
which is in the kernel of $\gamma'$. Thus the composition \ref{eq6.42} is
null homotopic and we can apply \ref{theor6.1} to construct a differnt
extension
\[
\gamma_k(p^k)\colon J_k(p^k)\to BW_n
\]
of $\gamma_l(p^k-1)$. We use \ref{prop6.24} to verify that the composition
\[
\kern-4pt\xymatrix@C=14pt{
P\ar@{->}[r]&J_k(p^k)\ar@{->}[r]&J_k(p^k)\!{}\Big/{}\!J_k(p^k\!{}-{}\!1)\simeq 
\Omega^2S^{2n+1}
\!{}\ltimes{}\!
S^{2np^k}
\ar@{->}[r]^->>{\xi}&S^{2np^k}\!{}\vee{}\! S^{2n(p^k+1)-1}
}
\]
induces an integral cohomology epimorphism.
Thus after composing with a homotopy equivalence on 
the wedge of spheres, we see that it is homotopic to the
quotient map $P\to P/P_0$. We apply \ref{theor6.1} to construct
$\gamma_k(p^k)$ with $x_{p^k}(k)$ and $y_{p^k+1}(k)$ in the kernel. Apply
\ref{lem6.40} again, this time to $\gamma_k(p^k)$ to see that
the classes $x_{p^k}(p-1)$ and $y_{p^k+1}(k-1)$ are in the kernel. Then we
repeatedly apply \ref{lem6.40} to restrictions of $\gamma_k(p^k)$ to
$J_s(p^k)$
for $s<p^k$ to see that $x_{p^k}(s-1)$ and $y_{p^k+1}(s-1)$ are
in the kernel for $s<k$.

The case $m>p^k$ is similar. Since $\gamma'$ extends $\gamma_k(m-1)$,
the classes $x_i(s)$ and $y_{i+1}(s)$ are in the kernel of $\gamma'$
when $i<m$ and $s<k$. By \ref{lem6.40} the kernel of $\gamma'$
also contains the classes $x_{m}(k-1)$, $y_{m+1}(k-1)$ and
the classes $\nu^j\cdot \overline{b(k)}$ and $\mu\nu^j\cdot
\overline{b(k)}$
when $j+p^k\leqslant m+1$. We now define $P$ and $u\colon P\to J_k(m)$
\begin{align*}
P=&P^{2mn}\vee P^{2(m+1)n-1}\\
u&=x_m(k)\vee y_{m+1}(k)
\end{align*}
and we calculate
\begin{align*}
x(k)\beta&=\left(\nu^{m-p^k}\cdot \overline{a(k)}\right)\beta\\
&=m\mu\nu^{m-p^k}\cdot \overline{a(k)}+\nu^{m-p^k}\cdot
\left(a(k)\rho^{k-1}\beta\right)\\
&=my_m(k)+p^{k-1}\nu^{m-p^k}\cdot b(k)
\end{align*}
$y_m(k)$ is in the kernel of $\gamma_k(m-1)$ and hence in the
kernel of $\gamma'$ and $\nu^{m-p^k}\cdot\nobreak b(k)$ is in the kernel since
$(m-p^k)+ p^k\leqslant m+1$. Similarly
\[
y_{k+1}(k)\beta=p^{k-1}\mu\nu^{m-p^k}\cdot b(k)
\]
is in $\ker \gamma'$. As before, this implies that the
restriction of $u$ to 
\[
S^{2mn-1}\vee S^{2(m+1)n-2}
\]
is null
homotopic and we can construct the required map $\xi$
satisfying \ref{theor6.1}. This allows for the construction of
$\gamma_k(m)$ which annihilates $x_m(k)$ and
$y_{m+1}(k)$. As before, we apply \ref{lem6.40} to conclude that
all classes $x_i(s)$ and $y_{i+1}(s)$ are annihilated
by $\gamma_k(m)$ when $i\leqslant m$ and $s\leqslant k$.
\end{proof}\setcounter{Theorem}{42}
\begin{Theorem}\label{theor6.43}
Suppose $n>1$. Then the composition
\[
\xymatrix{
\Sigma(\Omega G_k\wedge \Omega
G_k)\ar@{->}[r]^-<<<<{\Gamma_k}&E_k\ar@{->}[r]^{\eta_k}&J_k\ar@{->}[r]^->>>>{\gamma_k}&BW_n
}
\]
is null homotopic
\end{Theorem}
\begin{proof}
By \ref{theor6.41} $\gamma_k$ annihilates
\begin{align*}
x_m(i)&=\nu^{m-p^i}\cdot\overline{a(i)}\\
y_m(i)&=\mu\nu^{m-p^i-1}\cdot\overline{a(i)}
\end{align*}
for each $m\geqslant 2$ and $i\leqslant k$. By \ref{lem6.40} $\gamma_k$
annihilates
\begin{gather*}
\nu^j\cdot\overline{b(k)}\\
\mu\nu^j\cdot\overline{b(k)}
\end{gather*}
for each $j\geqslant 0$. The result follows \ref{prop5.7}, \ref{theor5.9} and
\ref{theor5.20}.
\end{proof}

As in the proof of \ref{prop2.12}, we have a homotopy equivalence
\[
\bigcup_{k\geqslant 0}J_k^{2np^{k+1}-2}\to \bigcup_{k\geqslant 0}J_k=J
\]
and consequently we can construct $\gamma_{\infty}\colon J\to BW_n$ and
redefine $\gamma_k$ as the restriction of $\gamma_{\infty}$ to $J_k$.
Similarly we define $\nu_{\infty}=\gamma_{\infty}\tau_{\infty}$ and we
have
\begin{CompatibilityTheorem}\label{theor6.44}
There are maps $\gamma_k$ and $\nu_k$ such that $\gamma_k\iota\sim\gamma_{k-1}$ and
$\nu_ke\sim \nu_{k-1}$. Furthermore, there are homotopy commutative
diagrams of fibration sequences
\[
\begin{tabular*}{\textwidth}{@{\hspace*{0pt}}l@{\hspace*{5pt}}l@{}}
\xymatrix{
\Omega G_k\ar@{=}[r]\ar@{->}[d]_{h_k}& \Omega
G_k\ar@{->}[d]_{\Omega\varphi_k}\\
T\ar@{->}[r]\ar@{->}[d]&\Omega S^{2n+1}\{p^r\}\ar@{->}[r]^-<<<{H}\ar@{->}[d]&BW_n\ar@{=}[d]\\
R_k\ar@{->}[d]\ar@{->}[r]&E_k\ar@{->}[d]\ar@{->}[r]^{\nu_k}&BW_n\\
G_k\ar@{=}[r]&G_k
}&
\xymatrix{\Omega D_k\ar@{=}[r]\ar@{->}[d]_{h'_k}&\Omega D_k\ar@{->}[d]\\
T\ar@{->}[d]\ar@{->}[r]&\Omega S^{2n+1}\{p^r\}\ar@{->}[d]\ar@{->}[r]^-<<<{H}&BW_n\ar@{=}[d]\\
W_k\ar@{->}[r]\ar@{->}[d]&J_k\ar@{->}[d]\ar@{->}[r]^{\gamma_k}&BW_n\\
D_k\ar@{=}[r]&D_k
}
\end{tabular*}
\]\index{$R_k$|LB}\index{$W_k$|LB}
where the left hand diagram maps into the right hand
diagram. The maps $h_k$ and $h'_k$ are the restrictions of maps
$h\colon \Omega G\to T$ and $h'\colon\Omega D\to T$ and there are
compatible
maps $g_k\colon T^{2np^{k+1}-2}\to \Omega G_k$ with $h_kg_k$ homotopic to
the
inclusion and compatible maps $f_k\!\colon G_k\to \Sigma T^{2np^k}$ with
$\widetilde{g}_kf_k$ homotopic to the identity.
\end{CompatibilityTheorem}

\chapter{Universal Properties}\label{chap7}

The aim of this chapter is to prove theorem B, corollaries C and~D,
and to discuss some applications. We describe an obstruction
theory for the existence and uniqueness of extensions~$\widehat{\alpha}$ and
evaluate the obstructions in some cases. We have, however,
no example of a homotopy Abelian $H$-space and map $\alpha\colon P^{2n}\to
Z$
for which no extension to an $H$-map $\widehat{\alpha}\colon T\to Z$
exists.

\section{Statement of Results}\label{subsec7.1}

In this section we describe the basic result relating
to the existence and uniqueness of an extension of a map
$\alpha\colon P^{2n}\to Z$ to an $H$-map $\widehat{\alpha}\colon T\to Z$.
The proofs of these
results are reserved for the next two sections. We note,
however, that a different argument for the obstruction to
uniqueness was obtained in~\cite{Gra12}.

We begin with some notation. Throughout this
chapter, $Z$ will be an arbitrary homotopy Abelian $H$-space.
We will call an $H$-map $\alpha\colon \Omega G_k\to Z$ proper if
the compositions:
\[
\xymatrix@C=42pt{
P^{2np^i-1}(p^{r+i-1})\vee
P^{2np^i}\left(p^{r+i-1}\right)\ar@{->}[r]^-<<<<<<{\enlarge{a(i)\vee c(i)}}&\Omega
G_i\ar@{->}[r]&\Omega G_k\ar@{->}[r]^{\enlarge{\alpha}}&Z
}
\]
are null-homotopic for each $i$, $1\leqslant i\leqslant k$. Let $G_k(Z)$ be
the Abelian group of all homotopy classes of proper
$H$-maps $\alpha\colon \Omega G_k\to Z$, where $0\leqslant k\leqslant
\infty$ ($G_{\infty}=G$).
Let 
\[
p_k(Z)=p^{r+k-1}\pi_{2np^k-1}(Z;Z/p^{r+k})
\]
by which we mean the subgroup of all elements of
$\pi_{2np^h-1}(Z;Z/p^{r+k})$
which are divisible by $p^{r+k-1}$. Let $[Z_1,Z_2]_H$ be the Abelian
group of all $H$-maps from $Z_1$ to~$Z_2$.

Clearly
\begin{equation}\label{eq7.1}
G_0(Z)=\left[P^{2n},Z\right]=\pi_{2n}(Z;Z/p^r)
\end{equation}\addtocounter{Theorem}{1}
\begin{Theorem}\label{theor7.2}
$\underset{\leftarrow}{\lim} G_k(Z)\cong[T,Z]_H$.
\end{Theorem}
\begin{Theorem}\label{theor7.3}
There is an exact sequence:
\[
\xymatrix{
0\ar@{->}[r]&p_k(\Omega
Z)\ar@{->}[r]^{\enlarge{e}}&G_k(Z)\ar@{->}[r]^{\enlarge{r}}&G_{k-1}(Z)\ar@{->}[r]^{\enlarge{\beta}}&p_k(Z)
}
\]
\end{Theorem}

We will see by example that this sequence is not exact
on the right. In fact, we have no example in which $\beta\neq 0$.
But there are examples in which $p_k(Z)\neq 0$.

\section{Inductive Analysis}\label{subsec7.2}

In this section we will prove Theorem~\ref{theor7.2}. It is a
consequence of the following propositions:
\begin{proposition}\label{prop7.4}
$\underset{\leftarrow}{\lim} G_k(Z)\cong G(Z)$
\end{proposition}
\begin{proposition}\label{prop7.5}
$[T,Z]_H\cong G(Z)$.
\end{proposition}
\begin{proof}[Proof of \ref{prop7.4}:]
We first establish that for any space $X$,\setcounter{equation}{5}
\begin{equation}\label{eq7.6}
\underset{\leftarrow}{\lim}[G_k,X]\cong [G,X]
\end{equation}
The argument here is a special case of the results in~\cite{Gra66}. First
we observe that the restrictions define an epimorphism:
\[
\xymatrix{[G,X]\ar@{->}[r]&\underset{\leftarrow}{\lim} [G_k,X]}
\]
by the homotopy extension property applied inductively for
each~$k$. Suppose, however, that $\alpha\in [G,X]$ lies in the
kernel; i.e., the restrictions:
\[
\xymatrix{G_k\ar@{->}[r]&G\ar@{->}[r]^-{\enlarge{\alpha}}&X}
\]
are all null homotopic. We construct a homotopy
commutative diagram in which the horizontal sequence
is a cofibration sequence:
\[
\xymatrix{
\raisebox{-6pt}{$\smash{\bigvee\limits_{k\geqslant
0}G_k}$}\ar@{->}[r]&G\ar@{->}[r]\ar@{->}[d]_{\enlarge{\alpha}}&\cC(G)\ar@{->}[r]\ar@{-->}[dl]^{\enlarge{\alpha'}}&\raisebox{-6pt}{$\smash{\bigvee\limits_{k\geqslant
0}\Sigma G_k}$}\ar@{->}[r]&\Sigma G\\
&X&&&
}
\]
However by \ref{theor2.14}(i), $\Sigma G$ is a wedge of Moore spaces, so
the map
\[
\bigvee\limits_{k\geqslant 0}\Sigma G_k\to \Sigma G
\]
has a right homotopy inverse. This implies that
\[
\bigvee_{k\geqslant 0}\Sigma G_k\cong\Sigma G\vee \cC(G)
\]
and consequently the map $G\to \cC(G)$ is null
homotopic. It follows that $\alpha$ is null homotopic.\qed

To complete the proof of~\ref{prop7.4}, consider the diagram:
\[
\xymatrix{
[\Omega
G,Z]_H\ar@{->}[r]^->>>>{\enlarge{L}}\ar@{->}[d]_{\enlarge{\varphi}}&\underset{\leftarrow}{\lim}[\Omega
G_k,Z]_H\ar@{->}[d]_{\enlarge{\varphi}}\\
[G,\Sigma
Z]\ar@{->}[r]^->>>>>{\enlarge{\cong}}\ar@{->}[d]_{\enlarge{\psi}}&\underset{\leftarrow}{\lim}[G_k,\Sigma
Z]\ar@{->}[d]_{\enlarge{\psi}}\\
[\Omega G,Z]_H\ar@{->}[r]^->>>>>{\enlarge{L}}&\underset{\leftarrow}{\lim}[\Omega G_k,Z]_H
}
\]
where $\varphi(\alpha)=(\Sigma \alpha)\nu$ and
$\varphi(\beta)=\mu(\Omega\beta)$. Clearly $\psi\varphi=1$.
The middle horizontal homomorphism is an
isomorphism by~\ref{eq7.6}. Since $\psi$ is an epimorphism, $L$ is an
epimorphism and since $\varphi$ is a monomorphism, $L$ is a
monomorphism. Clearly proper $H$-maps in $[\Omega G,Z]_H$
correspond to proper $H$-maps in $[\Omega G_k,Z]_H$ for each~$k$.
\end{proof}

The proof of \ref{prop7.5} will depend on an analysis of~$R$.
We define spaces $A$\index{$A$|LB} and~$C$\index{$C$|LB} and maps $a$ and~$c$ as
follows:
\begin{align*}
a\colon A&=\bigvee\limits_{k\geqslant 1}P^{2np^k}(p^{r+k-1})\to R\to E\\
c\colon C&=\bigvee\limits_{k\geqslant 1}P^{2np^k+1}(p^{r+k-1})\to R\to E
\end{align*}
by the maps $a(k)$ and~$c(k)$ from (\ref{theor4.4}) on the respective
factors. The maps $a$ and$c$ factor through
$R$ by \ref{prop5.3} and \ref{theor6.41}.\addtocounter{Theorem}{1}
\begin{proposition}\label{prop7.7new}
$R\simeq A\vee C\vee \Sigma P$ where the inclusion of~$\Sigma P$
in~$R$ factors through $\Gamma\colon \Omega G\ast\Omega G\to R$.
\end{proposition}
\begin{proof}[Proof of \ref{prop7.5} (based on \ref{prop7.7new}):]
Given an $H$-map $\alpha\colon T\to Z$, the
composition $\beta=\alpha h$:
\[
\xymatrix{
\Omega G\ar@{->}[r]^{\enlarge{h}}&T\ar@{->}[r]^{\enlarge{\alpha}}&Z
}
\]
is a proper $H$-map since $h$ is proper. We construct an inverse:
\[
\xymatrix{
T\ar@{->}[r]^{\enlarge{g}}&\Omega G\ar@{->}[r]^{\enlarge{\beta}}&Z
}
\]
However, since $g$ is not an $H$-map, we need an extra
argument to show that $\beta g$ is an $H$-map. Consider
the diagram:
\[
\xymatrix@C=35pt{
T\times T\ar@{->}[r]^->>>>>{\enlarge{g\times g}}\ar@{->}[d]_{\enlarge{\mu}}&\Omega G\times \Omega
G\ar@{->}[r]^{\enlarge{h\times h}}\ar@{->}[d]_{\enlarge{\mu}}&T\times
T\ar@{->}[d]_{\enlarge{\mu}}\\
T\ar@{->}[r]^{\enlarge{g}}&\Omega G\ar@{->}[r]^{\enlarge{h}}&T
}
\]
in which the left-hand square is not homotopy
commutative. Since the right-hand square and the
rectangle are homotopy commutative, the difference
between the two sides of the left-hand
square
\[
\Delta= (g\mu)^{-1}\mu(g\times g)\colon T\times T\to \Omega G
\]
factors through the fiber of~$h$:
\[
\xymatrix{
\Omega R\ar@{->}[r]&\Omega
G\ar@{->}[r]^{\enlarge{h}}&T\ar@{->}[r]&R\ar@{->}[r]&G. 
}
\]
However since $\beta$ is proper and~$Z$ is homotopy Abelian
the composition:
\[
\xymatrix{
\Omega R\ar@{->}[r]&\Omega G\ar@{->}[r]^{\enlarge{\beta}}&Z 
}
\]
is null homotopic by~\ref{prop7.7new} Thus $\beta\Delta$ is null homotopic
and $\beta g$ is an $H$-map.
\end{proof}

The remainder of this section will be devoted to a
proof of~\ref{prop7.7new}
We begin by
clarifying the relationship between $R$ and~$W$.
\begin{proposition}\label{prop7.8new}
$R\simeq T\ltimes C\vee W$.
\end{proposition}
\begin{proof}
By \ref{prop6.17}, the homomorphism
\[
H_*(R;Z_{(p)})\to H_*(W;Z_{(p)})
\]
is a split epimorphism. Since $R$ is a wedge of Moore spaces
by~\ref{prop6.16} and $W$ is a wedge of Moore spaces 
by~\ref{cor6.19}, the
map
$R\to W$ has a right homotopy inverse. In the proof of~\ref{prop6.17}, a map
of fibration sequences was studied; in the limit, this is of the form
\[
\xymatrix{
T\ar@{=}[r]\ar@{->}[d]&T\ar@{->}[d]\\
R\ar@{->}[r]\ar@{->}[d]&W\ar@{->}[d]\\
G\ar@{->}[r]&D
}
\]
where $D$ is the mapping cone of the map $c\colon C\to G$. It
follows from \ref{prop2.1} that there is a homotopy
pushout diagram
\[
\xymatrix{
T\times \dotcC (C)\ar@{->}[r]&W\\
T\times C\ar@{->}[u]\ar@{->}[r]&R\ar@{->}[u] 
}
\]
where $\dotcC(C)$ is the cone on~$C$. Since the inclusion
of~$T$ in~$R$ is null homotopic, there is an induced map
\[
T\ltimes C\simeq \dotcC T\cup T\times C\to R
\]
whose cofiber is $W$;
\[
T\ltimes C\to R\to W.
\]
Since the map $R\to W$ has a right homotopy inverse, the
result follows.
\end{proof}
\begin{proposition}\label{prop7.5old}
There is a homotopy commutative
diagram
\[
\xymatrix{
T\ltimes C\ar@{->}[r]\ar@{->}[d]_{g\ltimes c}&R\ar@{->}[d]\\
\Omega G\ltimes E\ar@{->}[r]^-<<<<{\pi\Gamma'}&G.
}
\]
\end{proposition}
\begin{proof}
The map $T\ltimes C\simeq T\times C\cup \dotcC T\to R$ is given
as follows. The restriction to $T\times C$ comes from a
trivialization of the pullback over $C$ in the diagram
\[
\xymatrix{
T\ar@{=}[r]\ar@{->}[d]&\ar@{->}[d]\ar@{=}[r]T&T\ar@{->}[d]\\
T\times C\ar@{->}[r]\ar@{->}[d]&R\ar@{->}[r]\ar@{->}[d]&\ar@{->}[d]J\\
C\ar@{->}[r]^{c}&\ar@{->}[r]G&D
}
\]
so the composition $T\times C\to T\times C\cup \dotcC T\to R\to G$ is
given by the map
\[
\xymatrix{
T\times C\ar@{->}[r]^-<<<{\pi_2}&C\ar@{->}[r]^{c}&G.
}
\]
To understand the map $\dotcC T\to R$, consider a general
fibration
\[
\xymatrix{
F\ar@{->}[r]&E\ar@{->}[r]^{\pi}&B.
}
\]
Using the homotopy lifting property, one can construct a map
\[
\xymatrix{
(PB,\Omega B)\ar@{->}[r]^{L}&(E,F)
}
\]
extending the connecting map $\Omega B\to F$ and such that
$\pi L\colon PB\to B$ is endpoint evaluation. In our case we use 
the map $L\colon (PG,\Omega G)\to (R,T)$ together with $g\colon T\to \Omega
G$
to obtain the composition
\[
\xymatrix{
(\dotcC T,T)\ar@{->}[r]^->>>{g}&(C\Omega G,\Omega G)\ar@{->}[r]^{\xi}&(PG,\Omega
G)\ar@{->}[r]^{L}&(R,T).
}
\]
This defines the map $\dotcC T\to R$
and the composition
\[
\xymatrix{
\dotcC T\ar@{->}[r]&R\ar@{->}[r]&G
}
\]
is given by
\[
\xymatrix{
\dotcC T\ar@{->}[r]&\dotcC (\Omega G)\ar@{->}[r]^{\epsilon}&G.
}
\]
According to the definition of $\Gamma'$ (see \ref{eq3.8}), the map 
\[
\xymatrix{
\Omega G\times E\cup PG\ar@{->}[r]^-<<<{\Gamma'}&E\ar@{->}[r]^{\pi}&G}
\]
is given by
\begin{gather*}
\xymatrix{\Omega G\times E\ar@{->}[r]^-<<<{\pi_2}&E\ar@{->}[r]^{\pi}&G
}\\
\xymatrix{PG\ar@{->}[r]^{\epsilon}&G}\\[-15pt]
\makebox[\textwidth]{\hfill\qed}
\end{gather*}\noqed
\end{proof}
\begin{corollary}\label{cor7.6}
$T\ltimes C\simeq T\wedge C\vee C$ and the composition
\[
\xymatrix{
T\wedge C\ar@{->}[r]^{\zeta}&T\ltimes C\ar@{->}[r]&R\ar@{->}[r]&G
}
\]
factors through $\xymatrix{
\Omega G*\Omega G\ar@{->}[r]^-<<<{\Gamma}&E\ar@{->}[r]^{\pi}&C.
}$
\end{corollary}
\begin{proof}
Since $C$ is a suspension, $T\ltimes C\simeq T\wedge C\vee C$.
By~\ref{prop7.5old}, the
composition in question factors up to homotopy as
\[
\xymatrix@C=28pt{
T\wedge C\ar@{->}[r]&T\ltimes C\ar@{->}[r]^->>>>>{g\ltimes c}&\Omega G\ltimes
E\ar@{->}[r]^-<<<<{\Gamma'}&E\ar@{->}[r]^{\pi}&G
}
\]
In the diagram below $T\wedge C$ factors through $\Omega G*\Omega E$
since the middle row is a fibration sequence by~\ref{eq3.17}:
\[
\xymatrix{
T\wedge C\ar@{->}[r]\ar@{->}[d]&T\ltimes C\ar@{->}[r]\ar@{->}[d]_{g\ltimes
c}&C\ar@{->}[d]_{c}\\
\Omega G*\Omega E\ar@{->}[r]\ar@{->}[d]&\Omega G \ltimes
E\ar@{->}[r]\ar@{->}[d]_{\Gamma'}&E\\
\Omega G* \Omega G\ar@{->}[r]^-<<<<{\Gamma}&E\ar@{->}[d]_{\pi}&\\
&G& 
}
\]
since the lower square commutes up to homotopy by \ref{prop3.15}.
\end{proof}

At this point we have $R\simeq(T\wedge C)\vee C\vee W$, where
$T\wedge C\to R\to G$ factors through
\[
\xymatrix{
\Omega G*\Omega G\ar@{->}[r]^-<<<{\Gamma}&E\ar@{->}[r]^{\pi}&G.
}
\]

The remainder of this section will be focused on proving
\begin{proposition}\label{prop7.7} 
$W\simeq A\vee W'$ where 
\[
W'=\bigvee\limits_{i\neq
p^s}P^{2ni}(p^{r+\nu_p(i)}).
\]
Furthermore, there is a factorization:
\[
\xymatrix{
W'\ar@{->}[r]\ar@{->}[d]&R\ar@{->}[d]\\
\Omega G *\Omega G\ar@{->}[r]^-<<<{\Gamma}&G
}
\]
\end{proposition}
The main ingredient for the proof of~\ref{prop7.7} is the
following result:
\begin{proposition}\label{prop7.8} 
Suppose $p^k<m<p^{k+1}$ and $s=\nu_p(m)$.
Then
there is a map $f(m)\index{$f(m)$|LB}\colon P^{2mn}(p^{r+s})\to J_k$ such that

\textup{(a)}\hspace*{0.5em}The composition 
\[
\xymatrix@C=23pt{P^{2mn}(p^{r+s})\ar@{->}[r]^-<<<{\enlarge{f(m)}}&J_k\ar@{->}[r]^{\enlarge{\eta_k}}&F_k
}
\]
induces a cohomology epimorphism.

\textup{(b)}\hspace*{0.5em}There is a factorization:
\[
\xymatrix{
P^{2mn}(p^{r+s})\ar@{->}[r]^-<<<{f(m)}\ar@{->}[d]_{w}&J_k\\
\Omega G_k * \Omega G_k\ar@{->}[r]^-<<<{\Gamma_k}&E_k\ar@{->}[u]^{\tau_k} 
}
\]
\end{proposition}
\begin{proof}[Proof of \ref{prop7.7}]
Suppose 
$p^k<m<p^{k+1}$.
 By \ref{prop7.8}, $\gamma_k f(m)$ is null\linebreak[4]
 \mbox{homotopic}, so
there is a factorization:
\[
\xymatrix{
&J_k\ar@{->}[rd]^{\eta_k}&\\
P^{2mn}(p^{r+s})\ar@{->}[ru]^{f(m)}\ar@{->}[rd]_{\overline{f}(m)}&&F_k\\
&W_k\ar@{->}[uu]\ar@{->}[ur]&
}
\]
By \ref{prop6.11}, $\overline{f}(m)$ induces an isomorphism in integral
homology
in dimension $2mn-1$. Assembling the maps $f(m)$ together for
all $k$ together with the maps $a(k)$ we get a map
\[
\xymatrix{
A\vee\bigvee\limits_{m\neq p^k}P^{2mn}(p^{r+s})\ar@{->}[r]&W
}
\]
which induces an isomorphism in homology.
Thus $W\simeq A\vee W'$ and $W'$ factors through $\Omega G*\Omega G$
by~\ref{prop7.8}.
\end{proof}

The construction of $f(m)$ when $s=0$ is immediate. We
simply set $f(m)=x_m(k)$. In case $s>0$ we need to
construct relative Whitehead products using $G_k$. In
the special case that $m=2p^k$ we will need to use an
$H$-space based Whitehead product.

The proof of \ref{prop7.8} will rely on four lemmas. We will call
an integer $m$ acceptable if there is a map $f(m)\colon P^{2mn}(p^{r+s})\to
J_k$
with $\nu_p(m)=s$, satisfying \ref{prop7.8}(a) and~(b).
\begin{lemma}\label{lem7.9}
Suppose $p^k<m<p^{k+1}$ and $x\colon P^{2mn}(p^{r+t})\to E_k$ is
a map such that the composition
\[
\xymatrix{
P^{2mn}(p^{r+t})\ar@{->}[r]^-<<<{x}&E_k\ar@{->}[r]^{\tau_k}&
J_k\ar@{->}[r]^{\eta_k}&F_k 
}
\]
induces an integral cohomology epimorphism. Suppose
$i\leqslant t\leqslant k$. Then there is a map $f(m+p^i)\colon
P^{2mn+2np^i}(p^{r+i})\to J_k$
satisfying \ref{prop7.8}\textup{(a)} and \textup{(b)}. Consequently if
$\nu_p(m+p^i)=i$,
$m+p^i$ is acceptable.
\end{lemma}
\begin{remark*}
This applies in particular when $m>p^k$ is
acceptable with $t=\nu_p(m)$.
\end{remark*}
\begin{proof}
Write $\Sigma P=P^{2mn}(p^{r+t})$ and let $\lambda_i\colon G_i\to G_k$
be the inclusion when $i\leqslant k$. We will define $f(m+p^i)$
using the relative Whitehead product:
\[
\xymatrix@C=39pt{
G_i\circ\Sigma
P\ar@{->}[r]^-<<<<<<{\{\lambda_i,x\}_r}&E_k\ar@{->}[r]^{\tau_k}&J_k\ar@{->}[r]^{\eta_k}&F_k.
}
\]
Our first task will be to evaluate this in $H_{2n(m+p^i)}$. 
Recall (\ref{eq3.10}) that $\{\lambda_i,x\}_r$ is the composition:
\[
\xymatrix@C=13pt{
G_i\circ \Sigma P\ar@{->}[r]^-{\psi}&\Sigma(\Omega G_i*\Omega \Sigma
P)\!{}\simeq{}\!\Omega G_i\!{}*{}\!\Omega \Sigma
P\ar@{->}[r]^-{\zeta}&\Omega G_i\!{}\ltimes{}\!
\Sigma P\ar@{->}[r]&\Omega G_k\!{}\ltimes{}\! E_k\ar@{->}[r]^-{\Gamma'}&E_k
}
\]
For the first part of this composition, consider the
diagram\setcounter{equation}{13}
\begin{equation}\label{eq7.10}
\begin{tabular*}{302pt}{@{}cp{9pt}c@{}}
\xymatrix@C=23pt@R=26pt{
G_i\circ\Sigma P
\ar@{->}[r]^->>>{\psi}
&
\Sigma \Omega G_i\wedge \Omega\Sigma P\\
\mbox{}&\mbox{}\\
G_i\wedge P
\ar@{->}[r]^->>>>>{\nu\wedge 1}
\ar@{->}[uu]^{\simeq}
&
\Sigma \Omega G_i\wedge P
\ar@{->}[uu]_{1\wedge i}}
&
\parbox[t]{9pt}{\mbox{}\newline\vspace*{-24pt}
$\simeq{}$\newline
\vspace*{37pt}
${}\simeq{}$
}
&
\xymatrix@C=23pt@R=21.25pt{
\Omega G_i*\Omega\Sigma P
\ar@{->}[rd]^{\zeta}
&\\
&
\Omega G_i\ltimes \Sigma P\\
\Omega G_i*P
\ar@{->}[uu]_{1*i}
\ar@{->}[ur]_{\zeta'}
}
\end{tabular*}
\end{equation}
where $\nu$ is the co-$H$ space structure map on~$G_i$. The left
hand square is homotopy commutative by \ref{prop3.2} and the
right hand triangle defines $\zeta'$ (see footnote to~\ref{eq3.17}).

Choose a generator $b_i\in H_{2np^i+1}(G_{i})\cong Z/p$ and
let $\sigma b_i\in H_{2np^i}(\Omega G_i)$ be the image of this generator
under $\nu_*$. Choose a generator $f\in H_{2mn-1}(P)$.
Then by the above diagram we have
\[
\zeta_*\psi_*(b_i\otimes f)=\sigma b_i\otimes 1 \otimes f\in
H_{2mn+2np^i}(\Omega G_i\ltimes \Sigma P),
\]
by applying the commutative square in the proof of \ref{eq3.17}.
We now construct a diagram where the two right hand
squares are homotopy commutative by \ref{eq3.8}, \ref{prop3.11}(b) and
\ref{prop3.11}(d):
\begin{equation}\label{eq7.11}
\begin{split}
\xymatrix@C=37pt{
\Omega G_i\ltimes \Sigma P\ar@{->}[r]^{\Omega\lambda_i\ltimes x}&\Omega G_k\ltimes
E_k\ar@{->}[r]^{\Gamma'}
\ar@{->}[d]\ar@{->}[r]^{\Gamma'}&E_k\ar@{->}[d]^{\eta_i\tau_k}\\
&\Omega D_k\ltimes F_k\ar@{->}[r]^-<<<<<<<{\Gamma'}&F_k\\
&\Omega D_k\times F_k\ar@{->}[u]\ar@{->}[r]\ar@{->}[u]&\Omega S^{2n+1}\times
F_k\ar@{->}[u]^{a} 
}
\end{split}
\end{equation}
To evaluate $\eta_k\tau_k\{\lambda_i,x\}_r$ in $\bmod\, p$ homology we observe
that the image of $\sigma b_i\otimes\mid\otimes f$ is
$a_*(\alpha\otimes\beta)$ where
$\alpha\in H_{2np^i}(\Omega S^{2n+1})$ is the image of $\Omega b_i$ under
the homomorphism
induced by the composition
\[
\xymatrix{
\Omega G_i\ar@{->}[r]& \Omega G_k\ar@{->}[r]& \Omega D_k\ar@{->}[r]& \Omega S^{2n+1}
}
\]
and $\beta$ is the image of $1\otimes f$ under the homomorphism
\[
\xymatrix{
\Sigma P\ar@{->}[r]^{x}&E_k\ar@{->}[r]^{\eta_k\tau_k}&F_k.
}
\]
By hypothesis $\beta\in H_{2mn}(F_k)$ is a
generator.\addtocounter{Theorem}{2}
\begin{lemma}\label{lem7.12}
The image of $\sigma b_i\in H_{2np^i}(\Omega G_k)$ under the
homomorphism
\[
\Omega G_k\to \Omega D_k\to \Omega S^{2n+1}
\]
is a unit multiple of the generator.
\end{lemma}
\begin{proof}
By \ref{eq5.2}, the composition in question is homotopic to
the composition:
\[
\xymatrix{
\Omega G_k\ar@{->}[r]^->>>{\Omega\varphi_k}&\Omega
S^{2n+1}\{p^r\}\ar@{->}[r]&\Omega S^{2n+1}.
}
\]
By \ref{theor6.44}, this factors as
\[
\xymatrix{
\Omega G_k\ar@{->}[r]^{h_k}&T\ar@{->}[r]&\Omega S^{2n+1}.
}
\]
However the composition
\[
\xymatrix{
G_k\ar@{->}[r]^{\nu}&\Sigma\Omega G_k\ar@{->}[r]^{\Sigma h_k}&\Sigma T
}
\]
induces a monomorphism in $\bmod\, p$ homology, so
$(h_k)_*(\sigma b_i)$ is a  nonzero generator which is mapped to a
unit multiple of $v^{p^i}$ under the map $T\to \linebreak[4]\Omega S^{2m-1}$.  
\end{proof}

We now complete the proof of~\ref{lem7.9}. Since $i\leqslant t$, we can
find a map
\[
P^{2mn+2np^i}(p^{r+i})\to G_i\circ P^{2mn}(p^{r+t})
\]
which induces an isomorphism in $\bmod\, p$ homology in
dimension $2mn+2np^i$ using \ref{theor2.14}(i) and \ref{prop3.25}. Then let
$f(m+p^i)$ be the composition:
\[
\xymatrix{
P^{2mn+2np^i}(p^{r+i})\ar@{->}[r]&G_i\circ
P^{2mn}(p^{r+t})\ar@{->}[r]^-<<<<{\{\lambda_i,x\}_r}&E_k\ar@{->}[r]^{\tau_k}&J_k.
}
\]
By \ref{eq7.10}, \ref{eq7.11}, and \ref{lem7.12}, $\eta_kf(m+p^i)$ induces
an
epimorphism in $\bmod\, p$ cohomology and consequently
in integral cohomology as well. By \ref{prop3.15}, we have a
homotopy commutative diagram
\[
\xymatrix{
\Omega G_i*\Omega\Sigma P\ar@{->}[r]^{\zeta}\ar@{->}[d]&\Omega G_i\ltimes \Sigma
P\ar@{->}[d]\\
\Omega G_k*\Omega E_k\ar@{->}[r]^{\zeta}\ar@{->}[d]&\Omega G_k\ltimes
E_k\ar@{->}[r]\ar@{->}[d]_{\Gamma'}&\Omega D_k\ltimes
J_k\ar@{->}[d]_{\Gamma'}\\
\Omega G_k*\Omega G_k\ar@{->}[r]^-<<<<<{\Gamma}&E_k\ar@{->}[r]&J_k
}
\]
so $f(m+p^i)$ factors through $\Gamma$.
\end{proof}

We need to consider the case $m=2p^k$ separately. We will construct
\[
\xymatrix{
P^{4np^k}(p^{r+k})\ar@{->}[r]^-<<<<{f(2p^k)}&W.
}
\]
It is the first example of a mod $p^{r+k}$ Moore space in
$W$ by \ref{cor6.19}.
\begin{proposition}\label{prop7.13}
$2p^k$ is acceptable.
\end{proposition}

We will construct $f(2p^k)$ as an $H$-space based
Whitehead product. Let $\overline{G}_k$ be the $2np^k$ skeleton of~$G_k$.
Then $\overline{G}_k\circ G_k$ has dimension $4np^k$.
\begin{lemma}\label{lem7.14}
$\overline{G}_k\circ G_k$ is a wedge of Moore spaces.
\end{lemma}
\begin{proof}
Consider the cofibration sequence
\[
\xymatrix{
\overline{G}_k\circ G_k\ar@{->}[r]&G_k\circ
G_k\ar@{->}[r]^->>>>>{Q}&S^{2np^k+1}\circ G_k\simeq \Sigma^{2np^k}G_k.
}
\]
Since $G_k\circ G_k$ is a wedge of Moore spaces by~\ref{theor2.14}(l) and
\ref{prop3.21}, it suffices to show that this cofibration splits.
Choose a basis $\{a_i,b_i\}$ for $H_*(G_k)$ with $0\leqslant i\leqslant k$
and $\beta^{(r+k)}(b_i)=a_i\neq 0$. Consider the classes $b_k\circ b_i$
corresponding
to $b_k\otimes b_i$. We can construct a basis for $H_*(G_k\circ G_k)$
which contains the elements $b_k\circ b_i$ and
$\beta^{(r+i)}(b_k\circ b_i)$. Since $G_k\circ G_k$ is a wedge of Moore
spaces
we can construct maps of Moore spaces into $G_k\circ G_k$
realizing these basis elements and from this a
right homotopy inverse to~$Q$. Thus
\[
\makebox[98pt]{}G_k\circ G_k\simeq\overline{G}_k\circ G_k\vee
\Sigma^{2np^k}G_k.\makebox[98pt]{}\qed
\]\noqed
\end{proof}

Now consider the principal fibration sequence:
\[
\xymatrix{
\Omega \overline{G}_k\times \Omega G_k\ar@{->}[r]&
\Omega \overline{G}_k * \Omega G_k\ar@{->}[r]&
\overline{G}_k\vee G_k.
}
\]

We will study the integral homology Serre spectral
sequence of this fibration. The principal action
defines a module structure
\[
E^r_{o,q'}\otimes E^r_{p,q}\to E^r_{p,q+q'}.
\]
For $i=1,2$, let $a(i)\in H_{2np^k}(\overline{G}_k\vee G_{k};Z)$ be the
image
of~$a_k$ in the $i^{\text{th}}$ axis. Let $\sigma a\in
H_{2np^k-1}(\Omega\overline{G}_k;Z)$ be the
desuspended image under $\nu_*$ of $a_k$, and $\sigma a(i)$ the
image of $\sigma a$ in the $i^{\text{\itshape th}}$ axis in
$H_{2np^k-1}(\Omega\overline{G}_k\times \Omega G_k;Z)$.
Using the universal coefficient theorem, we define a
monomorphism
\[
H_p(\overline{G}_k\vee G_k;Z)\otimes H_q(\Omega\overline{G}_k\times\Omega
G_k;Z)\to E^2_{p,q}.
\]
Then $d_{2np^k}(a(i)\otimes 1)=1\otimes \sigma a(i)$. Define $\xi\in
E^2_{2np^k,2np^k-1}$ by
\[
\xi=a(1)\otimes\sigma a(2)+a(2)\otimes \sigma a(1).
\]
$\xi$ has order $p^{r+k}$ and $d_{2np^k}(\xi)=0$. Since $E^2_{p,q}=0$ when
$p>{2np^k+1}$, $\xi$ survives to an element of order $p^{r+k}$ in
$E^{\infty}$.
By \ref{lem7.14}, \ref{prop3.21} and \ref{theor2.14}(l), $\Omega \overline{G}_k*\Omega
G_k$ is a wedge of
Moore spaces and by an easy homology calculation,
every element in the homology has order at most $p^{r+k}$.
Consequently $\xi$ converges to an element $[\xi]$ of order $p^{r+k}$ in
$H_{4np^k-1}(\Omega \overline{G}_k*\Omega G_k;Z)$. Since $\Omega
\overline{G}_k*\Omega G_k$ is a wedge of
Moore spaces, there is a map
\[
\varphi\colon P^{4np^k}(p^{r+k})\to \Omega \overline{G}_k*\Omega G_k
\]
whose homology image contains a class
$\eta\in H_{4np^k}(\Omega \overline{G}_k*\Omega G_k)$ with
$\beta^{(r+k)}(\eta)$ equal to the
$\bmod\, p$ reduction of~$[\xi]$.

Now let $\gamma$ be the composition:
\[
\xymatrix{
\Omega \overline{G}_k*\Omega G_k\ar@{->}[r]&\Omega G_k*\Omega
G_k\ar@{->}[r]^-<<<{\Gamma_k}&E_k\ar@{->}[r]^{\tau_k}&J_k\ar@{->}[r]^{\eta_k}&F_k
}
\]
\begin{lemma}\label{lem7.15}
The composition
\[
\xymatrix{
P^{4np^k}(p^{r+k})\ar@{->}[r]^{\varphi}&\Omega\overline{G}_k*\Omega
G_k\ar@{->}[r]^-<<<{\gamma}&F_k
}
\]
induces an isomorphism in $\bmod\, p$ homology in dimension
$4np^k$.
\end{lemma}
\begin{remark*}
This implies that $\gamma\varphi$ induces an
epimorphism in integral cohomology.
\end{remark*}
\begin{proof}
Consider the two diagrams of principal fibrations
\[
\begin{tabular*}{\textwidth}{@{}l@{\hspace*{62pt}}l@{}}
\xymatrix{
\Omega\overline{G}_k\times \Omega G_k\ar@{=}[r]\ar@{->}[d]&
\Omega\overline{G}_k\times \Omega G_k\ar@{->}[d]\\
L\ar@{->}[r]\ar@{->}[d]&\Omega\overline{G}_k*\Omega G_k\ar@{->}[d]\\
G_{k-1}\vee G_{k-1}\ar@{->}[r]&\overline{G}_k\vee G_k
}
&\xymatrix{
\Omega^2S^{2n+1}\ar@{=}[r]\ar@{->}[d]& \Omega^2S^{2n+1}\ar@{->}[d]\\
F_{k-1}\ar@{->}[r]\ar@{->}[d]&F_k\ar@{->}[d]\\
D_{k-1}\ar@{->}[r]&D_k
}
\end{tabular*}
\]
where $L$ is the total space of the induced fibration. The
homotopy commutative square
\[
\xymatrix{
\overline{G}_k\vee G_k\ar@{->}[r]\ar@{->}[d]&D_k\ar@{->}[d]\\
\overline{G}_k\times G_k\ar@{->}[r]&S^{2n+1}
}
\]
induces a map from the left hand pair of fibrations to
the right hand pair.
We apply \ref{prop2.1} to obtain the following homotopy
commutative diagram:
\[
\xymatrix{
\Omega\overline{G}_k*\Omega
G_k\Big/L\ar@{->}[r]^{\gamma}\ar@{->}[d]_{\simeq}&F_k\Big/F_{k-1}\ar@{->}[d]_{\simeq}\\
(S^{2np^k}\vee P^{2np^k+1}(p^{r+k}))\rtimes \Omega \overline{G}_k\times
\Omega G_k\ar@{->}[r]^-<<<<{\gamma'}&\Omega S^{2n+1}\ltimes P^{2np^k+1}(p)
}
\]
Let $e\in H_{2np^k}(S^{2np^k})$ be the image of $a(1)$ and $f\in
H_{2np^k}(P^{2np^k+1}(p^{r+k}))$
be the image of $a(2)$ under the quotient map
\[
\xymatrix{
\overline{G}_k\vee G_k\ar@{->}[r]&\overline{G}_k\vee
G_k\Big/G_{k-1}\vee G_{k-1}\simeq S^{2np^k}\vee P^{2np^k+1}(p^{r+k}).
}
\]
 Choose $g$ with $\beta^{(r+k)}(g)=f$. Then the image
of $[\xi]$ under the left hand equivalence is $e\otimes \sigma
a(2)+f\otimes \sigma a(1)$.
The image of $\eta$ is a class $\eta'$ such that
\[
\beta^{(r+k)}(\eta')=e\otimes \sigma a(2)+f\otimes \sigma a(1).
\]
At this point in the Bockstein spectral sequence there are very
few classes left and the only possibility for $\eta'$ is
\[
\eta'=e\otimes \sigma b(2)+g\otimes \sigma a(1)
\]
where $\beta^{(r+k)}b(2)=a(2)$.

The components of the map $\gamma'$ are the inclusion
$\xymatrix{S^{2np^k}\ar@{->}[r]&P^{2np^k+1}(p)}$ and
\[
\sigma^{r+k-1}\colon
\xymatrix{P^{2np^k+1}(p^{r+k})\ar@{->}[r]&P^{2np^k+1}(p).}
\]
By
\ref{lem7.12}, the image of $\sigma b(2)$
is $v^{p^k}\neq 0$. We conclude that the image of $\eta'$ is nonzero
\[
\xymatrix@C=12pt{
P^{4np^k}(p^{r+k})\ar@{->}[r]^{\varphi}&\Omega\overline{G}_k*\Omega
G_k\ar@{->}[r]^-<{\gamma}&F_k\ar@{->}[r]&F_k\big/F_{k-1}\simeq
P^{2np^k+1}(p)\rtimes \Omega S^{2n+1}
}
\]
in $\text{mod}\, p$ homology,
from which the conclusion follows.
\end{proof}
\begin{proof}[Proof of \ref{prop7.13}]
Let $f(2p^k)$ be the composition
\[
\xymatrix{
P^{4np^k}(p^{r+k})\ar@{->}[r]^->>>{\varphi}&\Omega\overline{G}_k*\Omega
G_k\ar@{->}[r]^-<<<{\Gamma_k}&E_k\ar@{->}[r]^{\eta_k}&J_k.
}
\]
The result follows from \ref{lem7.15}.
\end{proof}
\begin{proof}[Proof of \ref{prop7.8}]
Write $m=e_0+e_1p+\dots+e_kp^k$ where $0\leqslant e_i<p$.
Let $\ell(m)$ bet the number of coefficients $e_i$ which are nonzero.
We first deal with the case $\ell(m)=1$. Then $m=e_kp^k$ with
$1<e_k<p$ since $p^k<m<p^{k+1}$. The case $e_k=2$ is \ref{prop7.13}. If
$e_k>2$ we apply \ref{lem7.9} with $x=f((e_k-1)p^k)$ and $i=t=k$,
to establish this case by induction. In case $\ell(m)=2$, we
first consider the case $m=p^i+e_kp^k$ for $i<k$. In case $e_k\geqslant 2$
we apply \ref{lem7.9} with $x=f(e_kp^k)$ and $t=k$. In case $e_k=1$
we apply \ref{lem7.9} with $x=a(k)$. In this case $t=k-1\geqslant i$.
We now consider the general case with $\ell(m)=2$. In this
case $m=e_ip^i+e_kp^k$ with $i<k$. We do this by induction
on $e_i$ with $e_i<p$ as before. The general case is by
induction on $\ell(m)$ and then induction on the
coefficient of the least power of~$p$ in the expansion using
\ref{lem7.9} repeatedly.
\end{proof}
\begin{proof}[Proof of \ref{prop7.7new}]
By \ref{prop7.8new} and \ref{cor7.6} $R=T\ltimes C\vee W\simeq T\wedge C\vee
C\vee W$
where $T\wedge C$ is a wedge of Moore spaces which factors
through~$\Gamma$. By \ref{prop7.7} $W\simeq A\vee W'$ where $W'$ is a wedge
of
Moore spaces which factors through~$\Gamma$.
Set $\Sigma P=W'\vee C\wedge W$.
\end{proof}

\setcounter{Theorem}{16}
\section{The exact sequence}\label{subsec7.3new}

In this section we will define the homomorphisms $e$, $r$, and~$\beta$
and prove Theorem~\ref{theor7.3}.

We define 
\[
e\colon p_k(\Omega Z)\to G_k(Z)
\]
as follows. Let $\phi\colon P^{2np^k-1}(p^{r+k})\to \Omega Z$, and extend
the
adjoint 
\[
\widetilde{\phi}\colon P^{2np^k}(p^{r+k})\to Z
\]
to an $H$-map
\[
\widehat{\phi}\colon \Omega P^{2np^k+1}(p^{r+k})\to Z.
\]
Since $\phi$ is divisible by $p^{r+k-1}$, so is $\widehat{\phi}$ and
consequently the
composition
\[
\xymatrix{
\Omega G_k\ar@{->}[r]^->>>>{\enlarge{\Omega \pi'}}&\Omega
P^{2np^k+1}(p^{r+k})\ar@{->}[r]^->>>>{\enlarge{\widehat{\phi}}}&Z
}
\]
is a proper $H$-map. We define $e(\phi)=\widehat{\phi}\Omega\pi'$. $e$~is
clearly a homomorphism. To see that $e$ is a
monomorphism, we suppose $e(\phi)$ is null homotopic.
Since $\pi'$ is a co-$H$ map, we have a homotopy commutative
diagram:
\[
\xymatrix@C=30pt{
G_k\ar@{->}[r]^{\enlarge{\nu}}\ar@{->}[dd]_{\enlarge{\pi'}}&\Sigma \Omega
G_k\ar@{->}[dr]^{\enlarge{\Sigma e(\phi)}}\ar@{->}[dd]_{\enlarge{\Sigma\Omega\pi'}}&\\
&&\Sigma Z\\
P^{2np^k+1}(p^{r+k})\ar@{->}[r]&\Sigma \Omega
P^{2np^k+1}(p^{r+k})\ar@{->}[ru]_{\enlarge{\Sigma\widehat{\phi}}}&
}
\]
Since $e(\phi)$ is null homotopic, the upper composition is
null homotopic, so the lower composition factors over the
cofiber of $\pi'$:
\[
\xymatrix{
G_k\ar@{->}[r]^->>>>{\pi'}&P^{2np^k+1}(p^{r+k})\ar@{->}[r]&\Sigma
G_{k-1}\ar@{->}[r]&\Sigma G_k.
}
\]
But since $\Sigma G_k$ splits as a wedge of Moore spaces,
the map
\[
P^{2np^k+1}(p^{r+k})\to \Sigma G_{k-1}
\]
is null homotopic. It follows that the composition:
\[
\xymatrix{
P^{2np^k+1}(p^{r+k})\ar@{->}[r]&\Sigma \Omega
P^{2np^k+1}(p^{r+k})\ar@{->}[r]^-<<<<{\enlarge{\Sigma\widehat{\phi}}}&\Sigma Z
}
\]
is null homotopic. Since $Z$ in an $H$ space, we
conclude that
\[
\xymatrix{
P^{2np^k}(p^{r+k})\ar@{->}[r]&\Omega
P^{2np^k+1}(p^{r+k})\ar@{->}[r]^-<<<<{\enlarge{\widehat{\phi}}}&Z
}
\]
is null  homotopic. Since $\widehat{\phi}$ is an $H$-map,
$\widetilde{\phi}$ is
null homotopic and consequently $\phi$ is as well. Thus $e$ is a
monomorphism.

The map
\[
r\colon G_k(Z)\to G_{k-1}(Z)
\]
is given by restriction. Clearly $re=0$. Suppose
$r\alpha=0$ for some proper $H$-map $\alpha\colon \Omega G_k\to Z$.
We construct an extension $\overline{\phi}$ in the diagram:
\[
\xymatrix{
G_{k-1}\ar@{->}[r]\ar@{->}[d]_{\enlarge{\nu}}&G_k\ar@{->}[r]\ar@{->}[d]_{\enlarge{\nu}}&P^{2np^k+1}(p^{r+k})\ar@{-->}[ddl]_{\enlarge{\overline{\phi}}}\\
\Sigma \Omega
G_{k-1}\ar@{->}[r]\ar@{->}[dr]_{\enlarge{\ast}}
&\Sigma\Omega G_k
\ar@{->}[d]_{\enlarge{\Sigma\alpha}}
&\\
&\Sigma Z&
}
\]
We include the loops on the right-hand triangle
into the diagram:
\[
\xymatrix{
&\Omega
G_k\ar@{->}[r]^-<<<<{\enlarge{\Omega\pi'}}\ar@{->}[d]_{\enlarge{\Omega\nu}}&\Omega
P^{2np^k+1}(p^{r+k})\ar@{->}[ddl]_{\enlarge{\Omega\overline{\phi}}}\\
\Omega G_k\ar@{=}[ru]\ar@{->}[d]_{\enlarge{\alpha}}&\Omega\Sigma\Omega
G_k\ar@{->}[l]_{\enlarge{\Omega\epsilon}}\ar@{->}[d]_{\enlarge{\Omega
\Sigma \alpha}}&\\
Z&\Omega \Sigma Z\ar@{->}[l]_{\enlarge{\mu}}&
}
\]
to see that $\alpha$ is homotopic to
$\mu(\Omega\overline{\phi})(\Omega\pi')
\sim\widehat{\phi}\Omega\pi'$ where $\widehat{\phi}$ is the composition
$\mu\Omega\overline{\phi}$. Restricting $\widehat{\phi}$
to $P^{2np^k}(p^{r+k})$ defines $\widetilde{\phi}$ whose adjoint is
\[
\phi\colon P^{2np^k-1}(p^{r+k})\to \Omega Z.
\]
To see that $\phi$ is divisible by $p^{r+k-1}$ , it suffices to
show that $\overline{\phi}$ is divisible by $p^{r+k-1}$. However since
$\alpha$ is
proper, the upper composition in the diagram:
\[
\xymatrix{
P^{2np^k-1}(p^{r+k-1})\vee
P^{2np^k}(p^{r+k-1})\ar@{->}[r]^-<<<<<<<<{\enlarge{a(k)\vee
c(k)}}\ar@{->}[dr]&\Omega
G_k\ar@{->}[r]^{\enlarge{\alpha}}\ar@{->}[d]_{\enlarge{\Omega
\pi'}}&\ar@{->}[d]Z\\
&\Omega P^{2np^k+1}(p^{r+k-1})\ar@{->}[r]^-<<<{\enlarge{\Omega\overline{\phi}}}&\Omega \Sigma Z
}
\]
is null homotopic. Consequently the lower composition
and its adjoint are null homotopic:
\[
\xymatrix@C=40pt{
P^{2np^k}(p^{r+k-1})\vee P^{2np^k+1}(p^{r+k-1})\ar@{->}[r]^-<<<<<<{\enlarge{-\delta_1\vee
p}}&P^{2np^k+1}(p^{r+k})\ar@{->}[r]^-<<<<<<{\enlarge{\overline{\phi}}}&\Sigma Z
}
\]
It follows from \ref{eq1.9} that $\overline{\phi}$ is divisible by
$p^{r+k-1}$.
Thus $\widetilde{\phi}$ and $\phi$ are divisible as well.

Finally we define
\[
\beta\colon G_k(Z)\to p_{k+1}(Z)
\]
as the composition
\[
\xymatrix{
P^{2np^{k+1}-1}(p^{r+k+1})\ar@{->}[r]^-<<<{\enlarge{\widetilde{\beta_{k+1}}}}&\Omega
E_k\ar@{->}[r]&\Omega G_k\ar@{->}[r]^-<<<{\enlarge{\alpha}}&Z
}
\]
Using the $H$-space structure on $Z$ we see that $\beta$ is a
homomorphism.\addtocounter{Theorem}{3}
\begin{Theorem}\label{theor7.20}
The composition $\beta(\alpha)$ is divisible by
$p^{r+k}$.
\end{Theorem}
This will follow from a sequence of lemmas
\begin{lemma}\label{lem7.17}
The homomorphism induced by the inclusion
\[
H_*(W_k)\to H_*(J_k)
\]
is a monomorphism.
\end{lemma}
\begin{proof}
By \ref{prop6.11}, $W_k$ has one cell in each dimension of the
form $2ni$ or $2ni-1$ for each $i\geqslant 2$; consequently
\[
H_j(W_k)=\begin{cases}
Z/p&\text{if}\ j=2ni\ \text{or}\ 2ni-1,\quad i\geqslant 2\\
0&\text{otherwise}.
\end{cases}
\]
By \ref{cor6.27}, the maps $x_i(k)\colon P^{2ni}\to J_k$ and $y_i(k)\colon
P^{2ni-1}\to J_k$
are nonzero in $\bmod\, p$ homology in dimensions $2ni$ and~$2ni-1$
respectively. By \ref{theor6.41}, the maps $x_i(k)$ and $y_i(k)$
factor through $W_k$ up to homotopy when $i\geqslant 2$. The
result follows.
\end{proof}

We will write $x_i\in H_{2ni}(J_k)$ and $y_i\in H_{2ni-1}(J_k)$
for the images of the generators in the homology of the
respective Moore spaces. Note that in the congruence
homotopy of~$J_k$ we have
\begin{align*}
\nu x_i(k)&\equiv x_{i+1}(k)\\
\mu x_i(k)&\equiv y_{i+1}(k)
\end{align*}
when $i\geqslant p^k$ by \ref{eq6.26}. The action of the principal
fibration defines an action
\[
H_*(\Omega D_k)\otimes H_*(J_k)\to H_*(J_k).
\]
Let $u,v$ be the Hurewicz images of $\nu$ and $\mu$ (see \ref{theor5.20}).
Then $Z/p[v]\otimes \Lambda(u)\subset H_*(\Omega D_k)$ acts on $H_*(J_k)$.
\begin{lemma}\label{lem7.18}
If $i\geqslant p^k$, $vx_i= x_{i+1}$ and $ux_i= y_{i+1}$.
\end{lemma}
\begin{proof} 
Apply \ref{prop5.16}.
\end{proof}
\begin{lemma}\label{lem7.19}
There are maps 
\[
r\colon P^{2np^{k+1}}(p^{r+k})\to W_k
\]
and
\[
d\colon P^{2np^{k+1}-1}(p^{r+k})\to W_k
\]
which are nonzero in $\bmod\, p$ homology
in dimensions $2np^{k+1}$ and $2np^{k+1}-1$ respectively and whose image in $J_k$ factors
through the map
\[
\xymatrix{
\Omega G_k*\Omega G_k\ar@{->}[r]^-<<<{\Gamma_k}&E_k\ar@{->}[r]^{\tau_k}&J_k.
}
\]
\end{lemma}
\begin{remark*} 
Both $\beta_{k+1}\rho$ and $\beta_{k+1}\delta_1$ satisfy the homological
condition of~\ref{lem7.19} and factor through $W_k$. However $\beta_{k+1}$ does not
factor through $\Gamma_k$ since it has order $p^{r+k+1}$, and there are
no elements in the homotopy of $\Omega G_k*\Omega G_k$ in that dimension
of that order.
\end{remark*}
\begin{proof} 
Let $x=f(p^k(p-1))\colon P^{2np^k(p-1)}(p^{r+k})\to J_k$
be the map constructed in \ref{prop7.8}. Apply \ref{lem7.9} with $i=k$
to construct
\[
r=f(p^{k+1})\colon P^{2np^{k+1}}(p^{r+k})\to J_k.
\]
The class $x_{p^{k+1}}\in H_{2np^{k+1}}(J_k)$ in the homology
image of $r$ since $r$ factors through $W_k$ and projects to
a nonzero class in $H_{2np^{k+1}}(F_k)$.

To construct $d$, we return to the class $x$ above
and observe that the homomorphism induced by $x$:
\[
H_*(P^{2np^k(p-1)}(p^{r+k}))\to H_*(J_k)
\]
contains both $x_{p^k(p-1)}$ and $y_{p^k(p-1)}$ in its image since
$x$ factors through $W_k$ which has $P^{2np^k(p-1)}(p^{r+k})$ as
a retract.

Consequently $x\beta\colon P^{2np^k(p-1)-1}(p^{r+k})\to J_k$ has
$y_{p^k(p-1)}$ in its
homology image. Now consider
\[
[\lambda_k,x\beta]_r\colon G_k\circ P^{2np^k(p-1)-1}(p^{r+k})\to J_k
\]
and apply \ref{eq7.10} and \ref{eq7.11}. It follows that
$v^{p^k}y_{p^k(p-1)}= y_{p^{k+1}}$ is in the homology image by
\ref{lem7.12}. We then
choose a map
\[
P^{2np^{k+1}-1}(p^{r+k})\to G_k\circ P^{2np^k(p-1)-1}
\]
which induces an isomorphism in dimension $2np^{k+1}-1$ to
construct $d$.
\end{proof}
\begin{proof}[Proof of \ref{theor7.20}]
 Since $W$ is a retract of $R$ by
\ref{prop7.8new} we can assume that
$r$ and $d$ factor through $R$. Choose units $u_1$ and $u_2$ so that the
maps
\begin{align*}
\Delta_1&=u_1\beta_{k+1}\rho-r\colon P^{2np^{k+1}}(p^{r+k})\to R_k\to W_k\\
\Delta_2&=u_2\beta_{k+1}\delta_1-d\colon P^{2np^{k+1}-1}(p^{r+k})\to R_k\to W_k
\end{align*}
are trivial in $\bmod\, p$ homology. This can be done since
the relevant factor of $W_k$ is $P^{2np^{k+1}}(p^{r+k+1})$ and both
$\beta_{k+1}\rho$ and $r$ are nontrivial in dimension $2np^{k+1}$ while
$\beta_{k+1}\delta_1$ and $d$ are nontrivial in dimension $2np^{k+1}-1$. It
follows from \ref{prop7.7new} that both $\Delta_1$ and $\Delta_2$ factor
through
\[
\bigvee\limits_{i=1}^k P^{2np^i}(p^{r+i-1})\vee P^{2np^i-1}(p^{r+i-1})\vee \Sigma
P.
\]
Since $\alpha$ is proper, we conclude that the
compositions
\begin{align*}
&\xymatrix{
P^{2np^{k+1}-1}(p^{r+k})\ar@{->}[r]^-<<<<{\widetilde{\Delta}_1}&\Omega
R_k\ar@{->}[r]&\Omega G_k\ar@{->}[r]^{\alpha}&Z
}\\
&\xymatrix{
P^{2np^{k+1}-2}(p^{r+k})\ar@{->}[r]^-<<<<{\widetilde{\Delta}_2}&\Omega
R_k\ar@{->}[r]&\Omega G_k\ar@{->}[r]^{\alpha}&Z
}
\end{align*}
are both null homotopic. Since the maprs $r$ and $d$ factor
through $\Omega G_k*\Omega G_k$, these terms are null homotopic
and we conclude that the compositions
\begin{align*}
&\xymatrix@C=30pt{
P^{2np^{k+1}-1}(p^{r+k})\ar@{->}[r]^-<<<<{\widetilde{\beta}_{k+1}\rho}&\Omega
R_k\ar@{->}[r]&\Omega G_k\ar@{->}[r]^{\alpha}&Z
}\\
&\xymatrix@C=30pt{
P^{2np^{k+1}-2}(p^{r+k})\ar@{->}[r]^-<<<<{\widetilde{\beta}_{k+1}\delta_1}&\Omega
R_k\ar@{->}[r]&\Omega G_k\ar@{->}[r]^{\alpha}&Z
}
\end{align*}
are null homotopic. Now consider the cofibration
sequence
\begin{multline*}
P^{2np^{k+1}-2}(p^{r+k})\vee \rho^{2np^{k+1}-1}(p^{r+k})\\
\xymatrix@C=35pt{\mbox{}\ar@{->}[r]^-<<<<{-\delta_1\vee
\rho}&P^{2np^{k+1}-1}(p^{r+k+1})\ar@{->}[r]^{p^{r+k}}&P^{2np^{k+1}-1}(p^{r+k+1}).
}
\end{multline*}
From this we see that the composition
\[
\xymatrix@C=28pt{
P^{2np^{k+1}-1}(p^{r+k+1})\ar@{->}[r]^-<<<{\widetilde{\beta}_{k+1}}&\Omega
R_k\ar@{->}[r]&\Omega G_k\ar@{->}[r]^{\alpha}&Z
}
\]
is divisible by $p^{r+k}$.
\end{proof}
\begin{proposition}\label{prop7.24}
$\beta r=0$
\end{proposition}
\begin{proof}
Since $\alpha$ is proper, the composition on the right in the diagram:
\[
\xymatrix@C=40pt{
P^{2np^k-1}(p^{r+k})\ar@{->}[r]^-<<<<<{\enlarge{\sigma\vee\sigma\beta}}\ar@{->}[d]^{\enlarge{\widetilde{\beta}_k}}&P^{2np^k-1}(p^{r+k-1})\vee
P^{2np^k}(p^{r+k-1})\ar@{->}[d]^{\enlarge{a(k)\vee c(k)}}&\\
\Omega G_{k-1}\ar@{->}[r]&\Omega G_k\ar@{->}[r]^{\enlarge{\alpha}}&Z
}
\]
is null homotopic. The diagram commutes up to
homotopy by~\ref{theor4.4} from which the result follows.
\end{proof}
\begin{proposition}\label{prop7.25}
If $\beta(\alpha)=0$, $\alpha\sim r\alpha'$ for some
proper $H$-map $\alpha'\colon \Omega G_k\to Z$.
\end{proposition}
\begin{proof}
By \ref{theor4.4},
$G_k=G_{k-1}\cup_{\alpha_k}CP^{2np^k}(p^{r+k})$. Since $G$ is a retract of
$\Sigma T$,
$G_{k-1}$ and $G_k$ are co-$H$ spaces and there is a homotopy commutative
diagram
\[
\xymatrix{
G_{k-1}\ar@{->}[r]\ar@{->}[d]&G_k\ar@{->}[d]\\
G_{k-1}\vee G_{k-1}\ar@{->}[r]&G_k\vee G_k
}
\]
The map $\alpha_k\vee\alpha_k\colon P^{2np^k}(p^{r+k})\vee
P^{2np^k}(p^{r+k})\to G_{k-1}\vee G_{k-1}$ factors
through the fiber of the lower horizontal map in the diagram and defines a
homotopy equivalence with the $2np^k$
skeleton of the fiber of that map. However the composition
\[
\xymatrix{
P^{2np^k}(p^{r+k})\ar@{->}[r]^-<<<{\alpha_k}&G_{k-1}\ar@{->}[r]&G_{k-1}\vee
G_{k-1}
}
\]
factors through this fiber, from which we get a
homotopy commutative square
\[
\xymatrix@C=42pt{
P^{2np^k}(p^{r+k})\ar@{->}[r]^{\alpha_k}\ar@{->}[d]&G_{k-1}\ar@{->}[d]\\
P^{2np^k}(p^{r+k})\vee
P^{2np^k}(p^{r+k})\ar@{->}[r]^-<<<<<<{\alpha_k\vee\alpha_k}& G_{k-1}\vee
G_{k-1};
}
\]
that is, $\alpha_k$ is a co-$H$ map. Consequently there is
also a homotopy commutative diagram:
\[
\xymatrix{
P^{2np^k}(p^{r+k})\ar@{->}[r]^-<<<<<<{\alpha_k}\ar@{->}[d]&G_{k-1}\ar@{->}[d]_{\nu}\\
\Sigma \Omega P^{2np^k}(p^{r+k})\ar@{->}[r]&\Sigma \Omega G_{k-1}
}
\]
However, the composition on the left and the bottom
is $\Sigma\widetilde{\alpha}_k$, where
\[
\widetilde{\alpha}_k\colon P^{2np^k-1}(p^{r+k})\rightarrow\Omega G_{k-1}
\]
is the adjoint of $\alpha_k$. Let $b=\pi_{k-1}\beta_{k}\colon
P^{2np^k}(p^{r+k})\to G_{k-1}$. By
\ref{theor4.4}, $p^{r+k-1}b$ is homotopic to $\alpha_k$. This leads to a
homotopy commutative diagram:
\[
\xymatrix@C=35pt{
P^{2np^k}(p^{r+k})\ar@{=}[r]\ar@{->}[d]_{\alpha_k}&P^{2np^k}(p^{r+k})\ar@{->}[r]^{p^{r+k-1}}\ar@{->}[d]_{\Sigma\widetilde{\alpha}_k}&P^{2np^k}(p^{r+k})\ar@{->}[d]_{\Sigma
\widetilde{b}}\\
G_{k-1}\ar@{->}[r]^{\nu}&\Sigma \Omega G_{k-1}\ar@{->}[r]&\Sigma \Omega
G_{k-1}
}
\]
Taking cofibers vertically, we get a composition
\[
\xymatrix@C=18pt{
G_k\ar@{->}[r]&\Sigma(\Omega
G_{k-1}\cup_{\widetilde{\alpha}_k}CP^{2np^k-1}(p^{r+k}))\ar@{->}[r]&\Sigma(\Omega
G_{k-1}\cup_{\widetilde{b}}CP^{2np^k-1}(p^{r+k})).
}
\]
By hypothesis, $\alpha$ extends to a map
\[
\xymatrix{
\Omega
G_{k-1}\cup_{\widetilde{b}}CP^{2np^k-1}(p^{r+k})\ar@{->}[r]^-<<<<{\overline{\alpha}}&Z.
}
\]
Composing these maps together defines a map $\alpha''$
\[
\xymatrix{
G_k\ar@{->}[r]&\Sigma(\Omega
G_{k-1}\cup_{\widetilde{b}}CP^{2np^k-1}(p^{r+k})\ar@{->}[r]^-<<<<{\Sigma\overline{\alpha}})&\Sigma
Z
}
\]
whose restriction to $G_{k-1}$ is the composition
\[
\xymatrix@C=48pt{
G_{k-1}\ar@{->}[r]^{\nu}&\Sigma\Omega
G_{k-1}\ar@{->}[r]^{\Sigma\alpha}&\Sigma Z.
}
\]
We now form the homotopy commutative diagram\setcounter{equation}{25}
\begin{equation}\label{eq7.16} 
\begin{split}
\xymatrix@C=48pt@R=48pt{
\Omega G_k\ar@{->}[r]^{\Omega \alpha''}&\Omega\Sigma Z\ar@{->}[r]^{\mu}&Z\\
\Omega G_{k-1}\ar@{->}[r]^->>>>>>>{\Omega \nu}\ar@{->}[u]&\Omega\Sigma \Omega
G_{k-1}\ar@{->}[r]^{\Omega
\epsilon}\ar@{->}[u]^{\Omega\Sigma\alpha}&\Omega
G_{k-1}\ar@{->}[u]^{\alpha} 
}
\end{split}
\end{equation}
where the lower composition is homotopic to the identity.
The upper composition is an $H$-map extending $\alpha$. We
will modify this slightly to satisfy our requirements.
In the diagram below, the left hand triangle is homotopy
commutative due to \ref{eq7.16}
and the left hand square follows
from~\ref{theor4.4}.
\[
\xymatrix@C=31pt@R=31pt{
P^{2np^k-1}(p^{r+k})
\ar@{->}[d]_{\widetilde{\beta}_k}
\ar@{->}[r]^-<<<<{\sigma\vee \sigma \beta}&
P^{2np^k-1}(p^{r+k-1})\!{}\vee{}\! P^{2np^k}(p^{r+k-1})
\ar@{->}[r]^-<<<<{-\delta_1\vee \rho}
\ar@{->}[d]_{(\Omega\pi_k)(\widetilde{\alpha(k)}\!{}\vee{}\!\widetilde{c(k)})}&
P^{2np^k}(p^{r+k})
\ar@{-->}[dddl]^{\epsilon}\\
\Omega G_{k-1}
\ar@{->}[rdd]_{\makebox[46pt]{$\alpha$}}
\ar@{->}[r]&
\Omega G_k
\ar@{->}[d]_{\Omega\alpha''}&\\
&\Omega\Sigma Z
\ar@{->}[d]_{\mu}&\\
&Z&\\
}
\]
By hypothesis $\beta(\alpha)=\alpha\widetilde{\beta}_k$ is null homotopic.
Since the upper horizontal sequence is a cofibration
sequence, there is an extension 
\[
\epsilon \colon P^{2np^k}(p^{r+k})\to Z.
\]
Extend $\epsilon$ to an $H$-map $\epsilon'\colon \Omega
P^{2np^k+1}(p^{r+k})\to Z$ and
define 
\[
\alpha'=\mu\Omega\alpha''-\epsilon '\Omega\pi',
\]
where
$\pi'$ is the projection of $G_k$ onto $P^{2np^k+1}(p^{r+k})$.
Since $Z$ is homotopy-Abelian and $\alpha'$ is the difference between
two $H$-maps, $\alpha'$ is an $H$-map. Since
the restriction of $\epsilon '\Omega\pi'$ to $\Omega G_{k-1}$ is null
homotopic, $\alpha'$ extends $\alpha$ by~\ref{eq7.16}. From \ref{theor4.4} we
construct a homotopy commutative square:
\[
\xymatrix@C=56pt{
P^{2np^k}(p^{r+k})\ar@{->}[r]^{i}&\Omega P^{2np^k+1}(p^{r+k})\\
P^{2np^k-1}(p^{r+k-1})\vee
P^{2np^k}(p^{r+k-1})\ar@{->}[r]^-<<<<<<<<<<<<{(\Omega\pi_k)(\widetilde{a(k)}\vee
\widetilde{c(k)})}\ar@{->}[u]^{-\delta_1\vee \rho}&\Omega G_k\ar@{->}[u]^{\Omega \pi'} 
}
\]
Consequently,
\begin{align*}
\makebox[46pt]{}\alpha'(\Omega \pi_k)(\widetilde{a(k)}\vee
\widetilde{c(k)}&\sim(\mu\Omega\alpha''-\epsilon'\Omega\pi')(\Omega\pi_k)(\widetilde{a(k)}\vee
\widetilde{c(k)})\\
&\sim\epsilon (\delta_1\vee \rho)-\epsilon'i(-\delta_1\vee \rho)\sim
*\makebox[56pt]{}\qed
\end{align*}\noqed
\end{proof}
This completes the proof of~\ref{theor7.3}\qed

\section{Applications}\label{subsec7.4}

In this section we will discuss various applications of
the results developed in the previous sections.\addtocounter{Theorem}{1}
\begin{proposition}\label{prop7.27}
Suppose $p^{r+1}\pi_*(Z)=0$. Then there is a
natural exact sequence:
\begin{multline*}
\xymatrix{
0\ar@{->}[r]&p^r\left[P^{2np}(p^{r+1}),Z\right]\ar@{->}[r]^-<<<<{\enlarge{e}}&[T,Z]_H
\ar@{->}[r]^->>>>{\enlarge{r}}&\left[P^{2n}(p^r),Z\right] 
}\\
\xymatrix{
\ar@{->}[r]^->>>>{\enlarge{\beta}}&p^r\left[P^{2np-1}(p^{r+2}),Z\right].
}
\end{multline*}
In particular, if $p^r\pi_i(Z)=0$ for $2np-2\leqslant i\leqslant 2np$,
$r$~is an
isomorphism.
\end{proposition}
\begin{note*}
We are not asserting that $\beta$ is onto. In fact, we have
no examples where $\beta\neq 0$.
\end{note*}
\begin{proof}
This is immediate from~\ref{theor7.2} and~\ref{theor7.3} since
$p_k(Z)=0$
and $p_k(\Omega Z)\linebreak[4]= 0$ for all $k\geqslant 2$.
\end{proof}
\begin{corollary}\label{cor7.28}
Suppose $T$ and $T'$ are two homotopy
Abelian Anick spaces for the same values of~$n$, $r$ and $p>3$. Then there
is a homotopy equivalence via an $H$-map
and the $H$-space expondent is~$p^r$.
\end{corollary}
\begin{proof}
We apply \ref{prop7.27} with $Z=T'$. It suffices to show
that $p^r\pi_i(T')=0$ for
$2np-2\leqslant i\leqslant 2np$. We apply the fibration sequence:
\[
\xymatrix{
W_n\ar@{->}[r]&T'_{2n-1}\ar@{->}[r]^->>>>{\enlarge{E}}&\Omega
T_{2n}(p^r)\ar@{->}[r]^-<<<{\enlarge{H}}&BW_n
}
\]
from \cite{GT10}. According to \cite{Nei83}, $p^r\pi_*(T_{2n})=0$ and
according to \cite{CMN79b} $p\pi_*(W_n)=0$, so $p^{r+1}\pi_*(T')=0$.
However, since $p\geqslant 3$, $\pi_i(W_n)=0$ when $2np-2\leqslant
i\leqslant 2np$, so
$p^r\pi_i(T')=0$ in this range. Thus
\[
\left[T,T'\right]_H\simeq\left[P^{2n}(p^r),T'\right]=Z/p^r
\]
and the result follows. In particular, $p^r\pi_i(T)=0$ for
all $i$ as a consequence.
\end{proof}
\begin{corollary}\label{cor7.29}
Suppose $\alpha\colon P^{2n}(p^r)\to P^{2n}(p^s)$ with $s\leqslant r$
then there is a unique $H$-map $\widehat{\alpha}$ such
that the diagram:
\[
\xymatrix{
T_{2n-1}(p^r)\ar@{->}[r]^{\enlarge{\widehat{\alpha}}}&T_{2m-1}(p^s)\\
P^{2n}(p^r)\ar@{->}[r]^{\enlarge{\alpha}}\ar@{->}[u]^{\enlarge{i}}&P^{2m}(p^s)
\ar@{->}[u]^{\enlarge{i}}}
\]
homotopy commutes.
\end{corollary}
\begin{note*}
In case $r=s$, this result was the original motivation
for these conjectures, leading to a secondary composition
theory~\cite{Gra93a}.
\end{note*}
\begin{proof}
This follows from \ref{prop7.27} and \ref{cor7.28}.
\end{proof} 
\begin{proposition}\label{prop7.30}
There is an $H$-map $\theta_1\colon T_{2n-1}(p^r)\to T_{2np-1}(p^{r+1})$
which induces a homomorphism of degree $p^r$ in~$H^{2np}$.
Furthermore, the map~$e$ in~\ref{prop7.27}, evaluated on on $p^rf$ is
the consideration: 
\[
\xymatrix{
T_{2n-1}(p^r)\ar@{->}[r]^->>>>{\enlarge{\theta_1}}&T_{2np-1}(p^{r+1})\ar@{->}[r]^-<<<<{\enlarge{\widehat{f}}}&Z
}
\]
where $\widehat{f}$ is the unique extension of~$f$ to an $H$-map.
\end{proposition}
\begin{proof}
By~\ref{cor7.28}, $p^{r+1}\pi_*(T_{2np-1}(p^{r+1}))=0$, so we may apply
\ref{prop7.27} with $Z=T_{2np-1}(p^{r+1})$ and use naturality under
$\widehat{f}$.
This leads to a commutative square
\[
\xymatrix{
0\ar@{->}[r]&p^r\left[P^{2np}(p^{r+1}),Z\right]\ar@{->}[r]^{\enlarge{e}}&[T,Z]_H\\
0\ar@{->}[r]
&p^r\left[P^{2np}(p^{r+1}),T_{2np-1}(p^{r+1})\right]\ar@{->}[r]^-<<<<{\enlarge{e}}
\ar@{->}[u]_{\enlarge{\widehat{f}_*}}
\ar@{}[d]^{\rotatebox{90}{$\approx$}}
&
\left[T,T_{2np-1}(p^{r+1})\right]_H
\ar@{->}[u]^{\enlarge{\widehat{f}_*}}\\
&Z/p&
}
\]
The image of the generator in the lower left hand corner under~$e$
is an $H$-map $\theta_1$ which is nonzero and of order~$p$. To
evalueate $(\theta_1)^*$ in cohomology, use the diagram:
\[
\xymatrix{
T_{2n-1}(p^r)\ar@{->}[r]^->>>>>{\enlarge{\theta_1}}&T_{2np-1}(p^{r+1})\\
\Omega
G_1\ar@{->}[r]^->>>>>>{\enlarge{p^r(\Omega\pi')}}\ar@{->}[u]^{\enlarge{h_1}}&\Omega
P^{2np+1}(p^{r+1})\ar@{->}[u] 
}
\]
based on the definition of~$e$.
\end{proof}

Note that a similar construction can be made in case
$p^{r+k}\pi_*(Z)=0$. In this case
\[
[T,Z]_H\simeq G_k(Z)
\]
and 
\[
e\colon
\xymatrix{p^{r+k-1}\big[P^{2np^k}(p^{r+k}),Z\big]\ar@{->}[r]&G_k(Z)}
\]can be evaluated
on $p^{r+k-1}f$ as a composition:
\[
\xymatrix{
T_{2n-1}\ar@{->}[r]^->>>>{\enlarge{\theta_k}}&T_{2np^k-1}\ar@{->}[r]^-<<<{\enlarge{\widehat{f}}}&Z
}
\]
where $\theta_k$\index{thetak@$\theta_k$|LB} is an $H$-map of order $p$ inducing $p^{r+k-1}$ in
$H^{2np^k}$.

In section \ref{subsec1.5}, certain coefficient maps were
labeled for use:
\begin{align*}
\beta&\colon P^m(p^s)\to P^{m+1}(p^s)\\
\rho&\colon P^m(p^s)\to P^m(p^{s+1})\\
\sigma&\colon P^m(p^s)\to P^m(p^{s-1})
\end{align*}
Analogs of these maps were implicitly defined and used
in section~\ref{subsec5.2}:
\begin{align*}
\rho&\colon T_{2n}(p^s)\to T_{2n}(p^{s+1})\\
\sigma&\colon T_{2n}(p^s)\to T_{2n}(p^{s-1})
\end{align*}
and one can easily define $\beta$ as the
compositions:
\begin{align*}
T_{2n}(p^r)&\to S^{2n+1}\to T_{2n+1}(p^r)\\
T_{2n-1}(p^r)&\to \Omega S^{2n+1}\to T_{2n}(p^r)
\end{align*}
using \ref{cor7.29}, we can define
\[
\sigma\colon T_{2n-1}(p^r)\to T_{2n-1}(p^{r-1})
\]

We apply \ref{prop7.27} to construct $\rho$, but it is not unique in general
\[
\rho\colon T_{2n-1}(p^r)\to T_{2n-1}(p^{r+1})
\]
\begin{proposition}\label{prop7.31}
There is a split short exact sequence:
\[
\xymatrix@C=20pt{
0\ar@{->}[r]&p_1(\Omega
T_{2n-1}(p^{r+1})\ar@{->}[r]&\left[T_{2n-1}(p^r),T_{2n-1}(p^{r+1})\right]_H\ar@{->}[r]^-<<<{\enlarge{e}}&Z/p^r\ar@{->}[r]&0
}
\]
and 
\[
p_1(\Omega T_{2n-1}(p^{r+1}))\cong
p^r\left\{\left[P^{2np+1}(p^{r+1}),S^{2n+1}\right]\oplus\left[P^{2np+2}(p^{r+1}),S^{2n+1}\right]\right\}
\]
\end{proposition}
\begin{note*}
$p_1(\Omega T_{2n-1}(p^{r+1}))$ is known to be nonzero when $p^r$ divides
$n$
(\cite{Gra69}) and is known to be zero when $r\geqslant n$ (\cite{CMN79b}).
\end{note*}
\begin{proof}
In order to establish this exact sequence we
show that the map $\beta$ in fact is zero in this case.
Since $\beta_1$ factors through $E_0$, by \ref{theor4.4}, the
composition:
\[
\xymatrix{
P^{2np}(p^{r+1})\ar@{->}[r]^{\enlarge{\beta_1}}&P^{2n+1}(p)\ar@{->}[r]^{\enlarge{\varphi_0}}&S^{2n+1}\{p^r\}
}
\]
is null homotopic. Let $j$ be composition
\[
\xymatrix{
P^{2n+1}(p^r)\ar@{->}[r]^{\enlarge{\varphi_0}}&S^{2n+1}\{p^r\}\ar@{->}[r]^->>>>{\enlarge{\sigma}}&S^{2n+1}\{p^{r+1}\}=T_{2n}(p^{r+1});
}
\]
then $\Omega j$ is homotopic to the composition:
\[
\xymatrix{
\Omega P^{2n+1}(p^r)\ar@{->}[r]^->>>>{\enlarge{\Omega \rho}}&\Omega
P^{2n+1}(p^{r+1})\ar@{->}[r]^{\enlarge{h_0}}&T_{2n-1}(p^{r+1})\ar@{->}[r]^{\enlarge{E}}&\Omega T_{2n}(p^{r*})
}
\]
since both compositions are $H$-maps which agree on $P^{2n}(p^r)$.
In the diagram below, the upper composition is null
homotopic and the lower sequence is a fibration sequence:
\[
\xymatrix{
P^{2np-1}(p^{r+1})\ar@{->}[r]^->>>>{\enlarge{\widetilde{\beta_1}}}\ar@{-->}[d]_{\enlarge{\xi}}&\Omega
P^{2n+1}(p^r)\ar@{->}[r]^{\enlarge{\Omega j}}\ar@{->}[d]^{\enlarge{h_0\Omega
\rho}}&\Omega T_{2n}(p^r)\ar@{=}[d]\\
W_n\ar@{->}[r]&T_{2n-1}(p^{r+1})\ar@{->}[r]^{\enlarge{E}}&\Omega T_{2n}(p^{r+1}).
}
\]
It follows that the map $\xi$ exists forming a
homotopy commutative square. But 
\[
\left[P^{2np-1}(p^{r+1}),W_n\right]=*,
\]
so the composition:
\[
\xymatrix{
P^{2np-1}(p^{r_*})\ar@{->}[r]^{\enlarge{\widetilde{\beta_1}}}&\Omega
P^{2n+1}(p^r)\ar@{->}[r]^->>>>{\enlarge{\Omega \rho}}&\Omega
P^{2n+1}(p^{r+1})\ar@{->}[r]^{\enlarge{h_0}}&T_{2n-1}(p^{r_*})
}
\]
is null homotopic. However $h_0\Omega \rho$ generates
\[
\left[\Omega
P^{2n+1}(p^r),T_{2n-1}(p^{r+1})\right]_H\cong\left[P^{2n}(p^r),T_{2n-1}(p^{r+1})\right]\cong
Z/p^r
\]
and consequently $\beta =0$. Finally
\[
p_1\left(\Omega
T_{2n-1}(p^{r+1})\right)=p^r\left[P^{2np}(p^{r+1}),T_{2n-1}P^{r+1})\right].
\]
But since
$
[P^{2np}(p^{r+1}),W_n]=0=[P^{2np}(p^{r+1}),BW_n]$,
\begin{align*}
\left[P^{2np}(p^{r+1}),T_{2n-1}(p^{r+1})\right]&=\left[P^{2np}(p^{r+1}),\Omega
T_{2n}(p^{r+1})\right]\\
&=\left[P^{2np+1}(p^{r+1}),S^{2n+1}\{p^{r+1}\}\right]\\
&=\left[P^{2np+1}(p^{r+1}),S^{2n+1}\right]\oplus\left[P^{2np+2}(p^{r+1}),S^{2n+1}\right],
\end{align*}
so
\[
p_1\left(\Omega
T_{2n-1}(p^{r+1})\right)=p^r\left\{\left[P^{2np+1}(p^{r+1}),S^{2n+1}\right]\oplus\left[P^{2np+2}(p^{r+1}),S^{2n+1}\right]\right\}.
\]
These groups are stable and trivial if $r\geqslant n$. However, if $p^r$
divides $n$, there is an element of $\pi_{2np}(S^{2n+1})$ of order $p^{r+1}$
and consequently $p_1(\Omega T_{2n-1}(p^{r+1}))\allowbreak\neq\nobreak 0$ in this case.
The exact sequence is split since $T_{2n-1}(p^r)$ has
exponent $p^r$.
\end{proof}

Finally we note that for every choice of~$\rho$, $\rho\sigma=\sigma\rho=p$,
$\beta=\sigma\beta\rho$ and $\beta\sigma^t=p^t\sigma^t\beta$, as in~section~\ref{subsec1.5}.
\begin{proposition}\label{prop7.32}
There is a unique $H$-map
\[
\xymatrix{
T_{2n-1}\ar@{->}[r]^->>>>{\enlarge{f}}&\Omega^2P^{2n+2}(p^r)
}
\]
up to a unit whose double adjoint has a right homotopy inverse.
\end{proposition}
\begin{proof}
By \cite{CMN79a}, $p^{r+1}\pi_*\left(\Omega^2P^{2n+2}(p^r)\right)=0$, so we
apply \ref{prop7.27}. We have
\begin{align*}
p^r\left[P^{2np}(p^{r+1}),\Omega^2P^{2n+1}(p^r)\right]&=p^r\left[P^{2np+2}(p^{r+1}),P^{2n+2}(p^r)\right]\\
p^r\left[P^{2np-1}(p^{r+1}),\Omega^2P^{2n+1}(p^r)\right]&=p^r\left[P^{2np+1}(p^{r+1}),P^{2n+2}(p^r)\right]
\end{align*}
However $p^r\pi_i\left(P^{2n+2}(p^r)\right)=0$ for $i<(4n+2)p-1$.
\end{proof}
Since $P^{2n+2}(p^r)$ is a retract of $\Sigma^2T_{2n-1}$, the double
adjoint has a
right homotopy inverse.
\begin{corollary}\label{cor7.33}
If $Z$ is $H$-equivalent to the loop space
on an $H$-space, every map
$\xymatrix{P^{2n}(p^r)\ar@{->}[r]^->>>{\enlarge{\alpha}}&Z}$ has an
extension to an $H$-map
$\xymatrix{T_{2n-1}\ar@{->}[r]^{\enlarge{\widehat{\alpha}}}&Z}$.
\end{corollary}
\begin{proof}
If $Z=\Omega W$ the adjoint of $\alpha$ extends
\[
\xymatrix{
\Omega P^{2n+2}(p^r)\ar@{->}[dr]^{\enlarge{\alpha'}}&\\
P^{2n+1}(p^r)\ar@{->}[r]^-<<{\enlarge{\widetilde{\alpha}}}\ar@{->}[u]&W
}
\]
and we construct $\widetilde{\alpha}$ as the composition:
\[
\xymatrix{
T_{2n-1}(p^r)\ar@{->}[r]&\Omega^2P^{2n+2}(p^r)\ar@{->}[r]^-<<<{\enlarge{\Omega \alpha'}}&Z
}
\]
using \ref{prop7.32}.
\end{proof}

\appendix

\chapter{The Case $n=1$ and the Case $p=3$}

In section~\ref{subsec4.2}, we applied index $p$ approximation to
reduce the obstructions to a homotopy-Abelian $H$-space structure to a
family of elements in the homotopy of~$E_k$ with $\bmod\, p^s$
coefficients. This reduction only works when $n>1$. We will
use a different method in this case. However, the material
in sections~\ref{subsec5.1} and~\ref{subsec6.2} on $D_k$, $J_k$ and $F_k$
does not
depend on~\ref{subsec4.2}, and we can still construct $\gamma_k\colon J_k\to
BW_n$
(see for example \cite{Gra08}).
\begin{Theorem}
For $p>2$, $r\geqslant 1$ and $n=1$, the Anick space is
homotopy equivalent to a double loop space and hence
has a homotopy-Abelian $H$-space structure. 
\end{Theorem}
\begin{proof}
Let $e\in H^4(BS^3;Z_{(p)})$ be a generator and $\kappa=p^re$.
Let $X$ be the homotopy fiber of~$\kappa$. Then we have a
homotopy commutative diagram of fibration sequences
\[
\xymatrix{
\Omega X\ar@{->}[r]&S^3\ar@{->}[r]^->>>>{\Omega \kappa}&K(Z;3)\\
S^3\{p^r\}\ar@{->}[r]\ar@{->}[u]^{\gamma}&S^3\ar@{->}[r]^{p^r}\ar@{=}[u]&S^3\ar@{->}[u]_{\Omega e}
}
\]
with $\gamma$ uniquely determined. Using $\gamma$, we construct a
diagram of vertical fibration sequences:
\[
\xymatrix{
S^1\ar@{->}[r]\ar@{->}[d]&\Omega^2S^3\ar@{->}[r]\ar@{->}[d]&S^1\ar@{->}[d]\\
T_1\ar@{->}[r]\ar@{->}[d]&\Omega S^3\{p^r\}\ar@{->}[d]\ar@{->}[r]^{\Omega
\gamma}&\Omega^2X\ar@{->}[d]\\
\Omega S^3\ar@{=}[r]&\Omega S^3\ar@{=}[r]&\Omega S^3
}
\]
Since the upper horizontal composition is a homotopy equivalence,
$T_1\simeq \Omega^2X$. Note that the right hand fibration is an
$H$-fibration and is an Anick fibration. 
\end{proof}
\begin{Theorem}
If $T_{2n-1}(3^r)$ is homotopy associative, $n=3^k$ with $k\geqslant 0$.
Furthermore if
$n>1$, then $r=1$.
\end{Theorem}
\begin{proof}
For any homotopy associative space~$T$, there is a map:
\[
T*T*T\to ST \cup_{H(\mu)}C(T*T)
\]
building the third stage of the classifying space
construction (\cite{Sug57},\linebreak[4]
\cite{Sta63}). The mapping
cone $X$ of this map has the cohomology of the bar construction
on the homology of~$T$ through dimension $8n-1$. In particular
the $\text{mod}\, p$ cohomology of the $6n$ skeleton of~$X$ has as a basis,
classes $u$, $v$, $u^2$, $uv$, $u^3$
where $|u|=2n$ and $|v|=2n+1$. The $6n$ skeleton of the subspace
$ST\cup_{H(\mu)}CT*T$ has cohomology generated by $u$, $v$, $u^2$, $uv$.
Now since the map $R\to G$ is a retract of the map
$H(\mu)\colon T*T\to \Sigma T$, the $4n+1$ skeleton of
$ST\cup_{H(\mu)}CT*T$
contains the $4n+1$ skeleton of $G\cup CR$ as a retract (See the proof of~\ref{cor6.15}).
But $[G\cup CR]^{4n}=P^{2n+1}\cup_{x_2}CP^{4n}$; consequently
\[
X^{6n}\simeq P^{2n+1}\cup_{x_2}CP^{4n}\cup e^{6n}.
\]
Note that $\mathcal{P}^n u=u^3$ generates $H^6(X;Z/3)$. Since
$\Sigma x_2$ is inessential we can
pinch the middle cells to a point after one
suspension, and obtain a space with cell structure
\[
S^{2n+1}\cup_{p^r}e^{2n+2}\cup e^{6n+1}
\]
with $\mathcal{P}^n\neq 0$. However, $\mathcal{P}^n$ is decomposable unless
$n=p^k=3^k$.
Furthermore, the decomposition of $\mathcal{P}^{p^k}$ by secondary
operations (\cite{Liu62}) implies that if $n>1$, we must have $r=1$.
\end{proof}

Note that such a space for $n>1$ would imply that the
``$\text{mod}\, 3$ Arf invariant class'' survives the Adams
spectral sequence. This does happen when $n=p$ with $T_5(3)=\Omega
S^3\langle 3\rangle$, but not
when $n=p^2$.
\backmatter

\providecommand{\bysame}{\leavevmode\hbox to3em{\hrulefill}\thinspace}
\providecommand{\MR}{\relax\ifhmode\unskip\space\fi MR }
\providecommand{\MRhref}[2]{%
  \href{http://www.ams.org/mathscinet-getitem?mr=#1}{#2}
}
\providecommand{\href}[2]{#2}

\markright{}

\begin{theindex}

  \item $A$\dotfill \LB{95}
  \item $A_*(\ )$\dotfill \LB{60}
  \item $BW_n$\dotfill \LB{3}
  \item $C$\dotfill \LB{95}
  \item $C_k$\dotfill \LB{51}
  \item $D_k$\dotfill \LB{51}
  \item $E$\dotfill \LB{5}, \LB{16}
  \item $E_0(m)$\dotfill \LB{67}
  \item $E_k$\dotfill \LB{5}, \LB{15}
  \item $F_0(m)$\dotfill \LB{67}
  \item $F_k$\dotfill \LB{52}
  \item $G$\dotfill \LB{5}
  \item $G(j,k)$\dotfill \LB{48}
  \item $G\circ H$\dotfill \LB{6}, \LB{19}
  \item $G^{[i]}$\dotfill \LB{30}
  \item $G^{[i]}H^{[j]}$\dotfill \LB{30}
  \item $G_k(Z)$\dotfill \LB{1}
  \item $J_k$\dotfill \LB{6}, \LB{52}
  \item $L_k$\dotfill \LB{42}
  \item $M_*(\ )$\dotfill \LB{60}
  \item $M_m$\dotfill \LB{59}
  \item $P^m$\dotfill \LB{8}
  \item $P^{2n}(p^r)$\dotfill \LB{1}
  \item $P_k$\dotfill \LB{55}
  \item $R$\dotfill \LB{5}
  \item $R_k$\dotfill \LB{71}, \LB{92}
  \item $T$\dotfill \LB{4}
  \item $T_{2n-1}$\dotfill \LB{1}
  \item $T_{2n}$\dotfill \LB{2}
  \item $U_k$\dotfill \LB{54}
  \item $W$\dotfill \LB{20}
  \item $W_k$\dotfill \LB{71}, \LB{92}
  \item $a$\dotfill \LB{10}
  \item $\overline{a(i)}$\dotfill \LB{62}
  \item $a(k)$\dotfill \LB{37}
  \item $ad^i$\dotfill \LB{30}
  \item $ad_r^i$\dotfill \LB{31}
  \item $\overline{b(i)}$\dotfill \LB{62}
  \item $c$\dotfill \LB{51}
  \item $c(k)$\dotfill \LB{40}
  \item $e$\dotfill \LB{34}
  \item $e[\ ,\ ]$\dotfill \LB{56}
  \item $e\pi_k$\dotfill \LB{56}
  \item $e_k$\dotfill \LB{14}, \LB{70}
  \item $f$\dotfill \LB{16}
  \item $f(m)$\dotfill \LB{98}
  \item $g$\dotfill \LB{16}
  \item $g_k$\dotfill \LB{39}
  \item $h$\dotfill \LB{7}, \LB{16}
  \item $p_k(Z)$\dotfill \LB{1}
  \item $u$\dotfill \LB{16}
  \item $v$\dotfill \LB{16}
  \item $v_i$\dotfill \LB{16}
  \item $w$\dotfill \LB{21}
  \item $x_i(k)$\dotfill \LB{82}
  \item $x_j$\dotfill \LB{63}
  \item $y_i(k)$\dotfill \LB{82}
  \item $y_j$\dotfill \LBmorespace{63}
  \item $[\ ,\ ]_H$\dotfill \LB{1}
  \item \{\ ,\ \}\dotfill \LB{22}
  \item \{\ ,\ \}$_r$\dotfill \LB{22}
  \item \{\ ,\ \}$_{\times}$\dotfill \LB{22}
  \item $\equiv$\dotfill \LB{56}
  \item $\nabla$\dotfill \LB{12}

  \indexspace

  \item $\alpha_k$\dotfill \LB{5}
  \item $\widetilde{\alpha_k}$\dotfill \LB{74}

  \indexspace

  \item $\beta$\dotfill \LB{8}
  \item $\beta_k$\dotfill \LB{6}, \LB{41}

  \indexspace

  \item $\Delta$\dotfill \LB{33}
  \item $\delta_t$\dotfill \LB{8}

  \indexspace

  \item $\epsilon$\dotfill \LB{19}
  \item $\eta_k$\dotfill \LB{52}

  \indexspace

  \item $\Gamma$\dotfill \LB{12}
  \item $\Gamma_{k}$\dotfill \LB{6}, \LB{70}
  \item $\gamma_{k}$\dotfill \LB{6}, \LB{70}
  \item $\Gamma^{\prime}$\dotfill \LB{24}

  \indexspace

  \item $\mu$\dotfill \LB{62}

  \indexspace

  \item $\nu$\dotfill \LB{3}, \LB{62}
  \item $\nu_{\infty}$\dotfill \LB{5}, \LB{15}
  \item $\nu_k$\dotfill \LB{5}, \LB{15}
  \item $\nu_p(m)$\dotfill \LB{42}

  \indexspace

  \item $\omega$\dotfill \LB{11}

  \indexspace

  \item $\pi'$\dotfill \LB{42}

  \indexspace

  \item $\rho$\dotfill \LB{8}

  \indexspace

  \item $\sigma$\dotfill \LB{8}
  \item $\sigma_k$\dotfill \LB{52}

  \indexspace

  \item $\tau_k$\dotfill \LB{6}, \LB{52}, \LB{70}
  \item $\theta_k$\dotfill \LB{113}

  \indexspace

  \item $\varphi$\dotfill \LB{4}
  \item $\varphi_k$\dotfill \LB{14}
  \item $\varphi'_k$\dotfill \LB{52}

  \indexspace

  \item $\widehat{\xi}$\dotfill \LB{12}
  \item $\xi_k$\dotfill \LB{52}

  \indexspace

  \item $\zeta$\dotfill \LB{25}
  \item $\zeta'$\dotfill \LB{29}

\end{theindex}
\end{document}